\documentclass[11pt,a4paper]{amsart} 
\usepackage{amssymb,amscd,amsmath, young,mathrsfs}
\usepackage{mathrsfs, mathdots} 
\usepackage{float}
\usepackage[usenames]{color}
\usepackage{soul}

\theoremstyle{plain}
\newtheorem{thm}{Theorem}[section]
\newtheorem{lem}[thm]{Lemma}
\newtheorem{pro}[thm]{Proposition}
\newtheorem{cor}[thm]{Corollary}

\newtheorem{sublemma}[thm]{Sublemma}

\newtheorem{lemma}[thm]{Lemma}
\newtheorem{con}[thm]{Conjecture}

\theoremstyle{remark}
\newtheorem{rem}[thm]{Remark}

\newtheorem{exm}[thm]{Example}
\newtheorem{asm}[thm]{Assumption}

\newtheorem{dfn}[thm]{Definition}
\newtheorem*{question*}{Question}

\newcommand{\N}{\mathbb{N}}
\newcommand{\Z}{\mathbb{Z}}
\newcommand{\Q}{\mathbb{Q}}

\newcommand{\C}{\mathbb{C}}

\newcommand{\tensor}{\otimes}

\newcommand{\p}{\mathfrak{p}}

\renewcommand{\epsilon}{\varepsilon}

\newcommand{\npr}{\ell}

\newcommand{\dimL}{n}
\newcommand{\overdim}{\overline{n}}
\newcommand{\unifmax}{r}

\DeclareMathOperator{\rk}{rk}

\DeclareMathOperator{\GL}{GL}

\DeclareMathOperator{\Aut}{Aut}

\DeclareMathOperator{\Hom}{Hom}

\def \bfn {{\bf n}}

\def \aAut {\mathfrak{Aut} \,}
\def \aEnd {\mathfrak{End} \,}

\def \p {\ensuremath{\mathfrak{p}}}

\author{Mark N.~Berman} 
\address{Department of Mathematics, Ort Braude College, P. O. Box 78, Snunit St. 51, Karmiel 2161002, Israel} \email{berman@braude.ac.il}

\author{Itay Glazer}
\address{Department of Mathematics, Northwestern University, 2033 Sheridan Road, Evanston, IL 60208, USA} \email{itayglazer@gmail.com}

\author{Michael M.~Schein} \address{Department of Mathematics,
  Bar-Ilan University, Ramat Gan 5290002,
  Israel}\email{mschein@math.biu.ac.il}


\begin{document}
 \title[Pro-isomorphic zeta functions under base extension]{Pro-isomorphic zeta functions of nilpotent groups and Lie rings under base extension} 

 \maketitle
\begin{abstract}
We consider pro-isomorphic zeta functions of the groups $\Gamma(\mathcal{O}_K)$, where $\Gamma$ is a unipotent group scheme defined over $\Z$ and $K$ varies over all number fields.  Under certain conditions, we show that these functions have a fine Euler decomposition with factors indexed by primes $\p$ of $K$ and depending only on the structure of $\Gamma$, the degree $[K:\Q]$, and the cardinality of the residue field $\mathcal{O}_K / \p$.  We show that the factors satisfy a certain uniform rationality and study their dependence on $[K:\Q]$.  Explicit computations are given for several families of unipotent groups.
\end{abstract}

\section{Introduction}

The field of subgroup growth studies connections between the structural features of a finitely generated group and its lattice of subgroups of finite index.  For instance, a celebrated theorem~\cite{LubotzkyMann/91, LubotzkyMannSegal/93} characterizes, in terms of their structural properties, the finitely generated groups having `polynomial subgroup growth,' namely those for which the number of subgroups of index $n$ grows polynomially in $n$.

We are interested in the confluence of two branches of this broad and rich topic. On the one hand, we study pro-isomorphic zeta functions of groups - these are analytic functions which keep track of finite index subgroups of a given group in a certain specialized setting. On the other hand, we consider base extension - a process whereby we enlarge the group by means of a field extension in a controlled way - and then ask what effect this has on the subgroup growth. Each of these aspects of subgroup growth is of independent interest, and their intersection is largely uncharted territory.

Let $G$ be a finitely generated group, and let $\mathcal{S}$ be a collection of subgroups of $G$.  Following~\cite{GSS/88}, for each $n \in \N$ define 
\[a^{\mathcal{S}}_n=\#\{H\in \mathcal{S}\mid [G:H]=n\},\]
which is finite since $G$ is finitely generated. 
Define the $\mathcal{S}$-zeta function of $G$ as
\begin{align} \label{equ:zeta.dfn}
\zeta_G^\mathcal{S}(s)=\sum_{n=1}^\infty a^\mathcal{S}_n n^{-s} = \sum_{H \in \mathcal{S} \atop [G:H] < \infty} [G:H]^{-s},
\end{align}
where $s$ is a complex variable.  Similarly, for each prime $p$ define the local $\mathcal{S}$-zeta function of $G$ at $p$ to be the analogous sum running over subgroups of $p$-power index:
\begin{align} \label{equ:zeta.local}
\zeta_{G,p}^\mathcal{S}(s)=\sum_{k=0}^\infty a^\mathcal{S}_{p^k} p^{-ks}.
\end{align}
Natural examples occur when $\mathcal{S}$ is the set of all subgroups or of all normal subgroups of $G$.  This paper studies another interesting case, where $\mathcal{S}$ is the collection of all {\emph{pro-isomorphic}} subgroups of $G$; a subgroup $H \leq G$ is called pro-isomorphic if its profinite completion is isomorphic to that of $G$.  One writes $\zeta^\leq_G (s)$, $\zeta^\vartriangleleft_G (s)$, and $\zeta^\wedge_G (s)$ for the zeta function~\eqref{equ:zeta.dfn}, where $\mathcal{S}$ is the set of all subgroups, normal subgroups, and pro-isomorphic subgroups, respectively.  This notation carries over to the local zeta functions~\eqref{equ:zeta.local}.
A rich theory of zeta functions of groups has been developed in the special case where $G$ is finitely generated, torsion-free, and nilpotent -- hereafter, a $\mathcal{T}$-group.

\subsection{Lie rings and uniformity}
A {\emph{Lie ring}} over a ring $R$ is a finitely generated $R$-module endowed with an $R$-bilinear anti-commutative multiplication satisfying the Jacobi identity.  A {\emph{Lie lattice}} over $R$ is a Lie ring over $R$ that is free as an $R$-module.  We refer to Lie lattices over $\Z$ simply as Lie lattices.  It is convenient to ``linearize'' by translating counting problems for subgroups of a $\mathcal{T}$-group to counting problems for subrings of a Lie ring; techniques of linear algebra may then be applied.  
Given a Lie ring $\mathcal{L}$, let $b_n(\mathcal{L})$ be the number of Lie subrings $\mathcal{M} \leq \mathcal{L}$ of index $n$ such that, for every prime $p$, there is an isomorphism $\mathcal{M} \tensor_\Z \Z_p \simeq \mathcal{L} \tensor_\Z \Z_p$ of Lie rings over $\Z_p$.  Define the global and local pro-isomorphic zeta functions
$$\begin{array}{lcr}
\zeta_{\mathcal{L}}^\wedge (s) = \sum_{n = 1}^\infty b_n (\mathcal{L}) n^{-s}, & &
\zeta^\wedge_{\mathcal{L},p} (s) = \sum_{k = 0}^\infty b_{p^k} (\mathcal{L}) p^{-ks}. \end{array}$$
One obtains an Euler decomposition $\zeta^\wedge_{\mathcal{L}} (s) = \prod_p \zeta^\wedge_{\mathcal{L},p}(s)$ as a consequence of the Chinese remainder theorem.
For every $\mathcal{T}$-group $G$ there exists a Lie ring $\mathcal{L}(G)$ such that for almost all primes (and for all primes if $G$ is of class two) one has 
$\zeta^\wedge_{G,p} (s) = \zeta^\wedge_{\mathcal{L}(G),p} (s)$~\cite[\S4]{GSS/88}; see Section~\ref{sec:linearization} below for more details.  
In this paper we will work directly with zeta functions of Lie rings rather than those of groups.

If $K$ is a number field of degree $d=[K:\Q]$ with ring of integers $\mathcal{O}_K$, and $\mathcal{L}$ is a Lie lattice of rank $\dimL$, then we may view $\mathcal{L} \tensor_\Z \mathcal{O}_K$ as a Lie lattice over $\Z$ of rank $d \dimL$. Our central aim is to investigate how the local pro-isomorphic zeta functions $\zeta^\wedge_{\mathcal{L} \tensor_\Z \mathcal{O}_K, p}(s)$ vary with $K$ and $p$. Prior to our work, complete answers to this question were known only for free nilpotent Lie rings~\cite{GSS/88}.  The idea of considering base extensions is implicit in work of du Sautoy and Lubotzky~\cite{duSLubotzky/96}; see Remark~\ref{rem:duSL} below for a comparison of their approach with ours.

We now discuss our main results.  Following~\cite[\S1.5]{Vergne/70}, a nilpotent Lie lattice of rank $n$ is called filiform if it has class $n-1$, which is the largest possible; these also appear in the literature under the name of maximal class Lie lattices.  

\begin{thm} \label{thm:intro.examples}
Let $\mathcal{L}$ be a nilpotent Lie lattice of one of the following types:
\begin{itemize}
\item Free nilpotent Lie lattices of class $c$ on $g$ generators, for $c \geq 2$ and $g \geq 1$; see Section~\ref{sec:free}.
\item Higher Heisenberg Lie lattices; see Section~\ref{sec:higher.heisenberg}.
\item The Lie lattices $\mathcal{L}_{m,n}$ introduced in~\cite{BKO/Dstar}; see Section~\ref{sec:lmn}.
\item A filiform Lie lattice of class $c \geq 2$ arising in a certain family; see Section~\ref{sec:max.class}.
\item The filiform Lie lattice $\mathcal{F}_4$ of class $4$; see Section~\ref{sec:filiform}.
\item The Lie lattice $\mathcal{Q}_5$ of class $3$ and rank $5$; see Section~\ref{sec:Q5}.
\item A Lie lattice of class $4$ constructed in~\cite{BK/15} whose local pro-isomorphic zeta functions do not satisfy functional equations; see Section~\ref{sec:bk}.
\end{itemize}
Let $d \in \mathbb{N}$.  Then there exists an explicit rational function $W_{\mathcal{L},d}(X,Y) \in \mathbb{Q}(X,Y)$, depending only on $\mathcal{L}$ and $d$, such that for any number field $K$ of degree $d$ and any rational prime $p$, the following holds:
\begin{equation} \label{equ:mini.euler.product}
\zeta^\wedge_{(\mathcal{L} \tensor_\Z \mathcal{O}_K),p}(s) = \prod_{\p | p} W_{\mathcal{L},d}(q_{\p}, q_{\p}^{-s}),
\end{equation}
where $q_{\p}$ is the cardinality of the residue field $\mathcal{O}_K / \p$.
\end{thm}

For each Lie lattice in the list, we compute the explicit functions $W_{\mathcal{L},d}$ in Section~\ref{sec:examples}.  In the case of free nilpotent Lie lattices, this result was obtained more than thirty years ago by Grunewald, Segal, and Smith~\cite[Theorem~7.1]{GSS/88}.
It is interesting to observe that, for all the Lie lattices $\mathcal{L}$ studied in Theorem~\ref{thm:intro.examples}, the dependence of the rational functions $W_{\mathcal{L},d}$ on $d$ is quite tame: they have the form
\begin{equation} \label{equ:rat.funct.form}
W_{\mathcal{L},d}(X,Y) = \frac{ \sum_{j = 1}^M X^{A_{j,0} + dA_{j,1}} Y^{B_j} }{\prod_{j = 1}^N (1 - X^{C_{j,0} + dC_{j,1}} Y^{D_{j}})}
\end{equation}
for some $M,N \in \N$ and for integers $A_{j,0}, A_{j,1}, B_j, C_{j,0}, C_{j,1}, D_j \geq 0$.

A striking feature of the pro-isomorphic zeta functions $\zeta^\wedge_{\mathcal{L} \tensor_\Z \mathcal{O}_K}(s)$ seen in Theorem~\ref{thm:intro.examples} is the existence of a fine Euler decomposition, namely a decomposition into a product whose factors are naturally indexed by primes of $K$ rather than rational primes; factorizations as in~\eqref{equ:mini.euler.product} are called mini Euler products by some authors.  Another property is finite uniformity.  More precisely, fix a number field $K$ of degree $d$, and let $\mathbf{e} = (e_1, \dots, e_\npr)$ and $\mathbf{f} = (f_1, \dots, f_\npr)$ be $\npr$-tuples satisfying $\sum_{i = 1}^\npr e_i f_i = d$.  We say that a prime $p$ has decomposition type $(\mathbf{e}, \mathbf{f})$ in $K$ if $p \mathcal{O}_K = \p_1^{e_1} \cdots \p_\npr^{e_\npr}$, where $[ \mathcal{O}_K / \p_i : \mathbb{F}_p] = f_i$ for all $1 \leq i \leq \npr$.  Noting that $q_\p = p^{[ \mathcal{O}_K / \p : \mathbb{F}_p]}$ for any prime $\p$ of $K$ dividing $p$, we observe:
\begin{cor} \label{cor:finitely.uniform}
Let $\mathcal{L}$ be a Lie lattice to which Theorem~\ref{thm:intro.examples} applies, and let $\mathbf{e}, \mathbf{f} \in \N^r$.  There exists a rational function $W_{\mathcal{L}, \mathbf{e}, \mathbf{f}}(X,Y) \in \Q(X,Y)$ such that for any number field of degree $d = \sum_{i = 1}^\npr e_i f_i$ and for any prime $p$ of decomposition type $(\mathbf{e}, \mathbf{f})$ in $K$, the following holds:
$$ \zeta^\wedge_{(\mathcal{L} \tensor_\Z \mathcal{O}_K),p}(s) = W_{\mathcal{L}, \mathbf{e}, \mathbf{f}}(p, p^{-s}).$$
\end{cor}
\begin{proof}
Set $W_{\mathcal{L}, \mathbf{e}, \mathbf{f}}(X,Y) = \prod_{i = 1}^\npr W_{\mathcal{L},d}(X^{f_i},Y^{f_i})$ and apply Theorem~\ref{thm:intro.examples}.
\end{proof}
A global zeta function whose local factors are described by finitely many rational functions is called finitely uniform.  Finitely many decomposition types of primes appear in any given number field $K$, so Corollary~\ref{cor:finitely.uniform} implies the finite uniformity of $\zeta^\wedge_{\mathcal{L} \tensor_\Z \mathcal{O}_K}(s)$ for the Lie lattices considered in Theorem~\ref{thm:intro.examples}.  There exist Lie lattices whose zeta functions counting subrings and ideals are not finitely uniform~\cite{duS-ecI/01}; see also~\cite{Voll/04}.  It is not known whether such examples occur in the pro-isomorphic setting.

\subsection{Rigidity and the fine Euler decomposition}
The fine Euler decompositions appearing in Theorem~\ref{thm:intro.examples} occur more generally than for the list of Lie lattices considered there.   Let $\mathcal{L}$ be a Lie lattice and let $p$ be a rational prime.  Let $Z(L_p)$ denote the center of the Lie algebra $L_p = \mathcal{L} \tensor_\Z \Q_p$.  We say that $L_p$ is {\emph{$Z(L_p)$-rigid}} if, for any finite extension $F/\Q_p$, the $\Q_p$-automorphisms of $L_p \tensor_{\Q_p} F$ are essentially determined by those of $L_p$; see Section~\ref{sec:Overview} for a more detailed discussion and Definition~\ref{def:rigidity} for a precise description of this property, which was first studied by Segal~\cite{Segal/89} and is crucial for the current work.  

Given a Lie lattice $\mathcal{L}$, 
let $\mathcal{P}(\mathcal{L})$ be the set of rational primes $p$ such that $L_p$ is $Z(L_p)$-rigid.
For any $p \in \mathcal{P}(\mathcal{L})$, we prove that the local pro-isomorphic zeta function $\zeta^\wedge_{\mathcal{L} \tensor \mathcal{O}_K, p}(s)$ splits into a product, indexed by primes $\p $ of $K$ dividing $p$, of rational functions in $q_\p^{-s}$, where $q_\p$ is the cardinality of the residue field $k_\p = \mathcal{O}_K / \p$.

\begin{thm} \label{thm:uniformity}
Let $\mathcal{L}$ be a Lie lattice such that $Z(\mathcal{L}) \leq \gamma_2 \mathcal{L}$.
\begin{enumerate}
\item
Let $d \in \N$.  There exists a natural number $N_d$ and pairs of integers $(a_i^{(d)}, b_i^{(d)})$ for every $1 \leq i \leq N_d$ with $b_i^{(d)} \neq 0$, such that for every $p \in \mathcal{P}(\mathcal{L})$,
every number field $K$ of degree $d$, and every prime $\p | p$ of $K$, there is a polynomial $V_\p \in \Q[X]$ for which the following holds for all $s \in \C$ with $\mathrm{Re} \, s \gg 0$:
$$ \zeta^\wedge_{\mathcal{L} \tensor \mathcal{O}_K, p}(s) = \prod_{\p | p} \frac{V_\p (q_\p^{-s})}{\prod_{i = 1}^{N_d} (1 - q_\p^{a_i^{(d)} + b_i^{(d)} s})}.$$
\item
There exist $\unifmax, e \in \N$, a collection of $e$-ary formulae $\psi_{1}, \dots, \psi_{\unifmax}$ in the language of rings, and families of rational functions $W_1^{(d)}(X,Y), \dots, W_{\unifmax}^{(d)}(X,Y) \in \Q(X,Y)$, parametrized by $d \in \N$, of the form
\begin{equation*} \label{equ:general.form}
W_{i}^{(d)}(X,Y) = \frac{ \sum_{j = 1}^{M_i} \kappa_{ij} X^{A_{ij}^{(0)} + dA_{ij}^{(1)}} Y^{B_{ij}} }{\prod_{j = 1}^{N_i} (1 - X^{C_{ij}^{(0)} + dC_{ij}^{(1)}} Y^{D_{ij}})},
\end{equation*}
where $\kappa_{ij}$, $A_{ij}^{(0)}$, $A_{ij}^{(1)}$, $B_{ij}$, $C_{ij}^{(0)}$, $C_{ij}^{(1)}$, $D_{ij} \in \Z$, such that the following holds for all but finitely many $p \in \mathcal{P}(\mathcal{L})$:

Let $K$ be any number field, and let $d = [K:\Q]$ be its degree.  For all $s \in \C$ with $\mathrm{Re} \, s \gg 0$ we have
$$\zeta^\wedge_{\mathcal{L} \tensor \mathcal{O}_K, p}(s) = \prod_{\p | p} \left( \sum_{i = 1}^{\unifmax} m_i (k_\p) W_i^{(d)}(q_\p, q_\p^{-s}) \right),$$
where $m_i (k_\p) = | \{ \xi \in k_\p^{e} : k_\p \models \psi_i(\xi) \} |$.  
\end{enumerate}
\end{thm}

We remark upon some features of this theorem, whose proof interprets the factors at $\p | p$ as suitable $p$-adic integrals and then applies methods of $p$-adic integration.
An analogous expression for local pro-isomorphic zeta functions at rational primes $p$, in terms of rational functions and numbers of rational points on definable sets, is given in~\cite[Theorem~7.2]{HMR/18}, and similar results hold~\cite{duSG/00, HMR/18} for local factors of other group-theoretic zeta functions.  However, the \emph{linear} dependence on $d = [K:\Q]$ of the exponents of $X$ in the functions $W_i^{(d)}(X,Y)$ is new and specific to pro-isomorphic zeta functions under base extension.  It follows from the proof that the right half-plane on which the conclusion of Theorem~\ref{thm:uniformity}(2) holds depends on $[K:\Q]$ but not on $p$.

Observe that the first part of Theorem~\ref{thm:uniformity} applies for all $p \in \mathcal{P}(\mathcal{L})$ but provides no control over the numerators of the factors of $ \zeta^\wedge_{\mathcal{L} \tensor \mathcal{O}_K, p}(s)$, while the dependence of the denominators  on the degree $[K:\Q]$ is unspecified.  The second part of Theorem~\ref{thm:uniformity}, by contrast, provides much more control over the shape of the factors and their dependence on $[K : \Q]$, but at the cost of excluding finitely many primes.  We expect that the denominator in the first part of Theorem~\ref{thm:uniformity} should depend on $d$ in the same way as the denominator in the second part; see Remark~\ref{rem:variants}.

If $\mathcal{P}(\mathcal{L})$ is the set of all rational primes, then Theorem~\ref{thm:uniformity} ensures a fine Euler decomposition for all the functions $\zeta^\wedge_{\mathcal{L} \tensor \mathcal{O}_K}(s)$.  This occurs for the Lie lattices in Theorem~\ref{thm:intro.examples}.  The
uniformity of the local factors at $\p$ exhibited there amounts to the cardinalities $m_i^{(d)}(k_\p)$ being expressible as rational functions in $q_\p = |k_\p |$.  Theorem~\ref{thm:uniformity}(1) refines a consequence of an earlier result of du Sautoy~\cite[Corollary~3.3]{duS/94}, based on work of Macintyre~\cite{Macintyre/90}.  
The theorem of du Sautoy states that, given $n \in \N$, there exists some integer $N$ and pairs of integers $(a_i, b_i)$, for $1 \leq i \leq N$, such that $\zeta^\wedge_{\mathcal{L},p}(s)$ is a rational function in $p^{-s}$ with denominator $\prod_{i = 1}^N (1 - p^{a_i + b_is})$ for {\emph{all}} rings $\mathcal{L}$ whose additive group is isomorphic to $\Z^n$, and {\emph{all}} rational primes $p$.

\subsection{Functional equations} \label{sec:intro.funct.eq}
Let $\mathcal{L}$ be a Lie ring, let $S$ be a set of primes, and suppose that there exists a rational function $W(X,Y) \in \Q(X,Y)$ describing local pro-isomorphic zeta functions of $\mathcal{L}$, in the sense that $\zeta^\wedge_{\mathcal{L},p}(s) = W(p,p^{-s})$ for all $p \in S$.  If a relation of the form $W(X^{-1},Y^{-1}) = (-1)^m X^a Y^b W(X,Y)$ holds for some $a,b,m \in \Z$, then we say that $\zeta^\wedge_{\mathcal{L},p}(s)$ satisfies a functional equation with symmetry factor $(-1)^m p^{a-bs}$ for all $p \in S$, written 
\begin{equation} \label{equ:def.funct.eq}
\zeta^\wedge_{\mathcal{L},p}(s) |_{p \to p^{-1}} = (-1)^m p^{a-bs} \zeta^\wedge_{\mathcal{L},p}(s).
\end{equation}
If $S$ is finite, then $W(X,Y)$ is not unique and the symmetry factor in~\eqref{equ:def.funct.eq} depends on the choice of $W(X,Y)$.
See~\cite{Voll/10} for a more general notion of functional equations.  Suppose that there exists a rational function $W(X,Y)$ and a number field $K$ such that $\zeta^\wedge_{\mathcal{L},p}(s) = \prod_{\p | p} W(q_\p, q_\p^{-s})$ for all rational primes $p$; this occurs, for instance, if $\mathcal{L} = \mathcal{M} \tensor_\Z \mathcal{O}_K$, where $\mathcal{M}$ is a Lie lattice from Theorem~\ref{thm:intro.examples}.  If $W(X^{-1}, Y^{-1}) = (-1)^m X^a Y^b W(X,Y)$, then
\begin{equation} \label{equ:unram.funct.eq}
 \zeta^\wedge_{\mathcal{L},p}(s) |_{p \to p^{-1}} = (-1)^{m\npr} \left( \prod_{\p | p} q_\p \right)^{a-bs} \zeta^\wedge_{\mathcal{L},p}(s)
 \end{equation}
where $\npr$ is the number of primes $\p$ dividing $p$, as above.
The symmetry factor depends only on the decomposition type of $p$ in $K$, and we may speak of a functional equation at all $p$.
In particular, the symmetry factor is $(-1)^{m\npr} p^{[K:\Q](a-bs)}$ for almost all primes, namely the ones unramified in $K$.

The existence of functional equations apparently holds some important, although mysterious, clue about the structure of a Lie ring, or of an associated finitely generated group. Indeed, for other types of zeta functions some very general results were obtained by Voll~\cite{Voll/10, Voll/17}, in which local functional equations with specified symmetry factors were established for large classes of rings; see also~\cite{LeeVoll/20}. Nevertheless, there are known examples of Lie lattices whose local ideal zeta functions do not satisfy functional equations; the first were found by Woodward~\cite[Theorems~2.32 and~2.74]{duSWoodward/08}.  For local pro-isomorphic zeta functions, sufficient conditions for the existence of functional equations were proved in~\cite[\S6]{duSLubotzky/96} and~\cite[Theorem~1.1]{Berman/11}.  On the other hand, a Lie lattice whose local pro-isomorphic zeta functions do not satisfy functional equations was constructed by Klopsch and the first author~\cite{BK/15}. 

The first author, Klopsch and Onn have recently formulated the following conjecture regarding local functional equations of pro-isomorphic zeta functions of Lie lattices.  A grading on a Lie lattice $\mathcal{L}$ is a decomposition
$$ \mathcal{L} = \bigoplus_{i \geq 1} \mathcal{L}_i$$
of the underlying free $\Z$-module satisfying $[\mathcal{L}_i, \mathcal{L}_j] \subseteq \mathcal{L}_{i+j}$ for all $i,j \in \N$.  Denote the associated filtration by $\{ \mathcal{F}_i = \bigoplus_{j \geq i} \mathcal{L}_j \}$.  To each grading we associate a weight $\sum_{i\geq 1} i\rk_\Z \mathcal{L}_i = \sum_{i \geq 1} \rk_\Z \mathcal{F}_i$.  Let $\mathrm{wt}(\mathcal{L})$ be the minimal weight among all gradings on $\mathcal{L}$.  We call a grading minimal if its weight is $\mathrm{wt}(\mathcal{L})$ and natural if $\{ \mathcal{F}_i \}$ coincides with the lower central series of $\mathcal{L}$, and we say that $\mathcal{L}$ is naturally graded if it admits a natural grading.  Natural gradings, if they exist, are minimal.

\begin{con} \label{graded.conj}\cite[Conjecture~1.3]{BKO/even}
Let $\mathcal{L}$ be a graded Lie lattice.  Suppose that there exist rational functions $W_1, \dots, W_k \in \Q(X,Y)$ such that, for almost all primes $p$, there exists $1 \leq i \leq k$ satisfying $\zeta^\wedge_{\mathcal{L},p}(s) = W_i(p,p^{-s})$.  
Suppose that the $W_i$ satisfy functional equations with uniform symmetry factor, namely that
there exist $a,b,m \in \Z$ such that 
\begin{equation} \label{equ:intro.funct.eq}
W_i(X^{-1}, Y^{-1}) = (-1)^{m} p^{a - b s} W_i(X,Y)
\end{equation}
for all $1 \leq i \leq k$.  Then $b = \mathrm{wt}(\mathcal{L})$.	
\end{con}

 Any grading on a Lie lattice $\mathcal{L}$ induces a grading on any base extension $\mathcal{L} \tensor \mathcal{O}_K$, viewed as a Lie lattice over $\Z$, and it is easy to see that $\mathrm{wt}(\mathcal{L} \tensor \mathcal{O}_K) = [K:\Q] \mathrm{wt}(\mathcal{L})$.  All the Lie lattices $\mathcal{L}$ considered in Theorem~\ref{thm:intro.examples} are graded.  The first four families, and consequently their base extensions, are naturally graded; by constrast, the remaining three Lie lattices are not.  For $\mathcal{L}$ in any one of the first six families,  
 and all $d \in \N$, our computations in Section~\ref{sec:examples} show that the rational functions $W_{\mathcal{L},d}$ admit functional equations as in~\eqref{equ:intro.funct.eq} with $b = \mathrm{wt}(\mathcal{L})$.    By~\eqref{equ:unram.funct.eq}, Conjecture~\ref{graded.conj} holds for the base extensions $\mathcal{L} \tensor \mathcal{O}_K$, where $K$ is any number field.

\subsection{Other zeta functions} \label{sec:siblings}
It is interesting to compare Theorems~\ref{thm:intro.examples} and~\ref{thm:uniformity} with analogous results 
for other group- and ring-theoretic zeta functions under base extension.  While fine Euler decompositions as in Theorem~\ref{thm:uniformity} only occur in exceptional cases for ideal zeta functions, they are the norm for representation zeta functions.  Given a group $G$, consider $\zeta^\mathrm{irr}_{G}(s) = \sum_{n = 1}^\infty \widetilde{r}_n n^{-s}$, where $\widetilde{r}_n$ is the number of isomorphism classes of $n$-dimensional complex representations of $G$, up to twist by one-dimensional representations.  If $G = \boldsymbol{\Gamma}(\mathcal{O}_K)$ for a unipotent group scheme $\boldsymbol{\Gamma}$ as in~\cite[\S2.1]{StasinskiVoll/14}, where $K$ is a number field, then $G$ is a $\mathcal{T}$-group, the sequence $\{ \widetilde{r}_n \}$ grows polynomially, and the Dirichlet series above converges in some right half-plane.  For such $G = \boldsymbol{\Gamma}(\mathcal{O}_K)$, Stasinski and Voll showed~\cite[Proposition~2.2]{StasinskiVoll/14} that $\zeta^\mathrm{irr}_{G}(s)$ has a fine Euler decomposition indexed by the primes of $\mathcal{O}_K$ and that its local factors are ``nearly uniform''~\cite[Theorem~A]{StasinskiVoll/14}.  Behind the scenes lies a well-behaved count of points on varieties; see also~\cite[Theorem~8.13]{HMR/18}.  If $\boldsymbol{\Gamma}$ lies in one of three families of unipotent group schemes associated to Lie lattices considered in~\cite[Theorem~B]{StasinskiVoll/14}, then there are explicit rational functions $W_{\boldsymbol{\Gamma}}(X,Y)$ such that
$ \zeta^{\mathrm{irr}}_{\boldsymbol{\Gamma}(\mathcal{O}_K)}(s) = \prod_{\p} W_{\boldsymbol{\Gamma}}(q_\p, q_\p^{-s})$
for all number fields $K$.  This is analogous to our Theorem~\ref{thm:intro.examples}, except that the functions $W_{\boldsymbol{\Gamma}}$ are independent of the degree $d = [K:\Q]$.  In fact, they have the form~\eqref{equ:rat.funct.form} with $A_{j,1} = C_{j,1} = 0$ for all $j$.

For ideal zeta functions, the most general statement currently available is~\cite[Theorem~4.21]{CSV/19}, which explicitly describes the local ideal zeta functions $\zeta^\vartriangleleft_{(\mathcal{L} \tensor_\Z \mathcal{O}_K),p}(s)$, when $p$ is unramified in $K$, for a collection of nilpotent Lie lattices $\mathcal{L}$ of class two that includes the free nilpotent Lie lattices of class two and the higher Heisenberg Lie lattices.  It includes some, but not all, members of the family studied in Section~\ref{sec:lmn}, and includes Lie lattices not covered by Theorem~\ref{thm:intro.examples}.  All these local ideal zeta functions satisfy functional equations with the symmetry factor specified by~\cite[Theorem~C]{Voll/10}.  See also~\cite{CSV/20} for a summary of the results of~\cite{CSV/19}.  It is conjectured that for 
any Lie lattice $\mathcal{L}$ considered in~\cite{CSV/19} and \emph{any} decomposition type $(\mathbf{e}, \mathbf{f})$ there exists an explicit rational function $W^\vartriangleleft_{\mathcal{L}, \mathbf{e}, \mathbf{f}}(X,Y)$ defined in terms of combinatorial data such that $\zeta^\vartriangleleft_{\mathcal{L} \tensor \mathcal{O}_K, p}(s) = W^\vartriangleleft_{\mathcal{L}, \mathbf{e}, \mathbf{f}}(p,p^{-s})$ for any prime $p$ with decomposition type $(\mathbf{e}, \mathbf{f})$ in $K$.  Corollary~\ref{cor:finitely.uniform} implies the analogue of this conjecture for pro-isomorphic zeta functions of some Lie lattices.  It is conjectured that all the functions $W^\vartriangleleft_{\mathcal{L}, \mathbf{e}, \mathbf{f}}(X,Y)$ satisfy functional equations and that the symmetry factors at ramified types are different from the generic one, but that their dependence on the decomposition type is not the same as in~\eqref{equ:unram.funct.eq}.  Finally, observe that Theorem~\ref{thm:intro.examples} covers two families of Lie lattices of unbounded nilpotency class.  By contrast, no subring or ideal zeta functions are known for any nilpotent Lie ring of class five or greater.

\subsection{Abscissa of convergence} \label{sec:abscissa}
An important analytic invariant of the global pro-isomorphic zeta function $\zeta^\wedge_{\mathcal{L}}(s)$ is its abscissa of convergence $\alpha^\wedge_{\mathcal{L}} = \inf S$, where $S$ is the set of real numbers $\beta$ such that $\zeta^\wedge_{\mathcal{L}}(s)$ converges on the right half-plane $\mathrm{Re} \, s > \beta$.  
The number $\alpha^\wedge_{\mathcal{L}}$ reflects the algebraic structure of $\mathcal{L}$ as a bound on the polynomial rate of growth of the sequence $s_n(\mathcal{L}) = b_1(\mathcal{L}) + \cdots + b_n(\mathcal{L})$.  Indeed, by a basic result in the theory of Dirichlet series, $\alpha^\wedge_{\mathcal{L}} = \limsup_{n \to \infty} \frac{\log s_n(\mathcal{L})}{\log n}$; see, for instance,~\cite[Theorem~1.3]{MV/07}.  For many of the Lie lattices $\mathcal{L}$ of Theorem~\ref{thm:intro.examples}, the abscissa of convergence $\alpha^\wedge_{\mathcal{L} \tensor_\Z \mathcal{O}_K}$ is given by a linear function in the degree $d = [K : \Q]$.  The abscissa of convergence $\alpha^\wedge_{\mathcal{L}}$ is a rational number by the proof of~\cite[Theorem~7.2]{HMR/18}; the analogous statement for the subring and ideal zeta functions was known earlier~\cite{duSG/00}.  

It is natural to ask whether $\zeta^\wedge_{\mathcal{L}}(s)$ may be meromorphically continued to a larger right half-plane $\mathrm{Re} \, s > \alpha^\wedge_{\mathcal{L}} - \delta$ for some $\delta > 0$.  If this is possible, then one can deduce a more precise description of the asymptotic growth of $s_n(\mathcal{L})$, as was done in~\cite[Theorem~1.7]{duSG/00} for the subring and ideal zeta functions and in~\cite[Theorem~A]{DungVoll/17} for representation zeta functions.

\subsection{Overview} \label{sec:Overview}
We will now give a brief overview of the methods of this paper.  A basic property of local pro-isomorphic zeta functions of Lie lattices is that they are readily interpreted as certain $p$-adic integrals.  
Indeed, let $\mathcal{L}$ be a Lie lattice and $p$ a prime.  As before, write $\mathcal{L}_p = \mathcal{L} \tensor_\Z \Z_p$ and $L_p = \mathcal{L} \tensor_\Z \Q_p$.  Let $\mathbf{G} = \aAut L_p$ be the algebraic automorphism group of $L_p$, so that $\mathbf{G}(F) = \Aut_F (L_p \tensor_{\Q_p} F)$ for any field extension $F/\Q_p$.  Let $\mathbf{G}(\Z_p)$ be the subgroup of $\mathbf{G}(\Q_p) = \Aut_{\Q_p} L_p$ consisting of elements that restrict to automorphisms of the $\Z_p$-lattice $\mathcal{L}_p$, and let $\mathbf{G}_p^+ \subset \mathbf{G}(\Q_p)$ be the submonoid of elements preserving $\mathcal{L}_p$.  Throughout this paper, automorphisms act from the right.  The map $\varphi \mapsto (\mathcal{L}_p)\varphi$ induces a bijection between the set $\mathbf{G}(\Z_p) \backslash \mathbf{G}_p^+$ of right cosets and the set of subrings of $\mathcal{L}_p$ that are isomorphic to $\mathcal{L}_p$ as Lie rings over $\Z_p$.  Let $\mu_p$ be the right Haar measure on $\mathbf{G}(\Q_p)$, normalized so that $\mu_p(\mathbf{G}(\Z_p)) = 1$.  Observe that $[ \mathcal{L}_p : (\mathcal{L}_p)\varphi] = | \det \varphi |_{\Q_p}^{-1}$, where the multiplicative valuation on $\Q_p$ is normalized so that $|p|_{\Q_p} = p^{-1}$.  We obtain the following relation, first appearing in~\cite[Proposition 3.4]{GSS/88}.

\begin{pro}[Grunewald-Segal-Smith] \label{pro:padic.integral}
For each prime $p$ and $s \in \mathbb{C}$ such that $\zeta^\wedge_{\mathcal{L}_p}(s)$ converges, we have the equality
\begin{equation} \label{equ:padic.integral}
\zeta^\wedge_{\mathcal{L}_p} (s) = \int_{\mathbf{G}_p^+} | \det (g) |_{\Q_p}^{s} d \mu_p (g).
\end{equation}
Moreover, the integral converges if and only if $\zeta^\wedge_{\mathcal{L}_p}(s)$ does.
\end{pro}
Integrals of the type appearing on the right-hand side of~\eqref{equ:padic.integral} have been studied by many authors, including Hey~\cite{Hey/29}, Satake~\cite{Satake/63}, Tamagawa~\cite{Tamagawa/63}, and Macdonald~\cite{Macdonald/95}. 

In order to study the local pro-isomorphic zeta functions of $\mathcal{L} \tensor_\Z \mathcal{O}_K$, it is thus essential to understand the automorphism groups of the $\Q_p$-Lie algebras $(\mathcal{L} \tensor_\Z \mathcal{O}_K) \tensor_\Z \Q_p = {L}_p \tensor_{\Q_p} (K \tensor_\Q \Q_p)$.  The crucial tool of the present paper is a certain rigidity property that was already identified by Segal~\cite{Segal/89} and used in~\cite{GSS/88} for the study of the pro-isomorphic zeta functions of base extensions of the free nilpotent Lie lattices.  To illustrate this property, write $R$ for the $\Q_p$-algebra $K \tensor_\Q \Q_p$, and observe that $\Aut_{\Q_p} (L_p \tensor_{\Q_p} R)$ certainly contains automorphisms of the following three types:
\begin{itemize}
\item $R$-linear automorphisms; these are understood if we know the algebraic automorphism group $\aAut L_p$.
\item Automorphisms that are trivial modulo the center of $L_p \tensor_{\Q_p} R$; these are easy to describe given the structure of $\mathcal{L}$.
\item Automorphisms induced by $\Q_p$-automorphisms of $R$; these form a finite group.
\end{itemize}
If, for any finite-dimensional semisimple $\Q_p$-algebra $R$, the group $\Aut_{\Q_p}(L_p \tensor_{\Q_p} R)$ is generated by automorphisms of the above three types, i.e.~is the smallest possible, then we say that $L_p$ is rigid over its center; see Definition~\ref{def:rigidity} below.  In Theorem~\ref{thm:rigidity} below we generalize a sufficient condition for rigidity proved by Segal in~\cite{Segal/89} and thereby establish rigidity, in particular, for most of the Lie algebras arising from the Lie rings in Theorem~\ref{thm:intro.examples}.  Section~\ref{sec:Q5} provides an example of a Lie algebra that, while rigid, fails to satisfy the criterion of Theorem~\ref{thm:rigidity}.

Fix a decomposition $\mathbf{G} = \mathbf{H} \ltimes \mathbf{N}$, where $\mathbf{N}$ is the unipotent radical of $\mathbf{G}$ and $\mathbf{H} \subseteq \mathbf{G}$ is a reductive subgroup.  The integral of Proposition~\ref{pro:padic.integral} may be reformulated as an integral over a suitable submonoid of $\mathbf{H} (\Q_p)$, at the cost of replacing the integrand by a more complicated function.  This simplified domain of integration allows the integral to be computed by means of a $p$-adic Bruhat decomposition.  
Moreover, du Sautoy and Lubotzky~\cite{duSLubotzky/96} have shown that, under a series of simplifying hypotheses, the integrand may be expressed as a product of functions, each of which can often be computed explicitly. 
The principal technical result of our paper is Proposition~\ref{pro:rigid.computation}, which says, roughly speaking, that if $L_p$ is rigid over its center, then the local pro-isomorphic zeta function $\zeta^\wedge_{\mathcal{L} \tensor \mathcal{O}_K, p}(s)$ splits into a product indexed by the primes of $K$ dividing $p$.  This is enough to imply Theorem~\ref{thm:uniformity} using techniques from $p$-adic integration.  
If, in addition, the simplifying hypotheses of~\cite{duSLubotzky/96} hold, then Corollary~\ref{cor:rigid.lifting} shows that the computation of $\zeta^\wedge_{\mathcal{L} \tensor \mathcal{O}_K, p}(s)$ is essentially the same as that of $\zeta^\wedge_{\mathcal{L},p}(s)$, with minor alterations depending on the degree $d = [K:\Q]$.

In the case of the higher Heisenberg Lie lattices, the explicit rational functions given in Theorem~\ref{pro:higher.final} below are new even when $K = \Q$, except for several small cases.  Along the way we prove a combinatorial identity involving the hyperoctahedral group, Lemma~\ref{lem:bm.to.sm}, that may have independent interest.  For the remaining examples mentioned in Theorem~\ref{thm:intro.examples}, the pro-isomorphic local zeta functions in the case $K = \Q$ were known previously; their computation is straightforward in some cases and quite intricate in others.  In Section~\ref{sec:examples} we apply Corollary~\ref{cor:rigid.lifting} to show that the same calculations, with slight modifications, treat arbitrary number fields $K$.

\subsection{Organization of the paper} 
In Section~\ref{sec:methodology}, we review the framework that will be used to analyze the integrals of Proposition~\ref{pro:padic.integral}; most of this material is due to du Sautoy and Lubotzky~\cite{duSLubotzky/96}.  In Section~\ref{sec:rigidity}, we define the rigidity property mentioned above and prove a sufficient condition for it to hold; this criterion generalizes one of Segal~\cite{Segal/89}, and its proof is a modification of Segal's argument.  Rigidity is used to deduce Proposition~\ref{pro:rigid.computation} and then combined with the setup of Section~\ref{sec:methodology} to prove Corollary~\ref{cor:rigid.lifting}.  
Some Lie lattices that do not satisfy the rigidity property are exhibited.  
In Section~\ref{sec:appendix} we deduce Theorem~\ref{thm:uniformity} from Proposition~\ref{pro:rigid.computation} using an argument of Grunewald, Segal, and Smith~\cite[\S2]{GSS/88} and variations of results on $p$-adic integration~\cite{BDOP/11, CluckersHalupczok/18}.
Finally, in Section~\ref{sec:examples} we study the Lie lattices $\mathcal{L}$ listed in Theorem~\ref{thm:intro.examples}, show that $L_p = \mathcal{L} \tensor_\Z \Q_p$ is rigid over its center for all primes $p$, and use Corollary~\ref{cor:rigid.lifting} to compute the pro-isomorphic zeta functions explicitly.  

\subsection{Notation}
For any finite extension $F / \Q_p$, the valuation $| \cdot |_F$ is normalized so that the valuation of a uniformizer is $1/q$, where $q$ is the cardinality of the residue field of $F$.  If $K$ is a number field, then $V_K$ will denote the set of non-Archimedean places of $K$.  If $\p \in V_K$, we let $K_\p$ denote the localization of $K$ at $\p$.  Let $\mathcal{O}_\p$ and $q_\p$ denote its ring of integers and residue cardinality, respectively, and write $|  \cdot  |_{\p}$ for its normalized valuation.

We write $\N = \{ 1, 2, 3, \dots \}$ and $\N_0 = \N \cup \{ 0 \}$.
For any $n \in \N$, we set $M_n$ to be the multiplicative algebraic monoid, over $\Z$, of $n \times n$ matrices.  If $m,n \in \N$, then $M_{m,n}$ denotes the additive algebraic group, over $\Z$, of $m \times n$ matrices.  We denote the set $\{ 1, 2, \dots, n \}$ by $[n]$ and the set $\{ 0, 1, \dots, n \}$ by $[n]_0$.  If $m$ and $n$ are integers, we write $[m,n]$ for the set $\{ k \in \Z : m \leq k \leq n \}$.

\subsection{Acknowledgements}
The third author was supported by grant 1246/2014 from the German-Israeli Foundation for Scientific Research and Development.  We are grateful to Yotam Hendel, Boris Kunyavskii, Christopher Voll, and the anonymous referee for helpful conversations and suggestions.

\section{Preliminaries} \label{sec:methodology}
In this section we recall some basic facts about pro-isomorphic zeta functions and $p$-adic integrals that will be used throughout the paper.

\subsection{Linearization} \label{sec:linearization}
We briefly discuss the correspondence between $\mathcal{T}$-groups and Lie rings mentioned in the introduction.  If $G$ is a $\mathcal{T}$-group of nilpotency class two, so that $[G,G] \leq Z(G)$, then define the Lie ring
\begin{equation} \label{equ:class.two}
\mathcal{L}(G) = G / Z(G) \times Z(G),
\end{equation}
with the natural multiplication $[(g_1 Z(G), z_1),(g_2 Z(G), z_2)] = (Z(G), [g_1, g_2])$.  Then associating a finite-index subgroup $H \leq G$ with the subring $H Z(G) / Z(G) \times (H \cap Z(G))$ gives an inclusion- and index-preserving bijection between finite-index subgroups of $G$ and finite-index subrings of $\mathcal{L}(G)$.  Normal subgroups correspond to ideals of $\mathcal{L}(G)$ under this bijection, whereas pro-isomorphic subgroups are associated with subrings $\mathcal{M} \leq \mathcal{L}(G)$ such that $\mathcal{M} \tensor_\Z \Z_p \simeq \mathcal{L}(G) \tensor_\Z \Z_p$ for all primes $p$.  For $\mathcal{T}$-groups $G$ of arbitrary nilpotency class, the Malcev correspondence gives a Lie ring $\mathcal{L}(G)$ affording a similar bijection between finite-index subgroups of $G$ and finite-index subrings of $\mathcal{L}(G)$, provided that the index is coprime to an effectively computable integer depending only on the Hirsch length of $G$; see~\cite[\S4]{GSS/88} for details.

Given a nilpotent Lie ring $\mathcal{L}$, define the subring and ideal zeta functions $\zeta^\leq_{\mathcal{L}}(s) = \sum_{n = 1}^\infty b^\leq_n(\mathcal{L}) n^{-s}$ and $\zeta^\vartriangleleft_{\mathcal{L}}(s) = \sum_{n = 1}^\infty b^\vartriangleleft_n(\mathcal{L}) n^{-s}$, where $b^\leq_n(\mathcal{L})$ and $b^\vartriangleleft_n(\mathcal{L})$ are the numbers of subrings and ideals of index $n$, respectively.  These zeta functions have Euler decompositions, and for any $\mathcal{T}$-group $G$ the equalities $\zeta^\leq_{G,p}(s) = \zeta^\leq_{\mathcal{L}(G),p}(s)$ and $\zeta^\vartriangleleft_{G,p}(s) = \zeta^\vartriangleleft_{\mathcal{L}(G),p}(s)$ are satisfied for all but finitely many primes $p$, or for all $p$ if $G$ is of class two.

\subsection{Simplification of the $p$-adic integral} \label{sec:assumptions}
A framework for treating $p$-adic integrals as in~\eqref{equ:padic.integral}, under favorable conditions, was described in~\cite[\S2]{duSLubotzky/96}.  Since we will make use of these methods repeatedly, we recall the main ideas here.  

Let $E$ be a number field with ring of integers $\mathcal{O}_E$, and let $\mathfrak{p}$ be a finite place of $E$.  Let $F/E_{\p}$ be a finite extension of the localization $E_{\p}$.  We denote by $\mathcal{O}$ the valuation ring of $F$.  Fix a uniformizer $\pi \in \mathcal{O}$, and let $q$ be the cardinality of the residue field $\mathcal{O} / \pi \mathcal{O}$.  
Let $\mathbf{G} \subseteq \mathbf{GL}_n$ be an affine group scheme over $E$.  Set $\mathbf{G}^\circ$ to be the connected component of the identity, and fix a decomposition $\mathbf{G}^\circ = \mathbf{N} \rtimes \mathbf{H}$, where $\mathbf{N}$ is the unipotent radical of $\mathbf{G}^\circ$ and $\mathbf{H}$ is reductive.  Define the group $G = \mathbf{G}(F)$, the subgroup $\mathbf{G}(\mathcal{O}) = G \cap \GL_n(\mathcal{O})$, and the submonoid $G^+ = G \cap M_n (\mathcal{O})$.  In addition, set $N = \mathbf{N}(F)$ and $H = \mathbf{H}(F)$.  
For any algebraic subgroup $\mathbf{S}$ of $\mathbf{G}$, set $\mathbf{S}(\mathcal{O}) = \mathbf{S}(F) \cap \mathbf{G}(\mathcal{O})$ and let $\mu_{\mathbf{S}(F)}$ denote the right Haar measure on $\mathbf{S}(F)$, normalized so that $\mu_{\mathbf{S}(F)}(\mathbf{S}(\mathcal{O})) = 1$.
We are interested in computing the integral
\begin{equation} \label{equ:padic.integral.discussion}
\mathcal{Z}_{\mathbf{G}, F} (s) = \int_{G^+} | \det g |_{F}^s d \mu_G (g).
\end{equation}

We will now introduce three conditions under which such integrals are treated in~\cite[\S2]{duSLubotzky/96}.  
In view of Proposition~\ref{pro:padic.integral}, for the purposes of this article we are most interested in the case where $\mathcal{L}$ is a Lie lattice and $\mathbf{G} = \aAut (\mathcal{L} \tensor_\Z \Q)$.  
The first two conditions hold for $\mathbf{G}(\Q_p)$ for almost all primes $p$ by general considerations.
A rigidity condition on $\mathcal{L} \tensor_\Z \Q_p$ ensures that they will hold for $\mathbf{G}(F)$, where $F/\Q_p$ is any finite extension.  The third condition, however, is much more restrictive; see Remark~\ref{rmk:almost.all.primes} below.

\begin{asm} \label{first.assumption}
We assume that $G = \mathbf{G}(\mathcal{O}) \mathbf{G}^\circ (F)$.
\end{asm}
It easily follows from Assumption~\ref{first.assumption} that $\mathcal{Z}_{\mathbf{G}, F}(s) = \mathcal{Z}_{\mathbf{G}^\circ, F}(s)$~\cite[Proposition~2.1]{duSLubotzky/96}.  Replace $\mathbf{G}$ with $\mathbf{G}^\circ$, so that $\mathbf{G}$ is a connected algebraic group.  

Furthermore, we assume that the embedding $\mathbf{G} \subseteq \mathbf{GL}_n$ has a particularly convenient form.  The group $\mathrm{GL}_n(F)$ naturally acts from the right on $F^n$; let $(e_1, \dots, e_n)$ be the standard basis.  Given a sequence $0 = d_0 < d_1 < \cdots d_t = n$, define $U_i$ to be the $F$-linear span of $e_{d_{i-1} + 1}, \dots, e_{d_i}$.  Setting $V_i = U_i \oplus \cdots \oplus U_t$ for every $i \in [t]$, note that $V_1 = F^n$ and $V_{t+1} = (0)$.  

\begin{asm} \label{second.assumption}
We assume that, for a suitable sequence as above, the following conditions are satisfied: 
\begin{itemize}
\item The subspace $U_i$ is $H$-stable for every $i \in [t]$.
\item The subspace $V_i$ is $N$-stable for every $i \in [t]$; moreover, $N$ acts trivially on the quotient $V_i / V_{i+1}$.
\end{itemize}
\end{asm}

Set $V = V_1$.  Under our assumptions, the group $G = N \rtimes H$ acts on the quotient $V/V_{i+1}$ for every $i \in [t]$; let $\psi_i^\prime : G \to \Aut(V/V_{i+1})$ be the corresponding homomorphism.  Since $V/V_{i+1}$ is spanned by the images of $e_1, \dots, e_{d_{i}}$, there is a natural identification of $\Aut(V/V_{i+1})$ with $\mathrm{GL}_{d_{i}}(F)$.  Putting $N_{i} = N \cap \ker \psi_i^\prime$, we note that $\psi_i^\prime$ factors through a map $\psi_i : G/N_{i} \to \Aut (V/V_{i+1})$.  

Let $\overline{n} \in N_i/N_{i+1}$ and $j \in [d_i]$.  Then $e_j \overline{n} - e_j \in V_{i+1}$ is well-defined modulo $V_{i+2}$; recall that $G$ acts on $V$ from the right.  Moreover, $H$ acts on $V_{i+1}/V_{i+2}$.  Hence for every $h \in H$ there is a map
\begin{alignat*}{2}
\tau(h) : N_i / N_{i+1} & \hookrightarrow &&(V_{i+1}/V_{i+2})^{d_i} \\
\overline{n} & \mapsto &&((e_1 \overline{n} - e_1)h, \dots, (e_{d_i} \overline{n} - e_{d_i})h).
\end{alignat*}
We identify $(V_{i+1}/V_{i+2})^{d_i}$ with $M_{d_i, \dim U_{i+1}}(F)$ by viewing a $d_i$-tuple $(v_1, \dots, v_{d_i}) \in (V_{i+1}/V_{i+2})^{d_i}$ as the matrix whose $j$-th row is $v_j$, expressed in terms of the basis of $V_{i+1}/V_{i+2}$ given by the images of $e_{d_i + 1}, \dots, e_{d_{i+1}}$.  Define the function $\theta_i^F : H \to \mathbb{R}$ by
\begin{equation} \label{equ:def.theta}
\theta_i^F (h) = \mu_{N_i / N_{i+1}}( \{ \overline{n} \in N_i / N_{i+1} : \tau(h)(\overline{n}) \in M_{d_i, d_{i+1} - d_i}(\mathcal{O}) \} ),
\end{equation}
where $\mu_{N_i / N_{i+1}}$ is the right Haar measure on $N_i / N_{i+1}$, normalized so that the set 
$\psi_{i+1}^{-1}(\psi_{i+1}(N_i / N_{i+1}) \cap M_{d_i}(\mathcal{O}))$ has measure $1$; recall that $\Aut (V / V_{i+1})$ has been identified with $\mathrm{GL}_{d_i}(F)$.  Thus $\mu_{N_i / N_{i+1}}$ is identified with the additive Haar measure on $F^{d_i(d_{i+1}-d_i)}$, normalized on $\mathcal{O}^{d_i(d_{i+1}-d_i)}$.  Note also that $\mu_G = \left( \prod_{i = 1}^{t-1} \mu_{N_i / N_{i+1}} \right) \mu_H$.
Define $H^+ = H \cap M_n(\mathcal{O})$ and $(G/N_{i})^+ = \psi_i^{-1}(\psi_i(G/N_{i}) \cap M_{d_{i}}(\mathcal{O}))$.  

\begin{asm} \label{lifting.condition}
We say that {\emph{the lifting condition holds}} if for every $i \in [2, t-1]$ and every $\overline{g} \in (NH^+ / N_i) \cap (G/N_i)^+$ there exists $\gamma \in G^+$ such that $\overline{g} = \gamma N_i$; note that this always holds for $i = 1$, hence the condition stated here is equivalent to~\cite[Assumption~2.3]{duSLubotzky/96}.
\end{asm}

\begin{pro} \cite[Theorem~2.2]{duSLubotzky/96} \label{pro:duSL}
For $h \in H$ define
$$ \theta^F(h) = \mu_N( \{ n \in N : nh \in G^+ \} ).$$
If Assumptions~\ref{first.assumption}, \ref{second.assumption}, and~\ref{lifting.condition} hold, then 
$$ \theta^F = \prod_{i = 1}^{t-1} \theta_i^F.$$
Furthermore,
$$ \mathcal{Z}_{\mathbf{G}, F}(s) = \int_{H^+} | \det h |_{F}^s \left( \prod_{i = 1}^{t-1} \theta_i^F (h) \right) d \mu_H.$$
\end{pro}

The second claim of Proposition~\ref{pro:duSL} follows from the first by the observation that every $g \in G^+$ decomposes uniquely as a product $g = nh$ with $h \in H^+$ and $n \in N$.  Since $\det g = \det h$, it follows that 
\begin{equation} \label{equ:separate.hn}
 \mathcal{Z}_{\mathbf{G}, F}(s) = \int_{H^+} | \det h |_{F}^s \theta^F(h)  d \mu_H.
 \end{equation}

\subsection{Consequences of the $p$-adic Bruhat decomposition}
An important benefit of rewriting the integral of~\eqref{equ:padic.integral} in the form~\eqref{equ:separate.hn} is that the domain of integration of the latter is partitioned conveniently by the $p$-adic Bruhat decomposition.  Under certain conditions, this may be used to evaluate the integral.  We state here the results that will be needed later in the paper.  The idea is essentially due to Igusa~\cite{Igusa/89} and was developed by du Sautoy and Lubotzky~\cite[\S5]{duSLubotzky/96} and further by the first author~\cite[\S4]{Berman/11}; the reader is invited to consult these references for details.

To streamline the notation, we will replace $\mathbf{H} \times_E F$ with $\mathbf{H}$; thus we treat $\mathbf{H}$ as a reductive group over $F$.  Recall that $\mathbf{H}(\mathcal{O}) = \mathbf{H}(F) \cap \mathbf{GL}_\dimL(\mathcal{O})$, and similarly for algebraic subgroups of $\mathbf{H}$.  Let $\mathbf{T} \subset \mathbf{H}$ be a maximal torus, and suppose that $\mathbf{T}$ splits over $F$; this allows us to fix an isomorphism $\kappa: \mathbf{T} \to \mathbf{G}_a^{\rk \mathbf{H}}$ defined over $F$.  Let $\Phi$ be the root system of $\mathbf{H}$, and let $\Delta = \{ \alpha_1, \dots, \alpha_\ell \}$ be a set of simple roots.  Denote by $\Phi^+$ the consequent set of positive roots.  Let $\Xi = \mathrm{Hom}(\mathbf{G}_m, \mathbf{T})$ be the set of cocharacters of $\mathbf{T}$.  Recall the natural pairing between characters and cocharacters: if $\beta \in \mathrm{Hom}(\mathbf{T}, \mathbf{G}_m)$ and $\xi \in \Xi$, then $\langle \beta, \xi \rangle$ is the integer satisfying $\beta(\xi(t)) = t^{\langle \beta, \xi \rangle}$ for all $t \in \mathbf{G}_m(F)$.  For every root $\alpha \in \Phi$, let $\mathbf{U}_\alpha \subset \mathbf{H}$ be the corresponding root subgroup, and let $\psi_\alpha : \mathbf{G}_a \to \mathbf{U}_\alpha$ be an isomorphism; note that the isomorphisms $\kappa$ and $\psi_\alpha$ may all be chosen to be defined over $\mathcal{O}$.  It is crucial to assume that these data have ``very good reduction'' in the sense of~\cite[II.2]{Igusa/89}, namely that $\mathbf{G}$, $\mathbf{T}$, and all the isomorphisms $\kappa$ and $\psi_\alpha$ have good reduction modulo $\pi$; the latter condition means that reduction modulo $\pi$ induces isomorphisms $\overline{\kappa} : \mathbf{T}(k) \to k^{\rk \mathbf{H}}$ and $\overline{\psi_\alpha} : k \to \mathbf{U}_\alpha(k)$, where $k = \mathcal{O} / (\pi)$ is the residue field.  

Let $W = N_{\mathbf{H}(F)}(\mathbf{T}(F))/\mathbf{T}(F)$ be the Weyl group; it acts by conjugation on the collection of root subgroups and hence on $\Phi$.
Define $\Xi^+ = \{ \xi \in \Xi : \xi(\pi) \in \mathbf{H}(\mathcal{O}) \}$, and for every $w \in W$ and $\alpha \in \Delta$ set
\begin{equation*}
\delta_{w,\alpha} = \begin{cases}
1 &: \alpha \in \Delta \cap w(\Phi^-) \\
0 &: \alpha \in \Delta \setminus w(\Phi^-).
\end{cases}
\end{equation*}
Here $\Phi^- = \Phi \setminus \Phi^+$ is the set of negative roots.  
Partitioning the domain of integration according to the $p$-adic Bruhat decomposition and analyzing the behavior of the integrand on each piece, the first author established the following result, which is immediate from the proof of~\cite[Proposition~4.2]{Berman/11} and generalizes~\cite[(5.4)]{duSLubotzky/96}.

\begin{pro} \label{pro:bruhat.decomposition}
Suppose that the maximal torus $\mathbf{T} \subset \mathbf{H}$ is $F$-split and that very good reduction holds.   
Then
\begin{equation*}
\int_{\mathbf{H}^+(F)} | \det h |_F^s \theta^F (h) d \mu_H (h) = \sum_{w \in W} q^{- \ell (w)} \sum_{\xi \in w \Xi_w^+} q^{\langle \prod_{\beta \in \Phi^+} \beta, \xi \rangle} | \det \xi(\pi) |_F^s \theta^F (\xi(\pi)),
\end{equation*}
where $\ell$ is the length function on $W$ with respect to the Coxeter generating set corresponding to $\Delta$ and 
\begin{equation} \label{equ:wxi.dfn}
w \Xi_w^+ = \left\{ \xi \in \Xi^+ : 
\alpha(\xi(\pi)) \in \pi^{\delta_{w,\alpha}} \mathcal{O} \,\, \text{for all} \, \, \alpha \in \Delta \right\}.
\end{equation}
\end{pro}

\begin{exm} \label{exm:abelian}
Let $\mathcal{L}$ be the abelian Lie lattice $\Z^{\dimL}$ of rank $\dimL$.  It is clear that all finite-index sublattices of $\mathcal{L}$ are ideals isomorphic to $\mathcal{L}$, so the pro-isomorphic zeta function of $\mathcal{L}$ coincides with the subring and ideal zeta functions.  Moreover, the algebraic automorphism group of $L = \mathcal{L} \tensor_\Z \Q = \Q^{\dimL}$ is $\mathbf{GL}_{\dimL}$.  If $F / \Q_p$ is a finite extension as above, and $\zeta_{\mathcal{O}^\dimL}(s) = \sum_{\mathcal{M} \leq \mathcal{O}^\dimL} [\mathcal{O}^\dimL : \mathcal{M}]^{-s}$ counts $\mathcal{O}$-sublattices of finite index, then
\begin{equation*} 
\zeta_{\mathcal{O}^\dimL}(s) = \int_{\mathrm{GL}_{\dimL}^+(F)} | \det A |^{s}_{F} d \mu_{\mathrm{GL}_n(F)} (A) = \prod_{i = 0}^{\dimL - 1} \frac{1}{1 - q^{i - s}}.
\end{equation*}
In particular, $\zeta_{\mathcal{L}}(s) = \prod_{i = 0}^{\dimL - 1} \zeta(s - i)$, where $\zeta(s)$ is the Riemann zeta function.
The first equality above follows from the same reasoning as Proposition~\ref{pro:padic.integral}, while
it is a simple exercise to derive the second equality from Proposition~\ref{pro:bruhat.decomposition}.  See also~\cite[Proposition~1.1]{GSS/88}, \cite[Theorem~15.1]{LubotzkySegal/03}, and~\cite[Example~2.19]{Voll/Newcomer} for an assortment of proofs not relying on Proposition~\ref{pro:bruhat.decomposition}; most of these are stated for $F = \Q_p$ but generalize readily.
\end{exm}

Two applications of Proposition~\ref{pro:bruhat.decomposition} appear below, in Sections~\ref{sec:higher.heisenberg} and~\ref{sec:lmn}.  In both cases the hypothesis of very good reduction is readily verified, and we will not address it explicitly.

\section{Rigidity} \label{sec:rigidity}
\subsection{A rigidity criterion}
Let $L$ be a finite-dimensional Lie algebra over a field $k$, and let $R$ be a finite-dimensional semisimple commutative $k$-algebra.  As in~\cite{Segal/89}, for any subset $X \subseteq L$ and ideal $I \leq L$, we set
$$C_{L/I}(X) = \{ x \in L : [x,X] \subseteq I \},$$
where $[x,X] = \{ [x,y] : y \in X \}$.
Write $C_L(X)$ for the centralizer $C_{L/(0)}(X)$.
Note that $I + kx \subseteq C_{L/[I,L]}(C_{L/[I,L]}(x))$ for any $x \in L$ and any ideal $I \leq L$.
Set $I_1 = [I,L]$ and
\begin{equation} \label{equ:yz.def} 
\mathcal{Y}(I) = \{ x \in L \setminus I : C_{L / I_1} (C_{L / I_1} (x)) = I + kx \}.
\end{equation}
For instance, if $I = Z(L)$ is the center of $L$, then $I_1 = 0$ and $x \in \mathcal{Y}(I)$ if and only if $x$ is not central and the set of elements of $L$ commuting with the centralizer of $x$ is as small as possible, namely $Z(L) + kx$.  All tensor products in this section are over $k$.  

\begin{dfn}
We say that an ideal $I \leq L$ is {\emph{highly invariant}} if, for every semisimple commutative $k$-algebra $R$ and every $k$-linear automorphism $\varphi \in \mathrm{Aut}_k (L \tensor R)$, we have $(I \tensor R)\varphi = I \tensor R$.
\end{dfn}

Observe that the center $Z(L)$ is a highly invariant ideal, as is any verbal ideal.  Recall that an ideal $I \leq L$ is called verbal if there exists a variety $\mathcal{V}$ of Lie algebras over $k$ (cf.~\cite[\S2]{Vaughan-Lee/70}) such that $I$ is the (unique) minimal element of the set $\{ J \leq L : L / J \in \mathcal{V} \}$.  For instance, the ideals $\gamma_i L$ appearing in the lower central series of $L$ are verbal.

The following easy statement is noted after~\cite[Lemma~1]{Segal/89}.
\begin{rem} \label{rmk:segal.remark}
Let $E/k$ be a field extension of arbitrary dimension.  Then, for any ideal $I \leq L$ and any $X \subseteq L$, we have
$ C_{L \tensor E  / I \tensor E}(X \tensor 1) = C_{L/I}(X) \tensor E$ since centralizers are defined by linear conditions over $E$.  Hence the same conclusion holds for semisimple $k$-algebras $R$ as above.
\end{rem}

The following lemma somewhat weakens the hypotheses of~\cite[Lemma~1]{Segal/89}.

\begin{lemma} \label{lem:rigidity1}
Let $Z \leq \gamma_2 L$ be a highly invariant ideal of $L$.  Set $Z_1 = [Z,L]$.
Suppose that $L/Z_1$ is indecomposable, that $\dim_{k} L / Z_1 > 1$, and that $\mathcal{Y}(Z)$ generates $L$ as a Lie algebra.  Let $R$ be any finite-dimensional semisimple commutative $k$-algebra.  Then there exists a unique epimorphism 
$$ \widehat{} : \mathrm{Aut}_k (L \otimes R) \to \mathrm{Aut}_k R$$
such that
$$  (\alpha w)\varphi \equiv (\alpha)\widehat{\varphi} \cdot (w)\varphi   \, \mathrm{mod} \, Z \otimes R$$
for any $\varphi \in \mathrm{Aut}_k (L \otimes R)$ and any elements $\alpha \in R$ and $w \in L \otimes R$.  Moreover, the natural injection $\mathrm{Aut}_k R \to \mathrm{Aut}_k (L \otimes R)$ splits \, $\widehat{}$.
\end{lemma}
\begin{proof}
Fix an automorphism $\varphi \in  \mathrm{Aut}_k (L \otimes R)$, and consider an element $x \in \mathcal{Y}(Z)$.  For any subset $X \subseteq L$, write $X \tensor 1 = \{ x \tensor 1: x \in X \} \subseteq L \tensor R$.  Then
$C_{(L \otimes R)/(I \otimes R)} (X \tensor R) = C_{(L \otimes R)/(I \otimes R)} (X \tensor 1) = C_{L/I}(X) \otimes R$
for any ideal $I \leq L$ by Remark~\ref{rmk:segal.remark}.  
In particular,
\begin{multline*} 
C_{(L \otimes R)/(Z_1 \otimes R)} (C_{(L \otimes R)/(Z_1 \otimes R)} (x \otimes 1)) =  C_{(L \otimes R)/(Z_1 \otimes R)} (C_{L / Z_1} (x) \tensor R) = \\  C_{L / Z_1} (C_{L / Z_1} (x)) \otimes R = (kx + Z) \otimes R = (x \otimes 1)R + Z \otimes R,
\end{multline*}
since we assumed $x \in \mathcal{Y}(Z)$.  Hence, $x \tensor 1 \in \mathcal{Y}(Z \tensor R)$.  
By assumption $Z$ is highly invariant.  Thus $Z \otimes R$ is stable under $\varphi$, and so is $\mathcal{Y}(Z \tensor R)$.  Therefore,
$$ C_{(L \otimes R)/(Z_1 \otimes R)} (C_{(L \otimes R)/(Z_1 \otimes R)} ((x \otimes 1)\varphi)) = ((x \otimes 1)\varphi)R + Z \otimes R.$$
Let $\alpha \in R$.  Since $(x \otimes \alpha)\varphi \in C_{(L \otimes R)/(Z_1 \otimes R)} (C_{(L \otimes R)/(Z_1 \otimes R)} ((x \otimes 1)\varphi))$, there is a uniquely defined element $(\alpha)\beta_x \in R$ such that $(x \otimes \alpha)\varphi \equiv (\alpha)\beta_x \cdot (x \otimes 1)\varphi \, \mathrm{mod} \, Z \otimes R$.  
The map $\beta_x : R \to R$ is clearly $k$-linear and bijective.

Since $\mathcal{Y}(Z)$ generates $L$, we may fix a subset $\mathcal{S} \subset \mathcal{Y}(Z) \cap (L \setminus \gamma_2 L)$ such that $\mathcal{S}$ generates $L$.  If $x \in \mathcal{S}$, then there exists $y \in \mathcal{S}$ such that $[x,y] \not\in Z_1$; otherwise, $kx + Z_1$ would be a proper direct summand of $L/Z_1$, contradicting our hypotheses.  For any $\alpha \in R$ we have
\begin{multline} \label{equ:commutator}
(\alpha)\beta_x \cdot ([x,y] \otimes 1)\varphi = [(\alpha)\beta_x \cdot (x \tensor 1)\varphi, (y \tensor 1)\varphi]  \equiv ([x \otimes \alpha, y \otimes 1])\varphi = \\ ([x \otimes 1, y \otimes \alpha])\varphi \equiv [(x \tensor 1)\varphi, (\alpha)\beta_y \cdot (y \tensor 1)\varphi] = (\alpha) \beta_y \cdot ([x,y] \otimes 1) \varphi \, \mathrm{mod} \, Z_1 \otimes R.
\end{multline}
Observe that the image of $[x,y] \tensor R$ in $(L \tensor R)/(Z_1 \tensor R)$ is a $(\dim_k R)$-dimensional $k$-subspace.  The same is true of the image of $([x,y] \tensor R)\varphi$.  Since this image lies inside the image of $(([x,y] \tensor 1)\varphi)R$, which is at most $(\dim_k R)$-dimensional, it follows that $\mathrm{Ann}_R (([x,y] \tensor 1)\varphi) = 0$.  Hence $(\alpha)\beta_x = (\alpha)\beta_y$ for all $\alpha \in R$.
Further, we claim that $\beta_x \in \mathrm{Aut}_k R$.  To establish this, it remains only to show that $\beta_x$ is multiplicative.  Indeed, for $y$ as above and $\alpha_1, \alpha_2 \in R$, we see that
\begin{multline*} 
(\alpha_1 \alpha_2)\beta_x \cdot ([x,y] \otimes 1)\varphi \equiv ([x,y] \otimes \alpha_1 \alpha_2)\varphi = ([x \otimes \alpha_1, y \otimes \alpha_2])\varphi \equiv \\ [(\alpha_1)\beta_x \cdot (x \otimes 1)\varphi , (\alpha_2)\beta_x \cdot (y \otimes 1)\varphi] \equiv  (\alpha_1)\beta_x (\alpha_2)\beta_x \cdot  ([x,y] \otimes 1) \varphi \, \mathrm{mod} \, Z_1 \otimes R,
\end{multline*}
and hence $(\alpha_1 \alpha_2) \beta_x =  (\alpha_1)\beta_x  (\alpha_2)\beta_x$ as above.  

For any $\beta \in \mathrm{Aut}_k R$, set
$$ \mathcal{S}_\beta = \{ x \in \mathcal{S} : (x \otimes \alpha)\varphi \equiv (\alpha)\beta \cdot (x \otimes 1)\varphi \, \mathrm{mod} \, Z \otimes R \, \text{for all} \, \alpha \in R \},$$
and define $L_\beta$ to be the subalgebra of $L$ generated by $\mathcal{S}_\beta$.  It is evident from the above that any $x \in L_\beta$ satisfies $(x \otimes \alpha)\varphi \equiv (\alpha)\beta \cdot (x \otimes 1)\varphi \, \mathrm{mod} \, Z \otimes R$ for all $\alpha \in R$, and that the same congruence holds modulo $Z_1 \tensor R$ if $x \in \gamma_2 L_\beta$.  Moreover, if $\beta \neq \beta^\prime$, then $[L_\beta, L_{\beta^\prime}] \subseteq Z_1$ by~\eqref{equ:commutator}.  Writing $\overline{L}_\beta$ for the image of $L_\beta$ in $L / Z_1$, we find that 
$$L / Z_1 = \bigoplus_{\beta \in \mathrm{Aut}_k R} \overline{L}_\beta$$
as $k$-Lie algebras since $Z \leq \gamma_2 L$.  As $L /  Z_1$ is indecomposable by assumption, there must be a single $\beta \in \mathrm{Aut}_k R$ such that $L_\beta = L$.  Now set $\widehat{\varphi} = \beta$.  The resulting map \, $\widehat{} : \mathrm{Aut}_k (L \otimes R) \to \mathrm{Aut}_k R$ clearly has all the claimed properties.
\end{proof}

The argument deducing~\cite[Theorem~2]{Segal/89} from~\cite[Lemma~1]{Segal/89} transfers to our setting and will be used to deduce Theorem~\ref{thm:rigidity} below from Lemma~\ref{lem:rigidity1}.  We emphasize that the ideas are due to D.~Segal.  Fix a finite-dimensional semisimple commutative $k$-algebra $R$ as above, and consider $L$ as a subalgebra of $L \otimes R$ via the natural embedding $x \mapsto x \otimes 1$.

\begin{dfn} \label{def:varphi}
Let $\varphi \in \mathrm{Aut}_k (L \otimes R)$.  We denote by $\widetilde{\varphi}$ the unique $R$-linear map $\widetilde{\varphi} : L \otimes R \to L \otimes R$ satisfying $\widetilde{\varphi} |_L = \varphi |_L$.
\end{dfn}

\begin{rem}
It is clear that $\widetilde{\varphi}$ is a Lie algebra endomorphism of $L \otimes R$.  In fact, it is not hard to show that $\widetilde{\varphi} \in \overline{\aAut L}(R)$, where $\overline{\aAut L}$ is the Zariski closure of $\aAut L$ in the algebraic endomorphism monoid $\aEnd L$; by~\cite[Lemma 1.2]{Putcha/82}, $\overline{\aAut L}$ is an algebraic submonoid.  However, $\widetilde{\varphi}$ need not be an automorphism of $L \tensor R$.  For instance, consider a two-dimensional abelian Lie $k$-algebra $L$ with basis $(x_1, x_2)$, and let $R/k$ be a finite extension of fields.  Viewed as a $k$-algebra, $L \tensor R$ is an abelian Lie algebra of dimension $2d$, where $d = [R : k]$.  Thus $\Aut_{k} (L \tensor R) \simeq \GL_{2d} (k)$.  Let $\alpha \in R \setminus k$.  Since $x_1 \tensor 1$ and $x_1 \tensor \alpha$ are linearly independent over $k$, there exists $\varphi \in \Aut_{k} (L \tensor R)$ such that
\begin{eqnarray*}
 (x_1 \tensor 1)\varphi & = & x_1 \tensor 1 \\
 (x_2 \tensor 1)\varphi & = & x_1 \tensor \alpha.
\end{eqnarray*}
However, $(x_1 \tensor 1)\varphi$ and $(x_2 \tensor 1)\varphi$ are not linearly independent over $R$, so $\widetilde{\varphi}$ is not an automorphism of $L \tensor R$.  
\end{rem}

\begin{lemma} \label{lem:rigidity2}
Suppose that $L$ is a finite-dimensional nilpotent $k$-Lie algebra and that $Z \leq \gamma_2 L$ is a highly invariant ideal satisfying the hypotheses of Lemma~\ref{lem:rigidity1}.  Let $\varphi \in \mathrm{Aut}_k (L \otimes R)$.  Then $\widetilde{\varphi} \in \mathrm{Aut}_R(L \otimes R)$, where $\widetilde{\varphi}$ is as in Definition~\ref{def:varphi}.
\end{lemma}
\begin{proof}
Since $\widetilde{\varphi}$ is $R$-linear by construction, we need only show that $\widetilde{\varphi}$ is an automorphism.  Set $\psi = \varphi \circ \widehat{\varphi}^{-1} \in \mathrm{Aut}_k (L \otimes R)$, where $\widehat{\varphi}$ is as in Lemma~\ref{lem:rigidity1} and $\mathrm{Aut}_k R$ is embedded into $\mathrm{Aut}_k (L \otimes R)$ in the natural way.  For any $v = \sum_i x_i \otimes \alpha_i \in L \otimes R$, we have
\begin{equation} \label{equ:semidirect.product}
\widetilde{\varphi}(v) =  \sum_i \varphi(x_i) \otimes \alpha_i \equiv \psi(v) \, \mathrm{mod} \, Z \otimes R
\end{equation}
by the definition of $\widehat{\varphi}$.  As $\psi$ is surjective, this implies $\mathrm{Im} (\widetilde{\varphi}) + Z \otimes R = L \otimes R$.  However, $Z \otimes R \subseteq \gamma_2 (L \otimes R)$ by assumption.  Thus $\mathrm{Im}(\widetilde{\varphi})$ generates $L \otimes R$, since $L \tensor R$ is nilpotent.  Hence $\widetilde{\varphi}$ is surjective.  Now $\dim_k L \otimes R < \infty$, so $\widetilde{\varphi}$ is also injective.
\end{proof}

\begin{cor} \label{cor:semidirect.product}
Let $L$ be a finite-dimensional nilpotent $k$-Lie algebra and let $Z \leq \gamma_2 L$ be a highly invariant ideal such that $L/Z_1$ is indecomposable and $\mathcal{Y}(Z)$ generates $L$.  Let 
$$J = \ker (\mathrm{Aut}_k (L \otimes R) \to \mathrm{Aut}_k ((L \otimes R)/(Z \otimes R))) $$
be the subgroup of automorphisms that are trivial modulo $Z \otimes R$.  Then
$$ \mathrm{Aut}_k (L \otimes R) = ( J \cdot \mathrm{Aut}_R (L \otimes R)) \rtimes \mathrm{Aut}_k R.$$
\end{cor}
\begin{proof}
Let $\varphi \in \Aut_k (L \tensor R)$.  Since $L$ is nilpotent, our hypotheses imply $\dim_k L/Z_1 > 1$.  Thus the hypotheses of Lemma~\ref{lem:rigidity1} hold, and we set $\psi = \varphi \circ \widehat{\varphi}^{-1} $ as in the proof of Lemma~\ref{lem:rigidity2}.   Then $\widetilde{\varphi}^{-1} \circ \psi \in J$ by~\eqref{equ:semidirect.product}.  But 
$\varphi = \widetilde{\varphi} \circ (\widetilde{\varphi}^{-1} \circ \psi) \circ \widehat{\varphi}$, 
and so $  \mathrm{Aut}_k (L \otimes R) = \mathrm{Aut}_R (L \otimes R) \cdot J \cdot \mathrm{Aut}_k R$.  The splitting noted in the statement of Lemma~\ref{lem:rigidity1} completes the proof.
\end{proof}

\begin{dfn} \label{def:rigidity}
Let $L$ be a Lie algebra over a field $k$ and $Z \leq L$ a highly invariant ideal.
\begin{enumerate}
\item The Lie algebra $L$ is said to be {\emph{absolutely indecomposable}} if $L \tensor_k E$ is indecomposable as an $E$-algebra for every field extension $E/k$.
\item The Lie algebra $L$ is said to be {\emph{$Z$-rigid}} if, for any finite separable extension $K/k$, the following equality of algebraic groups over $k$ holds:
\begin{equation*}
\aAut (L \tensor_k K) = (\mathfrak{J} \cdot \mathrm{Res}_{K/k} (\aAut L)) \rtimes \aAut K,
\end{equation*}
where $\mathfrak{J} = \ker (\aAut (L \tensor_k K) \to \aAut (L \tensor_k K)/(Z \tensor_k K))$.
\end{enumerate}
\end{dfn}

\begin{thm} \label{thm:rigidity}
Let $L$ be a finite-dimensional nilpotent $k$-Lie algebra, and let ${Z} \leq \gamma_2 {L}$ be a highly invariant ideal such that $\mathcal{Y}(Z)$ generates $L$, where $\mathcal{Y}(Z)$ is the set defined in~\eqref{equ:yz.def}.  Suppose that ${L} / Z_1$ is {absolutely indecomposable}, where $Z_1 = [Z,L]$.  Then $L$ is $Z$-rigid.
\end{thm}
\begin{proof}
Let $K / k$ be a finite separable extension, and let $E/k$ be any extension of fields.  Consider the $E$-Lie algebra ${L}_E = L  \tensor E$ and its highly invariant ideal ${Z}_E = Z \tensor E \leq \gamma_2 (L_E)$.  Then $\mathcal{Y}(Z) \tensor E \subseteq \mathcal{Y}(Z_E)$ by Remark~\ref{rmk:segal.remark}, so $\mathcal{Y}({Z}_E)$ generates ${L}_E$.  Finally, ${L}_E/[{Z}_E,{L}_E] = (L / Z_1) \tensor E$ is indecomposable by assumption.  Thus Corollary~\ref{cor:semidirect.product}, applied to ${L}_E$, ${Z}_E$, and the $E$-algebra $R = K \tensor_k E$, tells us that $ (\aAut (L \tensor_k K))(E) =  \Aut_E ((L \tensor_k K) \tensor_k E) = \Aut_E ({L}_E \tensor_E R)$ is equal to
\begin{multline*}
 (J_E \cdot \Aut_R(L_E \tensor_E R)) \rtimes \Aut_E R = \mathfrak{J}(E) \cdot (\aAut L)(R) \rtimes \Aut_E (K \tensor_k E) = \\
 (\mathfrak{J} \cdot \mathrm{Res}_{K/k} (\aAut L))(E) \rtimes (\aAut K)(E),
\end{multline*}
where $J_E = \ker (\Aut_E (L_E \tensor_E R) \to \Aut_E (L_E \tensor_E R)/(Z_E \tensor_E R))$.  Just as in the proof of~\cite[Theorem~2]{Segal/89}, one shows that $\mathfrak{J} \cdot \mathrm{Res}_{K/k}(\aAut L) \cdot \aAut K$ is a semidirect product of algebraic groups.  The equality $\aAut (L \tensor_k K) = (\mathfrak{J} \cdot \mathrm{Res}_{K/k} (\aAut L)) \rtimes \aAut K$ of algebraic groups over $k$ then follows from the equality on points over the separable closure of $k$; cf.~\cite[Corollary~1.30]{Milne/17}.
\end{proof}

The following corollary is essentially~\cite[Theorem~2]{Segal/89}.

\begin{cor}[Segal] \label{cor:segal}
Let $M$ and $Z$ be highly invariant ideals of a finite-dimensional nilpotent $k$-Lie algebra $L$ such that $Z \leq M \leq \gamma_2 L$ and $\dim_k L/M > 1$.  Setting $M_1 = [M,L]$ and $Z_1 = [Z,L]$, define
\begin{eqnarray*}
\mathcal{X}(M) & = & \{ x \in L \setminus M : C_{L / M_1}(x) = M + kx \} \\
\mathcal{Y}(M,Z) & = & \{ x \in L \setminus M : C_{L / Z_1}(C_{L/Z_1}(x)) = Z + kx \}.
\end{eqnarray*}
Assume that $\mathcal{X}(M)$ and $\mathcal{Y}(M,Z)$ each generate $L$.  Then $L$ is $Z$-rigid.
\end{cor}
\begin{proof}
We verify the hypotheses of Theorem~\ref{thm:rigidity}.  Clearly, $\mathcal{Y}(M,Z) \subseteq \mathcal{Y}(Z,Z) = \mathcal{Y}(Z)$, and hence $\mathcal{Y}(Z)$ generates $L$.  It remains to verify that $L / Z_1$ is absolutely indecomposable.  Suppose not.  Then for some field extension $E/k$ there exist proper subalgebras $L_1, L_2 \leq L_E$ containing $(Z_1)_E$ such that $L_1 / (Z_1)_E \oplus L_2 / (Z_1)_E = L_E / (Z_1)_E$.  In particular, $[L_1, L_2] \subseteq (Z_1)_E$.  Since $M_E/(Z_1)_E$ lies in the derived subalgebra of $L_E/(Z_1)_E$, it cannot contain either of the proper direct summands $L_1 / (Z_1)_E$ and $L_2 / (Z_1)_E$.  

Let $x \in L \setminus M$.  If $x \in \mathcal{X}(M)$, then Remark~\ref{rmk:segal.remark} implies that $C_{L_E / (M_1)_E}(x \tensor 1) = M_E + E(x \tensor 1)$.  There exist $x_1 \in L_1$ and $x_2 \in L_2$ such that $x \tensor 1 = x_1 + x_2 \, \mathrm{mod} \, (Z_1)_E$.  
If $x_1 \in L_1 \cap M_E$, then clearly $L_1 \subseteq C_{L_E/(M_1)_E} (x \tensor 1)$.  However, $L_1 \not\subseteq M_E + E(x \tensor 1) = M + Ex_2$ and hence $x \not\in \mathcal{X}(M)$.  The case $x_2 \in L_2 \cap M_E$ is treated analogously.  If neither $x_1$ nor $x_2$ lie in $M_E$, then $a_1 x_1 + a_2 x_2 \not\in M_E$ for all $a_1, a_2 \in E$; indeed, $M_E$ is a verbal ideal and is thus preserved by all endomorphisms of $L_E$, in particular by the projections to the components $L_1$ and $L_2$.  The $E$-linear span of $x_1$ and $x_2$ is contained in $C_{L_E/(M_1)_E} (x \tensor 1)$, so $\dim_E C_{L_E/(M_1)_E}(x \tensor 1) \geq \dim_E M_E + 2$, and again $x \not\in \mathcal{X}(M)$.  Therefore $\mathcal{X}(M) = \varnothing$, contradicting the assumption that $\mathcal{X}(M)$ generates $L$.  Thus $L$ is $Z$-rigid by Theorem~\ref{thm:rigidity}.
\end{proof}

\begin{rem} \label{rem:segal}
In some of the examples considered in this paper, the hypotheses of Corollary~\ref{cor:segal} do not apply, whereas those of Theorem~\ref{thm:rigidity} do; see Sections~\ref{sec:higher.heisenberg} and~\ref{sec:lmn}.  However, if $M = Z$, then observe that $\mathcal{X}(M) \subseteq \mathcal{Y}(M,Z) = \mathcal{Y}(Z)$.  Hence, if $\mathcal{X}(Z)$ generates $L$, then $\mathcal{Y}(Z)$ does as well.  It is often easier to verify that $\mathcal{X}(Z)$ generates $L$, when this holds, than to verify the hypotheses of Theorem~\ref{thm:rigidity} directly.
\end{rem}

\begin{rem} \label{rem:extensions}
Let $L$ be a $k$-Lie algebra, and let $Z \leq L$ be a highly invariant ideal.  It is clear from Remark~\ref{rmk:segal.remark} that if $L$ satisfies the criterion for $Z$-rigidity of either Theorem~\ref{thm:rigidity} or Corollary~\ref{cor:segal}, then the same is true of the $K$-algebra $L \tensor_k K$ for any extension $K/k$ of fields.
\end{rem}

\subsection{Fixed notations and definitions} \label{sec:consequences}
Let $\mathcal{L}$ be a nilpotent Lie lattice over $\Z$ such that $Z(\mathcal{L}) \leq \gamma_2 \mathcal{L}$; thus $\mathcal{L}$ has no abelian direct summands.  Let $L = \mathcal{L} \tensor_Z \Q$ be the associated $\Q$-Lie algebra.  Fix a prime $p$ and set $\mathcal{L}_p = \mathcal{L} \tensor_\Z  \Z_p$ and $L_p = L \tensor_\Q \Q_p$.

The following notation will be used for the rest of the paper.  Put $\dimL = \dim_{\Q} L$ and $\overdim = \dim_\Q L / Z(L)$.  
Consider a number field $K$ of degree $d = [K:\Q]$ and the semisimple $\Q_p$-algebra $R = K \tensor_{\Q} \Q_p$.  If $p \mathcal{O}_K = \mathfrak{p}_1^{e_1} \cdots \mathfrak{p}_\npr^{e_\npr}$, where the $\mathfrak{p}_i$ are distinct prime ideals of $\mathcal{O}_K$ of inertia degree $f_i = [\mathcal{O}_K / \mathfrak{p}_i : \mathbb{F}_p]$, then $R = K_{\p_1} \times \cdots \times K_{\p_\npr}$, where $K_{\p_i} / \Q_p$ is a field extension of ramification index $e_i$ and inertia degree $f_i$.  Note that $d = \sum_{i = 1}^\npr e_i f_i$.  Write $\mathcal{O}_{\p_i}$ for the ring of integers of $K_{\p_i}$ and $q_i = p^{f_i}$ for the cardinality of its residue field.  
For each $i \in [\npr]$, let $\boldsymbol{\alpha}_i = (\alpha_{i,1}, \dots, \alpha_{i, e_i f_i})$ be a $\Z_p$-basis of $\mathcal{O}_{\p_i}$, and let $\boldsymbol{\alpha}$ be the concatenation of $\boldsymbol{\alpha}_1, \dots, \boldsymbol{\alpha}_\npr$.  Denote the elements of $\boldsymbol{\alpha}$ by $\alpha_1, \dots, \alpha_d$.

For every $\p | p$ there is a natural injection $\iota_\p: K_{\p} \hookrightarrow \mathrm{End}_{\Q_p} K_{\p}$ of rings sending $\beta \in K_{\p}$ to the multiplication-by-$\beta$ map $\alpha \mapsto \alpha \beta$.  Since we have fixed a $\Q_p$-basis $\boldsymbol{\alpha}_i$ of $K_{\p_i}$, this induces an injection $\iota_{\p_i} : K_{\p_i} \hookrightarrow M_{e_i f_i}(\Q_p)$.  Moreover, since $\boldsymbol{\alpha}_i$ is an integral basis, given $\beta \in K_{\p_i}$ we see that $\beta \in \mathcal{O}_{\p_i}$ if and only if $\iota_{\p_i}(\beta) \in M_{e_i f_i}(\Z_p)$.  Observe that $\det \iota_{\p_i}(\beta) = N_{K_{\p_i}/\Q_p}(\beta)$ for all $\beta \in K_{\p_i}$.  For every $m \in \N$ we get a ring monomorphism $M_m(K_{\p_i}) \hookrightarrow M_{me_if_i}(\Q_p)$; the image of $A \in M_m(K_{\p_i})$ is the matrix obtained by replacing each matrix element $a_{jk}$ of $A$ by the corresponding $\iota_{\p_i} (a_{jk})$.  Slightly abusing notation, we also denote this map by $\iota_{\p_i}$.  An exercise in linear algebra shows that $| \det \iota_\p (A) |_{\Q_p} = | N_{K_{\p} / \Q_p} (\det A) |_{\Q_p} = | \det A |_{{K_\p}}$ for any $\p | p$ and any $A \in M_m(K_{\p})$; cf.~\cite[Theorem~1]{KSW/99}.  The basis $\boldsymbol{\alpha}_i$ determines an embedding of algebraic groups $\mathrm{Res}_{K_{\p_i}/ \Q_p} \mathbf{GL}_m \subset \mathbf{GL}_{me_if_i}$ over $\Q_p$; the corresponding map on $\Q_p$-points coincides with $\iota_{\p_i}$.

Let $\mathbf{G} = \aAut {L}_p$ denote the algebraic automorphism group of ${L}_p$ and define $\mathbf{J} = \ker(\mathbf{G} \to \aAut (L_p/Z(L_p)))$.  For any finite extension $F / \Q_p$, write $\mathbf{G}^+(F)$ for the submonoid of $\mathbf{G}(F) = \Aut_F (L \tensor_\Q F)$ consisting of elements that map the $\mathcal{O}_F$-lattice $\mathcal{L} \tensor_\Z \mathcal{O}_F$ into itself and $\mathbf{G}(\mathcal{O}_F)$ for the subgroup of elements inducing an automorphism of $\mathcal{L} \tensor_\Z \mathcal{O}_F$.  As in Assumption~\ref{first.assumption}, assume that $\mathbf{G}(\Q_p) = \mathbf{G}(\Z_p) \mathbf{G}^\circ (\Q_p)$; by~\cite[Lemma~4.1]{duSLubotzky/96} this holds for almost all primes $p$.  Replacing $\mathbf{G}$ by its connected component does not affect the integral of~\eqref{equ:padic.integral.discussion} by~\cite[Proposition~2.1]{duSLubotzky/96}, so we may suppose that $\mathbf{G}$ is connected.  Then $\mathbf{G} = \mathbf{N} \rtimes \mathbf{H}$, where $\mathbf{N}$ is the unipotent radical of $\mathbf{G}$ and $\mathbf{H}$ is connected and reductive.
For every algebraic subgroup $\mathbf{S} \leq \mathbf{G}$ set $\mathbf{S}^+(F) = \mathbf{S}(F) \cap \mathbf{G}^+(F)$ and $\mathbf{S}(\mathcal{O}_F) = \mathbf{S}(F) \cap \mathbf{G}(\mathcal{O}_F)$.  Let $\mu_{\mathbf{S}(F)}$ denote the right Haar measure on $\mathbf{S}(F)$, normalized so that $\mu_{\mathbf{S}(F)}(\mathbf{S}(\mathcal{O}_F)) = 1$.

Since $\mathbf{H}$ is reductive, there exists a subspace $U \subset L_p$ such that $L_p = U \oplus Z(L_p)$ and such that $U \tensor_{\Q_p} F$ is stable under the action of $\mathbf{H}(F)$ for any field extension $F/\Q_p$; cf.~\cite[Theorem~22.138]{Milne/17}.  Moreover, since we assumed $Z(\mathcal{L}) \leq \gamma_2 \mathcal{L}$, there is a subspace $U^\prime \subset U$, possibly trivial, such that $U^\prime \oplus Z(L_p) = \gamma_2 L_p$.  
We may fix a $\Z_p$-basis $b_1, \dots, b_\dimL$ of $\mathcal{L}_p$ that, viewed as a $\Q_p$-basis of $L_p$, consists of a concatenation of a lift of a basis of $U/U^\prime$, a basis of $U^\prime$, and a basis of $Z(L_p)$.  This choice of basis determines an embedding $\mathbf{G} \subset \mathbf{GL}_\dimL$.  Clearly there is a $\Z$-basis of $\mathcal{L}$ that, for almost all primes, induces a $\Z_p$-basis of $\mathcal{L}_p$ satisfying our hypotheses.  In the examples of Section~\ref{sec:examples} below, this will be the case for all primes.

\subsection{Consequences of rigidity}
Suppose, in addition to the hypotheses of Section~\ref{sec:consequences}, that $L_p$ is $Z(L_p)$-rigid in the sense of Definition~\ref{def:rigidity}, where $Z(L_p)$ is the center of $L_p$.  A consequence of rigidity, in conjunction with the interpretation of pro-isomorphic zeta functions as $p$-adic integrals in Proposition~\ref{pro:padic.integral}, is the existence of a fine Euler decomposition as in Theorem~\ref{thm:uniformity}.
To make this precise, set $U_1^F = U \tensor_{\Q} F$ and $U_2^F = Z(L_p \tensor_{\Q_p} F)$ for any finite extension $F/\Q_p$.  Then $L_p \tensor_{\Q_p} F = U_1^F \oplus U_2^F$, and each component is stable under the action of $\mathbf{H}(F)$.  Note that $\mathbf{N}(F)$ need not act trivially on $V^F / U_2^F$, where $V^F$ is the underlying $F$-vector space of $L_p \tensor_{\Q_p} F$, so we are not necessarily in the precise setup of Section~\ref{sec:assumptions}.  Nevertheless, the map $(\psi^\prime_2)^F : \mathbf{G}(F) \to \Aut (V^F/U_2^F)$ can be defined as in Section~\ref{sec:assumptions}; as a consequence of the assumption that $Z(L_p) \leq \gamma_2 L_p$, every element of $\ker (\psi^\prime_2)^F$ is unipotent.  Hence $N_2^F = \mathbf{N}(F) \cap \ker (\psi^\prime_2)^F = \mathbf{J}(F)$ and, as in Section~\ref{sec:assumptions}, we obtain a map $\psi_2^F : \mathbf{G}(F)/N_2^F \to \Aut_F (V^F/U_2^F) \subset \mathrm{GL}_{\overdim} (F)$, where the final embedding is determined by our choice of basis of $L$.

Define $\overline{N}^F = \mathbf{N}(F)/\mathbf{J}(F)$, and let $\mu_{\overline{N}^F}$ be the right Haar measure on $\overline{N}^F$ normalized so that $\mu_{\overline{N}^F}((\psi_2^F)^{-1}(\psi_2^F(\overline{N}^F) \cap M_{\overdim} (\mathcal{O}_F))) = 1$.  Finally, define a function $\widetilde{\theta}^F : \mathbf{H}(F) \to \mathbb{R}$ as follows.  If $h \in \mathbf{H}(F)$, then we set
\begin{equation*}
\widetilde{\theta}^F(h) = \mu_{\overline{N}^F} ( \{ \nu \in \overline{N}^F : \psi_2^F (\nu h) \in M_{\overdim} (\mathcal{O}_F) \} ).
\end{equation*}
\begin{rem} \label{rem:trivial.theta}
It is easy to see that if $\widetilde{\theta}^F(h) \neq 0$, then necessarily $h \in \mathbf{H}^+(F)$.  Moreover, if $\mathbf{N}(F)$ acts trivially on $U_1^F$, then $\widetilde{\theta}^F(h) = 1$ for all $h \in \mathbf{H}^+(F)$.
\end{rem}

Any $h \in \mathbf{H}(F)$ maps the center of $L_p \tensor_{\Q_p} F$ onto itself and thus induces an element $\varepsilon^F(h) \in \Aut_F Z(L_p \tensor_{\Q_p} F) \simeq \mathrm{GL}_{\dimL - \overdim}(F)$.  For every $\p | p$, write $\varepsilon_\p$ for the map $\varepsilon^{K_{\p}}$.

\begin{pro} \label{pro:rigid.computation}
Let $\mathcal{L}$ be a nilpotent Lie lattice such that $Z(\mathcal L) \leq \gamma_2 \mathcal L$, and let $p$ be a prime such that $L_p$ is $Z(L_p)$-rigid.  Here $\mathcal{L}_p = \mathcal{L} \tensor_\Z \Z_p$ and $L_p = \mathcal{L} \tensor_Z \Q_p$.  Let $K$ be a number field.  Set $d = [K:\Q]$ and $\dimL^\prime = \dim_{\Q} L/\gamma_2 L$.
Then
\begin{equation*} \label{equ:integral.factors}
 \zeta^\wedge_{\mathcal{L} \tensor_\Z \mathcal{O}_K, p}(s) = \prod_{\p | p} \int_{\mathbf{H}^+(K_\p)} \widetilde{\theta}^{K_\p}(h) |\det \varepsilon_\p(h)|_{\p}^{-d \dimL^\prime} |\det h|^s_{\p} d \mu_{\mathbf{H}(K_\p)}(h). 
 \end{equation*}
\end{pro}
\begin{proof}
For brevity, set $\boldsymbol{\Gamma} = (\aAut (L_p \tensor_{\Q_p} R))^\circ$, where $L_p \tensor_{\Q_p} R$ is viewed as a $\Q_p$-Lie algebra.  
Let $\Gamma = \boldsymbol{\Gamma}(\Q_p)$,
and let $\Gamma^+ \subset \Gamma$ be the submonoid consisting of $\Q_p$-automorphisms of $L_p \tensor_{\Q_p} R$ that preserve the lattice $\mathcal{L}_p \tensor_{\Z_p} \mathcal{O}_K$.  
Let $\boldsymbol{\Gamma} = \mathbf{N}_{\boldsymbol{\Gamma}} \rtimes \mathbf{H}_{\boldsymbol{\Gamma}}$, where $\mathbf{N}_{\boldsymbol{\Gamma}}$ is the unipotent radical of $\boldsymbol{\Gamma}$ and $\mathbf{H}_{\boldsymbol{\Gamma}}$ is reductive.  This induces a decomposition
$\Gamma = N_\Gamma \rtimes H_\Gamma$, where $N_\Gamma = \mathbf{N}_{\boldsymbol{\Gamma}}(\Q_p)$ and $H_\Gamma = \mathbf{H}_{\boldsymbol{\Gamma}}(\Q_p)$.  We have an embedding $\Gamma \subseteq \mathrm{GL}_{d \dimL}(\Q_p)$ determined by the $\Q_p$-basis
$$(\alpha_1 b_1, \dots, \alpha_d b_1, \alpha_1 b_2, \dots, \alpha_d b_2, \dots, \alpha_1 b_{\dimL}, \dots, \alpha_d b_{\dimL})$$ 
of $L_p \tensor_{\Q_p} R$.  
By Proposition~\ref{pro:padic.integral},
\begin{equation} \label{equ:separate.variables}
\zeta^\wedge_{\mathcal{L} \tensor_\Z \mathcal{O}_K, p}(s) = \int_\Gamma \chi_{\Gamma^+}(g) |\det g|_{\Q_p}^{-s} d \mu_{\Gamma}(g) = \int_{H_\Gamma}  f(h) | \det h|_{\Q_p}^{-s}  d \mu_{H_\Gamma}(h),
\end{equation}
where $f(h) = \mu_{N_\Gamma} (\{ \nu \in N_\Gamma : \nu h \in \Gamma^+ \})$, while $\mu_{H_\Gamma}$ and $\mu_{N_\Gamma}$ are right Haar measures as above.  Observe that $f(h) = 0$ whenever $h \not\in H_\Gamma \cap \Gamma^+$.  Let $J_\Gamma = \ker \psi^\prime_\Gamma \leq N_\Gamma$, where
$$\psi_\Gamma^\prime: \Gamma \to \Aut_{\Q_p} ((L_p \tensor_{\Q_p} R)/(Z(L_p) \tensor_{\Q_p} R)) \subseteq \mathrm{GL}_{d \overdim}(\Q_p).$$  
This induces a map $\psi_\Gamma : \Gamma / J_\Gamma \to \mathrm{GL}_{d \overdim}(\Q_p)$.  Let $\mu_{J_\Gamma}$ and $\mu_{N_\Gamma / J_\Gamma}$ be right Haar measures on $J_\Gamma$ and $N_\Gamma / J_\Gamma$, respectively, normalized so that $\mu_{J_\Gamma}(J_\Gamma \cap M_{d \dimL}(\Z_p)) = 1$ and $\mu_{N_\Gamma / J_\Gamma}(\psi_\Gamma^{-1}(\psi_\Gamma(N_\Gamma / J_\Gamma) \cap M_{d \overdim}(\Z_p))) = 1$.  
Clearly $\mu_{N_\Gamma} = \mu_{N_\Gamma / J_\Gamma} \cdot \mu_{J_\Gamma}$.  

Every element of $J_\Gamma$ acts trivially on $\gamma_2 (L_p \tensor_{\Q_p} R)$ and therefore 
\begin{equation} \label{equ:j.gamma}
J_\Gamma = \left\{
\left( \begin{array}{ccc} I_{d \dimL^\prime} & 0 & B \\
0 & I_{d(\overdim - \dimL^\prime)} & 0 \\
0 & 0 & I_{d(\dimL - \overdim)} \end{array} \right) : 
B \in M_{d \dimL^\prime, d(\dimL - \overdim)}(\Q_p) \right\}.
\end{equation}
Fix elements $h \in H_\Gamma$ and $\nu \in N_\Gamma$.  For any $\gamma \in  \Gamma$, let $\varepsilon(\gamma)$ be the automorphism of $Z(L_p \tensor_{\Q_p} R)$ induced by $\gamma$.  It follows from the $Z(L_p)$-rigidity of $L_p$ that we may take
$H_\Gamma = \prod_{\p | p} \iota_\p ( \mathbf{H}(K_\p))$.  Then $h$ corresponds to an $\npr$-tuple $h = (\iota_{\p_1} (h_1), \dots, \iota_{\p_\npr} (h_\npr))$, where $h_i \in \mathbf{H}(K_{\p_i})$ for every $i \in [\npr]$.
It is easy to see from~\eqref{equ:j.gamma} that, for $j \in J_\Gamma$, the condition $j \nu h \in \Gamma^+$ amounts to $d \dimL^\prime$ independent conditions on the rows of $B$.  Hence
\begin{equation} \label{equ:j.measure}
\setlength{\thickmuskip}{.9\thickmuskip}
\setlength{\medmuskip}{.9\medmuskip}
\mu_{J_\Gamma}( \{ j \in J_\Gamma : j \nu h \in \Gamma^+ \}) = | \det \varepsilon(\nu h) |_{\Q_p}^{-d \dimL^\prime} = | \det \varepsilon(h) |_{\Q_p}^{-d \dimL^\prime} = \prod_{i = 1}^\npr | \det \varepsilon^{K_{\p_i}}(h_i) |_{\p_i}^{- d \dimL^\prime}.
\end{equation}
Note that this quantity depends only on $h$.
Again by rigidity, we have $N_\Gamma / J_\Gamma = \prod_{\p | p} \overline{N}^{K_{\p}}$.  Moreover, since we chose each $\boldsymbol{\alpha}_i$ to be an integral basis, it follows that for any $g_i \in \mathbf{G}(K_{\p_i})$ we have $\iota_{\p_i} (g_i) \in M_{e_i f_i \dimL}(\Z_p)$ if and only if $g_i \in \mathbf{G}(\mathcal{O}_{\p_i})$.  Thus $\mu_{N_\Gamma / J_\Gamma}$ is the product of the measures $\mu_{\overline{N}^{K_\p}}$ defined earlier.  Recall that $\mu_{N_\Gamma} = \mu_{N_\Gamma / J_\Gamma} \cdot \mu_{J_{\Gamma}}$.  Consequently, 
\begin{multline*}
 f(h) = \left( \prod_{i = 1}^\npr | \det \varepsilon_{\p_i}(h_i) |_{\p_i}^{-d \dimL^\prime} \right) \mu_{N_\Gamma / J_\Gamma}( \{ \overline{\nu} \in N_\Gamma / J_\Gamma : \psi_\Gamma(\overline{\nu}h) \in M_{d \overdim}(\Z_p) \}) = \\
\prod_{i = 1}^\npr  | \det \varepsilon_{\p_i}(h_i) |_{\p_i}^{-d \dimL^\prime} \mu_{\overline{N}^{K_{\p_i}}}\{ \overline{\nu}_i \in \overline{N}^{K_{\p_i}} : \psi^{K_{\p_i}}_2(\overline{\nu}_i h_i \in M_{\overdim}(\mathcal{O}_{\p_i})) \}  = \\
\prod_{i = 1}^\npr | \det \varepsilon_{\p_i}(h_i) |_{\p_i}^{-d \dimL^\prime} \widetilde{\theta}^{K_{\p_i}} (h_i).
\end{multline*}
From this it is obvious that the rightmost integral of~\eqref{equ:separate.variables} splits into a product as claimed, observing that the integrand is supported on $H_\Gamma \cap \Gamma^+ = \prod_{\p | p} \iota_\p (\mathbf{H}^+(K_\p))$.
\end{proof}

\subsection{Simplifying assumptions} \label{sec:simplifying.assumptions}
Suppose further that Assumption~\ref{second.assumption} holds for the action of $\mathbf{G}(\Q_p)$ on the underlying $\Q_p$-vector space $V$ of $L_p$ and for some decomposition
$ V = U_1 \oplus \cdots \oplus U_t$
compatible with the basis $(b_1, \dots, b_\dimL)$.  Moreover, suppose that $V_t = U_t = Z(L_p)$ and there exists some $t^\prime \leq t$ such that $V_{t^\prime} = \gamma_2 L_p$.  
Since we have assumed $L_p$ to be $Z(L_p)$-rigid, it follows that, for any finite extension $F/\Q_p$, the action of $\mathbf{G}(F)$ on
$$ L_p \tensor_{\Q_p} F = (U_1 \tensor_{\Q_p} F) \oplus \cdots \oplus (U_t \tensor_{\Q_p} F)$$
satisfies Assumption~\ref{second.assumption}.  
\begin{rem} \label{rmk:almost.all.primes}
Let $\mathcal{L}$ be a Lie lattice of rank $\dimL$ and let $\mathbf{G} = \aAut L$.  By the proof of~\cite[Lemma~4.1]{duSLubotzky/96}, $\mathbf{G}(F)$ satisfies Assumption~\ref{first.assumption} for almost all primes $p$ and any finite extension $F / \Q_p$.  Similarly, for almost all primes $p$, Assumption~\ref{second.assumption} holds for $L_p$; if in addition $L_p$ is $Z(L_p)$-rigid, Assumption~\ref{second.assumption} holds for any base extension $L_p \tensor_{\Q_p} F$, where $F / \Q_p$ is finite.  
Indeed, it is easy to see using Lemmas~4.2 and~4.3 of~\cite{duSLubotzky/96} that one can choose a $\Z$-basis $(b_1, \dots, b_{\dimL})$ of $\mathcal{L}$ such that, for almost all primes $p$, the induced basis of $\mathcal{L}_p$ satisfies the hypotheses in the previous paragraph.  The crucial step in the proof of~\cite[Lemma~4.3]{duSLubotzky/96} is the claim that the action of $\mathfrak{N}(\Q)$ on $L$ has a non-zero fixed point, where $\mathfrak{N}$ is the unipotent radical of $\aAut L$; this is true because $\mathfrak{N}$ is trigonalizable over $\Q$~\cite[Theorem~15.4(ii)]{Borel/91}, which in turn is a consequence of the Lie-Kolchin theorem.

By contrast, and contrary to~\cite[Corollary~4.5]{duSLubotzky/96}, there may be no decomposition of $L_p$ affording Assumption~\ref{lifting.condition} at any prime; see~\cite[p.~6]{Berman/11} for an example.
\end{rem}
We require a stronger version of Assumption~\ref{lifting.condition}.  Informally, 
we assume that the lifting condition for the action of $\mathbf{G}(\Q_p)$ on $L_p$ is realized by a polynomial map defined over $\Z_p$;  
this condition ensures that Assumption~\ref{lifting.condition} holds not only for $L_p$, but for the base extension $L_p \tensor_{\Q_p} F$, where $F / \Q_p$ is any finite extension.
Recall the maps $\psi_i$ and the notation $d_i = \dim_{\Q_p} (U_1 \oplus \cdots \oplus U_i)$ from Section~\ref{sec:assumptions}.  
Observe that if $t^\prime \leq 3$, then $\psi_i$ is injective for all $i \in [2,t-1]$.
If $i \in [2,t-1]$ and $\kappa_i: \mathbf{GL}_{d_i} \to \mathbf{GL}_{\dimL}$ is a morphism of algebraic varieties, defined over $\Z_p$, write 
$$\kappa_i^{\Q_p} :  \Aut_{\Q_p}(V/V_{i+1}) \simeq \mathrm{GL}_{d_i}(\Q_p) \to \mathrm{GL}_{\dimL}(\Q_p) \simeq \Aut_{\Q_p}(V)$$
for the induced map on $\Q_p$-points.  
\begin{asm} \label{polynomial.lifting}
We assume for all $i \in [2,t-1]$ that $\psi_i$ is injective and that there is a morphism $\kappa_i : \mathbf{GL}_{d_i} \to \mathbf{GL}_{\dimL}$, defined over $\Z_p$, such that if $\overline{g} \in NH^+/N_i \cap (G/N_i)^+$, then the image $\gamma = \kappa_i^{\Q_p}(\psi_i(\overline{g})) \in \mathrm{GL}_{\dimL}(\Q_p)$ lies in $\mathbf{G}^+(\Q_p)$ and satisfies $\overline{g} = \gamma N_i$.  
\end{asm}
Assumption~\ref{polynomial.lifting} implies, for any finite extension $F/\Q_p$, that the map $\kappa_i^F$ on $F$-points realizes Assumption~\ref{lifting.condition} for the action of $\mathbf{G}(F)$ on $L_p \tensor_{\Q_p} F$.  Hence, under the preceding series of assumptions, the setup of Section~\ref{sec:assumptions} applies to the action of $\Aut_{F}((L_p \tensor_{\Q_p} F)/Z(L_p \tensor_{\Q_p} F))$ on $(L_p \tensor_{\Q_p} F)/Z(L_p \tensor_{\Q_p} F)$.  By Proposition~\ref{pro:duSL},
\begin{equation} \label{equ:tilde.decomposition}
 \widetilde{\theta}^{F} (h) = \prod_{i = 1}^{t - 2} \theta_i^{F}(h)
 \end{equation}
for all $h \in \overline{\mathbf{H}}(F) \simeq \mathbf{H}(F)$, where $\overline{\mathbf{H}}$ is the reductive part of $\aAut L_p/Z(L_p)$.  Many of the examples considered in Section~\ref{sec:examples} rely on the following formula.  
\begin{cor} \label{cor:rigid.lifting}
Suppose that $Z(\mathcal L) \leq \gamma_2 \mathcal L$ and that $L_p = L \tensor_\Q \Q_p$ is $Z(L_p)$-rigid.  Suppose that Assumptions~\ref{first.assumption},~\ref{second.assumption}, and~\ref{polynomial.lifting} hold.  Let $K$ be a number field of degree $d$, and put $\dimL^\prime = \dim_{\Q} L/\gamma_2 L$.
Then
\begin{align} \label{equ:technical.splitting} \nonumber
\zeta^\wedge_{\mathcal{L} \tensor_\Z \mathcal{O}_K, p}(s) = \prod_{\p | p} \int_{\mathbf{H}^+(K_\p)} \left( \prod_{j = 1}^{t-2} \theta_j^{K_\p} (h) \right) |\det \varepsilon_\p(h)|_{\p}^{-d \dimL^\prime} |\det h|^s_{\p} d \mu_{\mathbf{H}(K_\p)}(h) = \\
\prod_{\p | p} \int_{\mathbf{H}^+(K_\p)} \left( \prod_{j = 1}^{t-2} \theta_j^{K_\p} (h) \right) (\theta_{t-1}^{K_\p}(h))^d |\det h|^s_{\p} d \mu_{\mathbf{H}(K_\p)}(h).
\end{align}
\end{cor}
\begin{proof}
The first equality is immediate from Proposition~\ref{pro:rigid.computation} and~\eqref{equ:tilde.decomposition}.  We prove the second by showing that $\theta_{t-1}^{F}(h) = | \det \varepsilon^F(h) |_F^{- \dimL^\prime}$ for any finite extension $F / \Q_p$ and any $h \in \mathbf{H}^+(F)$.  In the notation of~\eqref{equ:def.theta}, we have $N_{t-1} = \mathbf{J}(F)$ and $N_t = (0)$.  Thus, 
\begin{alignat*}{1}
\theta_{t-1}^F(h) = \, & \mu_{\mathbf{J}(F)} ( \{ j \in \mathbf{J}(F) : \tau(h)(j) \in M_{\overdim, \dimL - \overdim}(\mathcal{O}_F) \}) = \\
& \mu_{\mathbf{J}(F)} ( \{ j \in \mathbf{J}(F) : jh \in \mathbf{G}^+(F) \} ) = | \det \varepsilon^F(h) |_F^{- \dimL^\prime}.
\end{alignat*}
Here the second equality holds because $h \in \mathbf{H}^+(F)$ and the third is analogous to~\eqref{equ:j.measure} in the proof of Proposition~\ref{pro:rigid.computation}.
\end{proof}

\begin{rem} \label{rem:main.proof}
To deduce Theorem~\ref{thm:intro.examples} from Corollary~\ref{cor:rigid.lifting}, it remains to show, for each of the seven classes of Lie lattices $\mathcal{L}$ considered in the statement of Theorem~\ref{thm:intro.examples}, that the integrals appearing in the right-hand side of~\eqref{equ:technical.splitting} can be expressed as $W_{\mathcal{L},d}(q_\p,q_\p^{-s})$, where $q_\p$ is the cardinality of the residue field of $K_\p$ and $W_{\mathcal{L},d}(X,Y) \in \Q(X,Y)$ is a rational function depending only on $\mathcal{L}$ and $d$.  This will be done in Section~\ref{sec:examples} by means of explicit computations.
\end{rem}

\begin{rem} \label{rem:duSL}
Our arguments in this section follow the path laid out by du Sautoy and Lubotzky~\cite{duSLubotzky/96}.  Assumptions~\ref{first.assumption} and~\ref{second.assumption} appear in~\cite{duSLubotzky/96}, as does Assumption~\ref{polynomial.lifting} implicitly.  
In~\cite[\S6]{duSLubotzky/96} the authors anticipate our line of inquiry by considering integrals of the form~\eqref{equ:separate.variables} where the reductive group $\mathbf{H}$ over the number field $E$ can be identified with the restriction of scalars of an algebraic matrix group defined over a finite extension $E^\prime / E$.  If $\mathcal{L}$ is a nilpotent Lie lattice such that $Z(\mathcal{L}) \leq \gamma_2 \mathcal{L}$ and $L = \mathcal{L} \tensor_\Z \Q$ is $Z(L)$-rigid, and if $K$ is a number field, then the situation of~\cite[\S6]{duSLubotzky/96} is obtained, with $E^\prime / E = K / \Q$, in the computation of $\zeta^\wedge_{\mathcal{L} \tensor_\Z \mathcal{O}_K, p}$ at almost all $p$.  Under further hypotheses (for instance, that the integrand of~\eqref{equ:separate.variables} is a character) it is shown~\cite[Theorem~6.9]{duSLubotzky/96} that the local pro-isomorphic zeta functions $\zeta^\wedge_{\mathcal{L} \tensor_\Z \mathcal{O}_K, p}$ satisfy functional equations for almost all $p$.  

By emphasizing the notion of rigidity and proving the sufficient condition of Theorem~\ref{thm:rigidity}, we provide a way of identifying Lie lattices $\mathcal{L}$ to which the framework of~\cite[\S6]{duSLubotzky/96} applies.  We show how the local zeta functions $\zeta^\wedge_{\mathcal{L} \tensor_\Z \mathcal{O}_K, p}(s)$ may be determined by a uniform calculation that depends only mildly on the number field $K$.  We do not make further assumptions about the integrand of~\eqref{equ:separate.variables}, and indeed our method applies to Lie lattices $\mathcal{L}$ such that $\zeta^\wedge_{\mathcal{L} \tensor_\Z \mathcal{O}_K, p}(s)$ does not satisfy a functional equation for any prime $p$; see Section~\ref{sec:bk}.
\end{rem}

\subsection{Some non-rigid Lie algebras} \label{sec:non-rigid}
Let $L$ be a nilpotent finite-dimensional Lie algebra over a field $k$.  If there exists a finite separable extension $K/k$ with non-trivial $\mathrm{Aut}_k(K)$ such that $L \tensor K$ is decomposable as a Lie algebra over $K$, then it is easy to see that $L$ is not $Z$-rigid for any highly invariant ideal $Z \leq \gamma_2 L$.  Indeed, suppose that $L \tensor K \simeq L_1 \oplus L_2$, where $L_1$ and $L_2$ are $K$-Lie algebras.  Let $\sigma \in \mathrm{Aut}_k(K)$ be a non-trivial element.  Consider the $k$-linear automorphism $\varphi \in \mathrm{Aut}_k(L \tensor K)$ corresponding to the automorphism $\mathrm{id} \oplus \sigma$ of $L_1 \oplus L_2$.  It is clear that, modulo $Z \tensor K$ for any $Z$ as above, the map $\varphi$ fails to be $\tau$-semilinear for any $\tau \in \mathrm{Aut}_k(K)$.  

For the simplest example of this phenomenon, suppose $\mathrm{char} \, k \neq 2$ and take $\alpha \in k$ such that $X^2 - \alpha \in k[X]$ is irreducible.  Consider the $k$-Lie algebra $L = \langle x_1, x_2, x_3, x_4, z_1, z_2 \rangle_k$, where $Z(L) = \langle z_1, z_2 \rangle_k$ and the remaining generators satisfy the relations
$$
\begin{array}{lcl}
[x_1, x_2] = z_1 & {[x_1, x_3] = z_2} & {[x_1, x_4] = 0} \\
{[x_3, x_4] = \alpha z_1} & [x_2, x_4] = z_2 & [x_2, x_3] = 0.
\end{array}
$$
While $L$ is indecomposable, it is easy to show that $L \simeq H \tensor k(\sqrt{\alpha})$ as $k$-Lie algebras, where $H = \langle x, y, z \rangle_k$ is the Heisenberg Lie algebra in which $[x,y] = z$ and $z$ is central; see Remark~\ref{rem:heisenberg}.  Hence, setting $K = k(\sqrt{\alpha})$, we find that $L \tensor K \simeq (H \tensor K) \oplus (H \tensor K)$ as $K$-Lie algebras.  Thus $L$ is not $Z$-rigid for any $Z \leq \gamma_2 L$.

In the case $k = \Q$ such Lie algebras were studied by Scheuneman~\cite[Proposition~1]{Scheuneman/67}.  They also appear as the Lie algebras $\mathcal{L}(G) \tensor_\Z \Q$ associated, via the correspondence of~\eqref{equ:class.two}, to the nilpotent groups $G$ arising from the irreducible polynomials $X^2 - \alpha \in \Z[X]$ by the construction of~\cite[Theorem~6.3]{GSegal/84}; see also~\cite{BKO/even}.  Using the description of $\aAut (\mathcal{L}(G) \tensor \Q)$ given in~\cite[Theorem~1.4]{BKO/even} and the observations above, the pro-isomorphic zeta functions $\zeta^\wedge_{\mathcal{L}(G)}(s)$ may be deduced from Remark~\ref{rem:heisenberg}.

\section{Uniform rationality} \label{sec:appendix}
In this section we prove Theorem~\ref{thm:uniformity}.  
Let $\mathrm{Loc}$ denote the collection of pairs $(F,\pi_F)$, where $F$ is a non-Archimedean local field with normalized additive valuation $v_F$ and $\pi_F \in F$ is a uniformizer; we will generally omit $\pi_F$ in the notation.  Write $\mathcal{O}_F$ for the ring of integers of $F \in \mathrm{Loc}$, and let $k_F = \mathcal{O}_F / (\pi_F)$ be the residue field and $q_F = |k_F|$ its cardinality.  The multiplicative valuation on $F$ is given by $| x |_F = q_F^{-v_F(x)}$ for $x \in F$.  Let $\mathrm{Loc}_{\gg} \subset \mathrm{Loc}$ denote the collection of $F \in \mathrm{Loc}$ with $\mathrm{char} \, k_F \geq M$ for a suitable $M$.  Let $\mathrm{Loc}^0 \subset \mathrm{Loc}$ denote the collection of local fields of characteristic zero.  

We briefly recall the Denef-Pas language $\mathfrak{L}_{\mathrm{DP}}$ in which we work; the reader is referred to~\cite[\S4.1]{CGH/14} or~\cite[\S2]{BDOP/11} for details. 
The language $\mathfrak{L}_{\mathrm{DP}}$ has three sorts: the valued field sort $\mathrm{VF}$ and the
residue field sort $\mathrm{RF}$ are endowed with the language of
rings, and the valued group sort $\mathrm{VG}$, which we will simply call
$\mathbb{Z}$, is endowed with the {Presburger language} $\mathfrak{L}_{\mathrm{Pres}}$ of ordered abelian groups.
The language $\mathfrak{L}_{\mathrm{DP}}$ also has two function
symbols: $v:\mathrm{VF}\backslash\{0\}\rightarrow\mathrm{VG}$ and
$\mathrm{ac}:\mathrm{VF}\rightarrow\mathrm{RF}$, interpreted as a
valuation map and an angular component map.

Any formula $\phi$ in $\mathfrak{L}_{\mathrm{DP}}$ with $m_{1}$
free $\mathrm{VF}$-variables, $m_{2}$ free $\mathrm{RF}$-variables,
and $m_{3}$ free $\mathbb{Z}$-variables yields a subset $\phi(F)\subseteq F^{m_{1}}\times k_{F}^{m_{2}}\times\mathbb{Z}^{m_{3}}$
for any $F\in\mathrm{Loc}$.  A collection $X=(X_{F})_{F\in\mathrm{Loc}_{\gg}}$
with $X_{F}=\phi(F)$ is called an \textsl{$\mathfrak{L}_{\mathrm{DP}}$-definable
set}.  A collection $f=(f_{F}:X_{F}\rightarrow Y_{F})_{F\in\mathrm{Loc}_{\gg}}$
is called an \textsl{$\mathfrak{L}_{\mathrm{DP}}$-}\textit{definable
function} if the associated collection of graphs is an \textsl{$\mathfrak{L}_{\mathrm{DP}}$-}definable
set. Similarly, any $\mathfrak{L}_{\mathrm{Pres}}$-formula $\theta$ in $m$ free variables
gives rise to an $\mathfrak{L}_{\mathrm{Pres}}$-definable subset $X\subseteq\mathbb{Z}^{m}$
(also called a \emph{Presburger set}).

We first establish a generalization of~\cite[Theorem~B]{BDOP/11}.  
\begin{thm} \label{thm:bdop}
Let $Y = (Y_F)_F$ be an $\mathfrak{L}_{\mathrm{DP}}$-definable subset of $\mathrm{VF}^{m}$,
let $d\in\mathbb{N}$, and let $\Phi_{i}:Y \rightarrow\mathbb{Z}$
be $\mathfrak{L}_{\mathrm{DP}}$-definable functions for $i\in\{1,2,3\}$.
Suppose that there exists some $\alpha \in\mathbb{R}$ such that for all $F \in \mathrm{Loc}_\gg$ the
integral $\int_{Y_{F}}q_{F}^{\Phi_{1}(y)s+\Phi_{2}(y)d+\Phi_{3}(y)}d\mu$
converges for all $s\in\mathbb{C}$ with $\mathrm{Re} \, s>\alpha$, where $\mu$ is the Haar measure on $F^m$, normalized so that $\mu(\mathcal{O}_F^m) = 1$. 

Then there exist $\unifmax, e \in \N$, a collection of $e$-ary formulae $\psi_1(\xi), \dots, \psi_{\unifmax}(\xi)$ in the language of rings independent of $d$, and rational functions
$W_{1}^{(d)}(X,Y),...,W_{\unifmax}^{(d)}(X,Y)\in\mathbb{Q}(X,Y)$ such that the following holds for all $F\in\mathrm{Loc}_{\gg}$
and all $s\in\mathbb{C}$ satisfying $\mathrm{Re} \, s>\alpha$:
\[
\int_{Y_{F}}q_{F}^{\Phi_{1}(y)s+\Phi_{2}(y)d+\Phi_{3}(y)}d\mu=\sum_{i=1}^{\unifmax}m_{i}(k_{F})\cdot W_{i}^{(d)}(q_{F},q_{F}^{-s}),
\]
where $m_{i}(k_{F})$ denotes the cardinality of the set $\{\xi\in k_{F}^{e}:k_{F}\models \psi_{i}(\xi)\}$
for all $i\in[\unifmax]$. Moreover, for each $i \in [\unifmax]$ there exist $M_i, N_i \in \N$ and integers $A_{ij}^{(0)}$, $A_{ij}^{(1)}$, $B_{ij}$, $\kappa^{(i)}_j$ for all $j \in [M_i]$ and $C_{ij}^{(0)}$, $C_{ij}^{(1)}$, $D_{ij}$ for all $j \in [N_i]$ such that for all $d \in \N$ the following holds:
\begin{equation} \label{equ:w.form}
W_{i}^{(d)}(X,Y) = \frac{ \sum_{j = 1}^{M_i} \kappa^{(i)}_j X^{A_{ij}^{(0)} + dA_{ij}^{(1)}} Y^{B_{ij}} }{\prod_{j = 1}^{N_i} (1 - X^{C_{ij}^{(0)} + dC_{ij}^{(1)}} Y^{D_{ij}})}.
\end{equation}
\end{thm}
\begin{proof}
We follow the proof of~\cite[Theorem~B]{BDOP/11}, which uses methods
of Pas~\cite{Pas/89}. If $F\in\mathrm{Loc}_{\gg}$, then for all $s\in\mathbb{C}$
with $\mathrm{Re} \, s> \alpha$ we have

\begin{equation} \label{eq:(0.1)}
\int_{Y_{F}}q_{F}^{\Phi_{1}(y)s+\Phi_{2}(y)d+\Phi_{3}(y)}d\mu=\sum_{\bfn = (n_{1},n_{2},n_{3}) \in \mathbb{Z}^{3}} q_{F}^{n_{1}s+n_{2}d+n_{3}}\mu\left(A_{\bfn,F}\right),
\end{equation}
where 
$A_{\bfn,F}=\{y \in Y_{F}:\Phi_{i}(y)=n_{i}\text{ for }i \in [3] \}$.

As in the proof of~\cite[Theorem~B]{BDOP/11}, after $m$ applications
of~\cite[Theorem~4.1]{BDOP/11} we obtain a uniform expression for
the measures $\mu(A_{\bfn,F})$, valid for all $F\in\mathrm{Loc}_{\gg}$; note that the bound $M$ defining $\mathrm{Loc}_\gg$ may have increased.
Rewriting the formulas in this expression using quantifier elimination~\cite[Theorem~4.1]{Pas/89} (see also~\cite[Theorem~3.7]{BDOP/11})
and the model-theoretic orthogonality of the sorts $\mathrm{VG}$
and $\mathrm{RF}$~\cite[Lemma 5.3]{Pas/89},
for some natural numbers $e,N,\nu_{1},...,\nu_{N}$, we can express~\eqref{eq:(0.1)}
as:
\begin{multline} \label{eq:(0.2)}
q_F^{-m} \sum_{i = 1}^N \sum_{j = 1}^{\nu_i} \sum_{\xi \in k_F^e \atop k_F \models \psi_{ij}(\xi)} \sum_{(l_1, \dots, l_m, n_1, n_2, n_3) \in \Z^{m+3} \atop \Z \models \theta_{ij}(l_1, \dots, l_m, n_1, n_2, n_3)} q_F^{n_1 s + n_2 d + n_3 - l_1 - \cdots - l_m} = \\
q_{F}^{-m}\sum_{i=1}^{N}\sum_{j=1}^{\nu_{i}}m_{ij}(k_{F})\sum_{
(l_{1},...,l_{m},n_{1},n_{2},n_{3}) \in\mathbb{Z}^{m+3} \atop
\mathbb{Z}\models\theta_{ij}(l_{1},...,l_{m},n_{1},n_{2},n_{3})
}q_{F}^{n_{1}s+n_{2}d+n_{3}-l_{1}-...-l_{m}},
\end{multline}
where the $\theta_{ij}$ are $\mathfrak{L}_{\mathrm{Pres}}$-formulae, the $\psi_{ij}$ are formulae in the language of rings, 
and $m_{ij}= | \{\xi\in k_{F}^{e}:k_{F} \models \psi_{ij}(\xi)\} |$.  We emphasize that the formulae $\theta_{ij}$ and $\psi_{ij}$ do not depend on $d$.
It remains to show that, for fixed $i$ and $j$, the inner summation in~\eqref{eq:(0.2)}
is given by a rational function in $q_{F}$ and $q_{F}^{-s}$ of the form~\eqref{equ:w.form}.

Let $S_{ij}$ denote the $\mathfrak{L}_{\mathrm{Pres}}$-definable
subset of $\mathbb{Z}^{m+3}$ defined by $\theta_{ij}$.  Then by~\cite[Theorem~2]{Cluckers/03} (see also~\cite[Theorem~2.1.9]{CGH/14}),
there is a finite partition of $S_{ij}$ into $\mathfrak{L}_{\mathrm{Pres}}$-definable
parts such that, for every part $A$, there exists a Presburger linear bijection $\rho:A\rightarrow \N_0^u$~\cite[Definition~1]{Cluckers/03}, where
$u$ depends only on the part $A$.  Changing coordinates using $\rho$,
we have 
\begin{equation} \label{equ:star}
\sum_{(l_{1},...,l_{m},n_{1},n_{2},n_{3}) \in A}q_{F}^{n_{1}s+n_{2}d+n_{3}-\sum l_{i}}=q_{F}^{b_{1}s+b_{2}d+b_{3}}\sum_{(e_{1},...,e_{u})\in \N_0^u}q_{F}^{\sum_{\ell=1}^{u}(a_{1,\ell}s+a_{2,\ell}d+a_{3,\ell})e_{\ell}}
\end{equation}
for some rational numbers $b_{1},b_{2},b_{3}\in\mathbb{Q}$ and $\{a_{1,\ell}, a_{2,\ell}, a_{3,\ell} \}\in\mathbb{Q}$.
Since $A \subseteq S_{ij}$, and~\eqref{eq:(0.2)} is a sum of positive terms for all real $s > \alpha$,
we conclude that~\eqref{equ:star} converges for all real $s > \alpha$ and hence for all $s \in \C$ in the right half-plane $\mathrm{Re} \, s > \alpha$.  

Since $\rho$ induces a bijection between the summands on the right-hand side of~\eqref{equ:star} and those on the left-hand side, we find that $b_j + \sum_{\ell = 1}^u a_{j,\ell} e_\ell \in \Z$ for all $j \in [3]$ and all $(e_1, \dots, e_u) \in \N_0^u$. 
This immediately implies that $b_j \in \Z$ and $a_{j, \ell} \in \Z$ for every $j \in [3]$ and $\ell \in [u]$.
For all $\ell \in [u]$ the variable $e_\ell$ in~\eqref{equ:star} runs over $\N_0$, so we must have $a_{1, \ell} s + a_{2, \ell}d + a_{3, \ell} < 0$ for all real $s > \alpha$.  Summing the geometric series in these variables, we may rewrite the right-hand side of~\eqref{equ:star} as
\begin{equation} \label{equ:app.rat.funct}
\frac{q_{F}^{b_{1}s+b_{2}d+b_{3}}}{\prod_{\ell=1}^{u}\left(1-q_F^{a_{1,\ell}s+a_{2,\ell}d+a_{3,\ell}}\right)}.
\end{equation}
Note that~\eqref{equ:app.rat.funct} converges for all $s \in \C$ in the right half-plane $\mathrm{Re} \, s > \alpha$ and that the partition of $S_{ij}$, and hence the constants $a_{j, \ell}, b_j$ appearing in~\eqref{equ:app.rat.funct}, are independent of $d$.  Our claim follows.
\end{proof}

We use the notation of Section~\ref{sec:consequences} freely.  In particular, recall that if $\mathcal{L}$ is a Lie lattice, then we denote $L = \mathcal{L} \tensor_\Z \Q$, and that we write $\mathcal{L}_p = \mathcal{L} \tensor_\Z \Z_p$ and $L_p = \mathcal{L} \tensor_\Z \Q_p$.  

\begin{proof}[Proof of Theorem~\ref{thm:uniformity}]
Let $K$ be a number field of degree $d$.  
If $p$ is a rational prime such that $L_p$ is $Z(L_p)$-rigid, then Proposition~\ref{pro:rigid.computation} implies that $ \zeta^\wedge_{\mathcal{L} \tensor \mathcal{O}_K, p}(s)$ is a product, parametrized by the primes $\p | p$, of the following factors:
\begin{multline} \label{equ:gj.integral}
\int_{\mathbf{H}^+(K_\p)} \widetilde{\theta}^{K_\p}(h) |\det \varepsilon^{K_\p}(h)|_{\p}^{-d \dimL^\prime} |\det h|^s_{\p} d \mu_{\mathbf{H}(K_\p)}(h) = \\ 
 \int_{\psi_2^{K_\p}(\mathbf{G}(K_\p)/\mathbf{J}(K_\p)) \cap \mathbf{M}_{\overdim}(\mathcal{O}_{\p})} | \det \overline{g} |^s_{\p} |\det \overline{\varepsilon}^{K_\p} (\overline{g}) |^{s-d \dimL^\prime}_{\p} d \mu_{\mathbf{G}(K_\p) / \mathbf{J}(K_\p)} (\overline{g}).
 \end{multline}
Here $\overline{\varepsilon}^{K_\p} : (\aAut L_p / Z(L_p))(K_\p) \to (\aAut Z(L_p))(K_\p)$ is the natural map, which is well-defined since $Z(L_p) \leq \gamma_2 L_p$.

Our strategy is the following.  Analogously to~\cite[\S2]{GSS/88}, we recast the integrals of~\eqref{equ:gj.integral} as the integral~\eqref{equ:unif.integral}, which is defined uniformly in the Denef-Pas language $\mathfrak{L}_{\mathrm{DP}}$, as well as in the generalized Denef-Pas language $\mathfrak{L}_{\mathrm{gDP}}$, for all $F \in \mathrm{Loc}$.  Theorem~\ref{thm:bdop} applies for $F \in \mathrm{Loc}_\gg$ and immediately yields the second part of Theorem~\ref{thm:uniformity}.  To treat small primes and deduce the first part of Theorem~\ref{thm:uniformity}, we work in $\mathfrak{L}_{\mathrm{gDP}}$ and use results of~\cite{CluckersHalupczok/18}.

By the same argument as in the proof of Proposition~\ref{pro:padic.integral}, the integral of~\eqref{equ:gj.integral} equals
\begin{equation} \label{equ:reinterpretation}
\sum_{\mathcal{M} \in \Lambda(\mathcal{L}_\p)} [\mathcal{L}_\p : \mathcal{M}]^{-s} [Z(\mathcal{L}_\p) : \mathcal{M} \cap Z(\mathcal{L}_\p)]^{d \dimL^\prime},
\end{equation}
where $\mathcal{L}_\p = \mathcal{L}_p \otimes_{\Z_p} \mathcal{O}_\p$ and $\Lambda(\mathcal{L}_\p)$ is the set of finite-index $\mathcal{O}_\p$-submodules of $\mathcal{L}_\p$ that are isomorphic to $\mathcal{L}_\p$ as Lie lattices over $\mathcal{O}_\p$.  We now apply the ideas of~\cite[\S2]{GSS/88}, adapted to our linearized setting.  

Fix a $\Z$-basis $(b_1, \dots, b_\dimL)$ of $\mathcal{L}$ such that $\langle b_{\overdim + 1}, \dots, b_\dimL \rangle_\Z = Z(\mathcal{L})$ and $\langle b_{\dimL^\prime + 1}, \dots, b_{\dimL} \rangle_\Z = \{ x \in \mathcal{L} : \exists a \in \Z \setminus \{ 0 \}, \, ax \in \gamma_2 \mathcal{L} \}$.  For any rational prime $p$ satisfying the hypotheses of Theorem~\ref{thm:uniformity}, the induced $\Z_p$-basis of $\mathcal{L}_p$ meets the requirements of Section~\ref{sec:consequences}.  Now let $\p | p$, and let $\mathcal{M}$ be a finite-index $\mathcal{O}_\p$-sublattice of $\mathcal{L}_\p$.  We say that an $\mathcal{O}_\p$-basis $(c_1, \dots, c_\dimL)$ of $\mathcal{M}$ is a \emph{good basis} if, for every $i \in [ \dimL ]$, we have $\langle c_i, \dots, c_\dimL \rangle_{\mathcal{O}_\p} = \mathcal{M} \cap \mathcal{L}_{\p, i}$, where $\mathcal{L}_{\p, i} = \langle b_i, \dots, b_\dimL \rangle_{\mathcal{O}_\p} \subseteq \mathcal{L}_\p$.  Set $\mathcal{L}_{\p, \dimL + 1} = (0)$.  

Let $T_\dimL$ be the additive algebraic group of upper triangular $\dimL \times \dimL$ matrices, defined over $\Z$.  For any $F \in \mathrm{Loc}$ we identify $T_\dimL (F)$ with $F^{\binom{n+1}{2}}$ in the natural way and let $\mu$ be the additively invariant measure on $T_\dimL (F)$ normalized so that $\mu(T_\dimL (\mathcal{O}_F)) = 1$.  For $\mathcal{M} \in \Lambda(\mathcal{L}_\p)$ as above, let $T(\mathcal{M}) \subset T_\dimL (\mathcal{O}_\p)$ be the set of all matrices $C \in T_\dimL (\mathcal{O}_\p)$ such that the rows of $C$, interpreted as elements of $\mathcal{L}_\p$ with respect to the basis $(b_1, \dots, b_\dimL)$, constitute a good basis of $\mathcal{M}$.  Analogously to~\cite[Lemma~2.5]{GSS/88}, we find that $T(\mathcal{M})$ is an open subset of $T_\dimL (\mathcal{O}_\p)$ and that 
\begin{equation} \label{equ:measure}
 \mu (T(\mathcal{M})) = (1 - q_\p^{-1})^\dimL q_\p^{- \sum_{i = 1}^\dimL i \lambda_i},
\end{equation}
where
$$ q_\p^{\lambda_i} = [ \mathcal{L}_{\p,i} : (\mathcal{M} \cap \mathcal{L}_{\p,i}) + \mathcal{L}_{\p,i+1}]$$
for all $i \in [ \dimL ]$.  Indeed, given $\mathcal{M}$, we can construct $T(\mathcal{M})$ inductively from the bottom row up.  For any $i \in [ \dimL ]$ the $i$-th row of a matrix $C \in T(\mathcal{M})$ must have a leading term of (additive) valuation $v_\p(c_{ii}) = \lambda_i$, and, if the lower rows have been constructed already, then the $i$-th row is well-defined modulo addition of $\mathcal{O}_\p$-multiples of the lower rows.  In particular, for any $C \in T(\mathcal{M})$, the following holds:
\begin{alignat}{2} \label{equ:integral.rewrite}
[\mathcal{L}_\p : \mathcal{M}] & =  q_{\p}^{\lambda_1 + \cdots + \lambda_\dimL} & = & |c_{11} |_{\p}^{-1} \cdots | c_{\dimL \dimL} |_{\p}^{-1} \\ \nonumber
[Z(\mathcal{L}_\p) : \mathcal{M} \cap Z(\mathcal{L}_\p)] & =  q_{\p}^{\lambda_{\overline{\dimL} + 1} + \cdots + \lambda_\dimL} & = & |c_{\overline{\dimL} + 1, \overline{\dimL} + 1} |_{\p}^{-1} \cdots | c_{\dimL \dimL} |_{\p}^{-1}.
\end{alignat}

If $C \in T_\dimL (K_\p)$, denote the $\ell$-th row of $C$ by $\mathbf{c}_\ell \in \mathcal{L}_\p$.  Then $C \in T_\dimL (\mathcal{O}_\p)$ if and only if $v_\p (c_{ij}) \geq 0$ for all $(i,j) \in [\dimL]^2$, whereas the $\mathcal{O}_\p$-span of the rows of $C$ is an endomorphic image of the Lie lattice $\mathcal{L}_\p$ if and only if there exist $a_{i\ell} \in \mathcal{O}_\p$ such that
$$ \left[ \sum_{\ell = 1}^{\dimL} a_{i \ell} \mathbf{c}_\ell, \sum_{\ell = 1}^{\dimL} a_{j \ell} \mathbf{c}_{\ell} \right] = \sum_{k = 1}^{\dimL} a^\prime_{ijk} \sum_{\ell = 1}^{\dimL} a_{k \ell} \mathbf{c}_{\ell},$$
where $a^\prime_{ijk} \in \Z$ are the structure constants satisfying $[b_i, b_j] = \sum_{k = 1}^{\dimL} a^\prime_{ijk} b_k$.  Clearly these conditions are expressed by an $\mathfrak{L}_{\mathrm{DP}}$-formula $\psi(C)$.  Thus the collection $Y = (Y_F)_{F \in \mathrm{Loc}}$, where $Y_{F} = \left\{ C \in F^{\binom{\dimL + 1}{2}} : F \models \psi(C) \wedge (\det C \neq 0) \right\}$, is an $\mathfrak{L}_{\mathrm{gDP}}$- and $\mathfrak{L}_{\mathrm{DP}}$-definable set.
Observe that $Y_{K_\p} = \bigcup_{\mathcal{M} \in \Lambda(\mathcal{L}_\p)} T(\mathcal{M})$, where the union is disjoint.  Hence by~\eqref{equ:measure} and~\eqref{equ:integral.rewrite} the expression in~\eqref{equ:reinterpretation} may be rewritten as
\begin{multline} \label{equ:unif.integral}
\sum_{\mathcal{M} \in \Lambda(\mathcal{L}_\p)} \int_{T(\mathcal{M})} (\mu(T(\mathcal{M})))^{-1} [\mathcal{L}_\p : \mathcal{M}]^{-s} [Z(\mathcal{L}_\p) : \mathcal{M} \cap Z(\mathcal{L}_\p)]^{d \dimL^\prime} d \mu =
\\
(1 - q_\p^{-1})^{- \dimL} \int_{Y_{K_\p}} \left( \prod_{i = 1}^{\overline{\dimL}} |c_{ii} |_{\p}^{s - i} \right) \left( \prod_{i = \overline{\dimL} + 1}^{\dimL} |c_{ii} |_{\p}^{s - \dimL^\prime d - i} \right) d \mu.
\end{multline}

This integrand is a function of $\mathcal{C}_s$-class, in the terminology of~\cite[\S4.5]{CluckersHalupczok/18}, and the proof of~\cite[Corollary~4.5.2]{CluckersHalupczok/18} applies.  It follows that there exist universal $N \in \N$, $b \in \Z$, and a non-zero $c \in \Q$, and a collection of pairs of integers $(a_i, b_i)$, where $i \in [N]$ and $b_i \neq 0$, such that for every $F \in \mathrm{Loc}^0$ the following holds:
\begin{equation*} \label{equ:cluck.halup}
\frac{1}{(1 - q_F^{-1})^\dimL} \int_{Y_F} \left( \prod_{i = 1}^{\overline{\dimL}} |c_{ii}|_F^{s-i} \right) \left( \prod_{i = \overline{\dimL} + 1}^{\dimL} |c_{ii}|_F^{s - \dimL^\prime d - i} \right) d \mu = 
 \frac{P_F (q_F^{-s})}{q_F^{(b + v_F(c))s} \prod_{i = 1}^N (1 - q_F^{a_i + b_i s})},
\end{equation*}
where $P_F \in \Q[X]$ is a polynomial.  Multiplying the numerator and denominator by $-q_F^{-a_i - b_is}$ for all $i \in [N]$ such that $b_i > 0$, we may assume that $b_i < 0$ for all $i$.
If $F = K_\p$, where $K$ is a number field of degree $d$ and $\p \in V_K$ divides a rational prime $p$ such that $L_p$ is $Z(L_p)$-rigid, then we have seen that the left-hand side is equal to~\eqref{equ:reinterpretation}, which is a power series in $q_F^{-s}$ with constant term $1$.  Therefore, in this case $X^{b + v_F(c)} P_F(X)$ must be a polynomial with constant term $1$.
The first claim of Theorem~\ref{thm:uniformity} follows.

The second part of Theorem~\ref{thm:uniformity} follows from~\eqref{equ:unif.integral} and  Theorem~\ref{thm:bdop}.  
\end{proof}

\begin{rem} \label{rem:variants}
We expect that the results of~\cite[\S4.5]{CluckersHalupczok/18} could likely be extended to integrands of the form~\eqref{equ:unif.integral}, where $d$ is treated as a formal variable as in the proof of Theorem~\ref{thm:bdop}.  
We could then obtain a slightly stronger version of Theorem~\ref{thm:uniformity}(1):  there exists a universal $N \in \N$ and triples $(a_i, b_i, c_i) \in \Z^3$ for $i \in [N]$ such that for all number fields $K$, \emph{all} primes $p$ such that $L_p$ is $Z(L_p)$-rigid, and all primes $\p | p$ of $K$, the fine local factor of $ \zeta^\wedge_{\mathcal{L} \tensor \mathcal{O}_K, p}(s)$ has denominator $\prod_{i = 1}^N (1 - q_\p^{a_i + c_i [K:\Q] + b_i s})$.   A more refined description of the numerator in Theorem~\ref{thm:uniformity}(1) may also be obtained by considering the characterization of the class $\mathcal{C}_s(\ast)$~\cite[\S4.5]{CluckersHalupczok/18}, where $\ast$ is a point.
\end{rem}

\section{Calculations for base extensions of Lie rings} \label{sec:examples}
\subsection{Free nilpotent Lie lattices} \label{sec:free}
Let $\mathcal{F}_{c,g}$ be the free nilpotent Lie lattice over $\Z$ of class $c$ generated by $g$ elements.  The pro-isomorphic zeta functions of $\mathcal{F}_{c,g}$ and its base extensions were determined by Grunewald, Segal, and Smith~\cite[Theorem~7.1]{GSS/88}.  We compute them here as a first illustration of our method, noting that our argument is essentially equivalent to that of~\cite{GSS/88}, although expressed in somewhat different terms.

Fix a $\Z$-basis of $\mathcal{F}_{c,g}$ compatible with a decomposition $F_{c,g} = \mathcal{F}_{c,g} \tensor_\Z \Q = U_1 \oplus \cdots \oplus U_c$ such that $U_i \oplus \cdots \oplus U_c = \gamma_i F_{c,g}$ for all $i \in [c]$; see~\cite[\S 5]{Hall/69} for constructions.  In particular, our chosen basis of $U_1$ is a collection $x_1, \dots, x_g$ of elements of $\mathcal{F}_{c,g}$ that generate it as a Lie lattice.
For every $i \in [c]$, let $m_i$ denote the $\Z$-rank of $\gamma_i \mathcal{F}_{c,g} / \gamma_{i+1} \mathcal{F}_{c,g}$.  By a result of Witt~\cite[Theorem~5.7]{Hall/69}, we have
$$ m_i = \frac{1}{i} \sum_{j | i} \mu(j) g^{i/j},$$
where $\mu$ is the M\"{o}bius function.  Note that $g | m_i$ for all $i \in [c]$.   
Let $E$ be any field of characteristic zero.  By freeness, any $E$-linear map $\varphi : \langle x_1, \dots, x_g \rangle_E \to F_{c,g} \tensor_\Q E$ inducing an isomorphism of vector spaces $\langle x_1, \dots, x_g \rangle_E \simeq (F_{c,g} \tensor_\Q E)/\gamma_2 (F_{c,g} \tensor_\Q E)$ can be extended to an $E$-automorphism $\varphi \in (\aAut F_{c,g})(E)$.  Since we have fixed a basis, we get an embedding $\aAut F_{c,g} \hookrightarrow \mathbf{GL}_{\sum_{i = 1}^c m_i}$ whose image consists of block upper-triangular matrices of the form
$$ \left( \begin{array}{cccc}
A & B_2 & \cdots & B_c \\
& A^{(2)} & \ast & \ast \\
& & \ddots & \vdots \\
& & & A^{(c)} 
\end{array} \right),$$
where $A \in \mathbf{GL}_g$ and $B_i \in \mathbf{M}_{g, m_i}$ for $i \in [2,c]$ are arbitrary and determine all the other blocks.  Every diagonal block $A^{(i)} \in \mathbf{GL}_{m_i}$ depends only on $A$ via the natural map $\aAut F_{c,g} / \gamma_2 F_{c,g} \to \aAut \gamma_i F_{c,g} / \gamma_{i+1} F_{c,g}$, whereas the off-diagonal blocks depend also on $B_2, \dots, B_c$.  Observe that $\det A^{(i)} = (\det A)^{i m_i / g}$ for all $i \in [2,c]$.  Indeed, the map $(A \mapsto \det A^{(i)}) \in \mathrm{Hom}(\mathbf{GL}_g, \mathbf{G}_m)$ must be some power of the determinant.  To find the power, consider an arbitrary diagonal matrix $A = \mathrm{diag}(t_1, \dots, t_g) \in \mathrm{GL}_g(\Q)$.  Now $\gamma_i F_{c,g} / \gamma_{i+1} F_{c,g}$ is spanned by classes representing elements of the form $[x_{j_1}, [x_{j_2}, \cdots, [x_{j_{i-1}}, x_{j_i}] \cdots ]]$.  With respect to such a basis, the induced automorphism of $\gamma_i F_{c,g} / \gamma_{i+1} F_{c,g}$ is given by a diagonal matrix with products $t_{j_1} \cdots t_{j_i}$ on the diagonal.  Thus its determinant is a product of $i m_i$ diagonal elements of $A$, and our claim follows.

For any prime $p$, set $L_p = F_{c,g} \tensor_\Q \Q_p$.  Then $L_p$ is $Z(L_p)$-rigid by~\cite[Theorem~1]{Segal/89}, which is proved by means of Segal's rigidity criterion (quoted here as Corollary~\ref{cor:segal}); indeed,~\cite{Segal/89} appears to have been motivated by this example.  
It is clear from the structure of $\aAut F_{c,g}$ that, for any prime $p$, the assumptions of Section~\ref{sec:assumptions} are satisfied for the decomposition of $L_p$ induced by $F_{c,g} = U_1 \oplus  \cdots \oplus U_c$.  The lifting condition holds because, for any $i \in [2,c-1]$, specifying a class $\overline{g}$ in the notation of Assumption~\ref{lifting.condition} amounts to choosing $A, B_2, \dots, B_i$.  Thus a lifting $\gamma$ is determined by these data and arbitrary matrices $B_{i+1}, \dots, B_c$ with elements in $\Z_p$.  In particular we may take the morphisms $\kappa_i$ of Assumption~\ref{polynomial.lifting} to correspond to the choice of zero matrices for $B_{i+1}, \dots, B_c$.  Hence the hypotheses of Corollary~\ref{cor:rigid.lifting} hold.
We are now in a position to recover~\cite[Theorem~7.1]{GSS/88}.  Recall that $\zeta_q(s) = (1 - q^{-s})^{-1}$.

\begin{thm}[Grunewald-Segal-Smith] \label{pro:zeta.free}
Set $\alpha = \frac{1}{g} \sum_{i = 1}^c im_i$.  Let $d \in \N$ and define $\beta(d) = 2 m_2 + \dots + (c-1) m_{c-1} + dc m_c$.
Let $K$ be any number field of degree $d$.  
Then $\zeta^\wedge_{\mathcal{F}_{c,g} \tensor \mathcal{O}_K}(s) = \prod_{\p \in V_K} W_{\mathcal{F}_{c,g},d}(q_\p, q_\p^{-s})$, where
\begin{equation} \label{equ:local.fcg}
W_{\mathcal{F}_{c,g},d}(X,Y) = \prod_{j = 0}^{g-1} \frac{1}{1 - X^{\beta(d) + j} Y^\alpha}.
\end{equation}
The functional equation
\begin{equation} \label{equ:free.funct.eq}
W_{\mathcal{F}_{c,g},d}(X^{-1}, Y^{-1}) = (-1)^g X^{g \beta(d) + \binom{g}{2}}Y^{g \alpha} W_{\mathcal{F}_{c,g},d}(X,Y)
\end{equation}
holds.  Let $\zeta_K(s)$ denote the Dedekind zeta function of $K$.  Then
\begin{equation*}
\zeta^\wedge_{\mathcal{F}_{c,g} \tensor \mathcal{O}_K}(s) = \prod_{j = 0}^{g-1} \zeta_K(\alpha s - \beta(d) - j);
\end{equation*}
the abscissa of convergence is $\alpha^\wedge_{\mathcal{F}_{c,g} \tensor \mathcal{O}_K} = \frac{\beta(d) + g}{\alpha}$.  
\end{thm}
\begin{proof}
We use the notation of Section~\ref{sec:assumptions} freely.
Clearly an isomorphism $\mathbf{H} \simeq \mathrm{GL}_n$ of algebraic groups is afforded by the map 
$$h = \mathrm{diag}(A, A^{(2)}, \dots, A^{(c)}) \mapsto A.$$  
Observe that $\det h = (\det A)^\alpha$.  Moreover, for any finite extension $F / \Q_p$ and any $i \in [c-1]$, it is easy to see that if $h \in \mathbf{H}^+(F)$, then~\eqref{equ:def.theta} amounts to
\begin{equation*}
\theta_i^F (h) = \mu ( \{ B_{i+1} \in M_{g, m_{i+1}} (F) : B_{i+1} A^{(i+1)} \in M_{g, m_{i+1}}(\mathcal{O}_F) \} ),
\end{equation*}
where $\mu (M_{g, m_{i+1}}(\mathcal{O}_F)) = 1$.  As there are $g$ independent conditions on each of the rows of $B_{i+1}$, we get
$ \theta_i^F (h)  = | \det A^{(i+1)} |_F^{-g} = | \det A |_F^{-(i+1) m_{i+1}}$.
We now have all the ingredients to apply Corollary~\ref{cor:rigid.lifting}:
\begin{equation*}
\zeta^\wedge_{\mathcal{F}_{c,g} \tensor \mathcal{O}_K,p}(s) = \prod_{\p | p} \int_{\mathrm{GL}_g^+ (K_\p)} | \det A |_{\p}^{\alpha s - \beta(d)} d \mu_{\mathrm{GL}_g (K_\p)} (A).
\end{equation*}
Now~\eqref{equ:local.fcg} follows from Example~\ref{exm:abelian}, and the functional equation is clear.  The abscissa of convergence of $\zeta_K(s)$ is $1$, and the claim regarding $\alpha^\wedge_{\mathcal{F}_{c,g} \tensor \mathcal{O}_K}$ follows.
\end{proof}

\begin{rem} \label{rem:graded.conj}
It is clear that $\mathcal{F}_{c,g}$ is naturally graded in the sense of Section~\ref{sec:intro.funct.eq}.  Observe that $\rk_\Z \gamma_i \mathcal{F}_{c,g} = m_i + m_{i+1} + \cdots + m_c$ for every $i \in [c]$.  Thus $\mathrm{wt}(\mathcal{F}_{c,g}) = g \alpha$.  It follows from~\eqref{equ:free.funct.eq} and~\eqref{equ:unram.funct.eq} that all the base extensions $\mathcal{F}_{c,g} \tensor \mathcal{O}_K$ satisfy Conjecture~\ref{graded.conj}; in each case, the finitely many excluded primes are exactly the ones that ramify in $K$. 
One verifies similarly that the Lie lattices considered in Sections~\ref{sec:higher.heisenberg}, \ref{sec:lmn}, and~\ref{sec:max.class} are naturally graded and that Conjecture~\ref{graded.conj} holds for all their base extensions.
\end{rem}

For comparison, we briefly survey what is known about the ideal and subring zeta functions of the Lie lattices $\mathcal{F}_{c,g} \tensor_\Z \mathcal{O}_K$.  In contrast to the pro-isomorphic situation, where {\emph{all}} local factors of $\zeta^\wedge_{\mathcal{F}_{c,g} \tensor \mathcal{O}_K}$ have been computed, for the ``sibling'' zeta functions there are often finitely many local factors that cannot currently be treated.  We will see the same phenomenon for other Lie lattices considered in this section.

If $K$ is an arbitrary number field, $p$ is unramified in $K$, and $g$ is arbitrary, then an explicit 
rational function $W_{g,K,p}(X,Y)$, depending only on $g$ and the decomposition type of $p$ in $K$ and satisfying $\zeta^\vartriangleleft_{\mathcal{F}_{2,g} \tensor \mathcal{O}_K,p}(s) = W_{g,K,p}(p,p^{-s})$, is indicated in~\cite[Section~5.2]{CSV/19} by Carnevale, the third author, and Voll; these rational functions satisfy the generic functional equation of~\cite[Theorem~C]{Voll/10}.  Theorems~2.7 and~2.9 of~\cite{duSWoodward/08} provide explicit rational functions for $\zeta^\vartriangleleft_{\mathcal{F}_{2,2} \tensor \mathcal{O}_K,p}(s)$ when $d \in \{ 2,3 \}$ and $p$ is arbitrary; they are originally due to Grunewald, Segal, and Smith~\cite{GSS/88} and to G.~Taylor~\cite{Taylor/01}.  The only other published result about ramified primes is~\cite[Theorem~3.8]{SV2/16}, where rational functions giving $\zeta^\vartriangleleft_{\mathcal{F}_{2,2} \tensor \mathcal{O}_K,p}(s)$ are computed for $p$ non-split in $K$.  These are found to satisfy a modified functional equation; see also~\cite[Conjecture~1.4]{SV1/15}, where a functional equation for $\zeta^\vartriangleleft_{\mathcal{F}_{2,2} \tensor \mathcal{O}_K,p}(s)$ is conjectured at every prime $p$.  The subring zeta factors $\zeta^\leq_{\mathcal{F}_{2,g} \tensor \mathcal{O}_K,p}(s)$ are known if $g = 2$ and $d \in \{1, 2 \}$, provided $p$ splits in $K$, or if $g = 3$ and $d = 1$; see Theorems~2.3, 2.4, and~2.16 of~\cite{duSWoodward/08}.  The representation zeta functions $\zeta^{\mathrm{irr}}_{\mathbf{F}_{2,g}(\mathcal{O}_K)}$, where $\mathbf{F}_{2,g}$ is the unipotent group scheme associated to $\mathcal{F}_{2,g}$, are known for arbitrary $g$ and $K$~\cite[Theorem~B]{StasinskiVoll/14}; they have a fine Euler decomposition, and the local factors at all $\p \in V_K$ satisfy functional equations with a uniform symmetry factor.

The functions $\zeta^\vartriangleleft_{\mathcal{F}_{3,2},p}(s)$ and $\zeta^\leq_{\mathcal{F}_{3,2},p}(s)$ were computed by Woodward; see~\cite[Theorem~2.35]{duSWoodward/08}.  In addition, $\zeta^{\mathrm{irr}}_{\mathbf{F}_{3,2}(\Z),p}(s)$ appears in~\cite[Table~5.1]{Ezzat/thesis}.  Given an arbitrary pair $(c,g)$, functional equations for almost all local ideal zeta factors $\zeta^\vartriangleleft_{\mathcal{F}_{c,g},p}(s)$ were proved by Voll~\cite[Theorem~4.4]{Voll/17}.  If $K$ is a number field, then functional equations for almost all local subring zeta factors $\zeta^\leq_{\mathcal{F}_{c,g} \tensor \mathcal{O}_K,p}(s)$ follow from the general result of Voll~\cite[Corollary~1.1]{Voll/10} mentioned in the introduction.

\subsection{The higher Heisenberg Lie rings} \label{sec:higher.heisenberg}
Let $m \geq 1$ and consider the Lie lattice $\mathcal{H}_m$ with presentation
\begin{equation*}
\langle x_1, \dots, x_m, y_1, \dots, y_m, z | [x_i, y_i] = z, 1 \leq i \leq m \rangle_{\Z},
\end{equation*}
with the usual convention that all other pairs of generators commute.
This is a central amalgamation of $m$ Heisenberg Lie lattices.  Note that $Z(\mathcal{H}_m) = [\mathcal{H}_m, \mathcal{H}_m] = \langle z \rangle$.  Write $H_m$ for the $\Q$-Lie algebra $\mathcal{H}_m \tensor_\Z \Q$.

\begin{pro} \label{pro:hh.rigid}
Let $k$ be any field.  The $k$-Lie algebra $\mathcal{H}_m \tensor_\Z k$ is $Z(\mathcal{H}_m \tensor_\Z k)$-rigid.
\end{pro}
\begin{proof}
It is easy to check that $x_i, y_i \in \mathcal{Y}(Z(\mathcal{H}_m \tensor_\Z k))$ for all $i \in [m]$.  Therefore $\mathcal{Y}(Z(\mathcal{H}_m \tensor_\Z k))$ generates $\mathcal{H}_m \tensor_\Z k$ as a $k$-Lie algebra.
Furthermore, $\mathcal{H}_m \tensor_\Z k$ is absolutely indecomposable.  Indeed, for any field $E/k$, let $\mathcal{H}_m \tensor_\Z E = M_1 \oplus M_2$ be a decomposition.  At least one of the components, say $M_1$, must be non-abelian.  The derived subalgebra of $\mathcal{H}_m \tensor_\Z E$, which is equal to its center, is one-dimensional over $E$ and so contained in $M_1$.  Thus $M_2$ is abelian.  Hence $M_2 \subseteq Z(\mathcal{H}_m \tensor_\Z E) \subseteq M_1$, and we conclude that $M_2 = 0$.  Then $\mathcal{H}_m \tensor_\Z k$ is $Z(\mathcal{H}_m \tensor_\Z k)$-rigid by Theorem~\ref{thm:rigidity}.
\end{proof}

Before proceeding we note that if $m \geq 2$, then the results of~\cite{Segal/89} (recalled here as Corollary~\ref{cor:segal}) are insufficient to prove the rigidity of $\mathcal{H}_m \tensor_\Z k$; we genuinely need the more general Theorem~\ref{thm:rigidity}.  It is an exercise to show that there is no highly invariant ideal (or, indeed, any ideal) $M \leq \gamma_2 (\mathcal{H}_m \tensor_\Z k)$ such that $\mathcal{X}(M)$ generates $\mathcal{H}_m \tensor_\Z k$.

It is useful to consider an equivalent presentation of $H_m$, following du Sautoy and Lubotzky~\cite[Section 3.3]{duSLubotzky/96}.  Consider the $\Q$-vector space $V_m = \Q^{2m} \times \Q$ with Lie bracket given by
\begin{equation*}
[ (v_1, w_1), (v_2, w_2)] = (0, v_1 \Omega v_2^T)
\end{equation*}
for all $v_1, v_2 \in \Q^{2m}$ and $w_1, w_2 \in \Q$, where $\Omega$ is the symplectic form 
\begin{equation*}
\Omega = \left( \begin{array}{cc} 0 & I_m \\ - I_m & 0 \end{array} \right).
\end{equation*}
Here $I_m$ denotes the $m \times m$ identity matrix.  The map $V_m \to H_m$ of vector spaces that sends the standard bases of $\Q^{2m}$ and $\Q$ to $(x_1, \dots, x_m, y_1, \dots, y_m)$ and $z$, respectively, is an isomorphism of Lie algebras.  Fixing the basis $(x_1, \dots, x_m, y_1, \dots, y_m, z)$ provides an embedding $\aAut H_m \hookrightarrow \mathbf{GL}_{2m + 1}$.
As in~\cite[Lemma 3.12]{duSLubotzky/96}, we find that $\aAut H_m$ is the algebraic subgroup of $\mathbf{GL}_{2m+1}$ consisting of the matrices of the form
\begin{equation*} \label{equ:amalgam}
\left( \begin{array}{cc} A & B \\ 0 & \lambda \end{array} \right),
\end{equation*}
where $\lambda \in \mathbf{GL}_1$, the matrix $A \in \mathbf{GL}_{2m}$ satisfies $A \Omega A^T = \lambda \Omega$, and $B$ is an arbitrary $2m \times 1$ matrix.  Comparing the determinants of the two sides of the relation $A \Omega A^T = \lambda \Omega$ (which is equivalent to $A \in \mathbf{GSp}_{2m}$), we observe that $(\det A)^2 = \lambda^{2m}$.  
Let $K$ be a number field of degree $d = [K:\Q]$ and let $p$ be a prime.  For every $\p | p$, set $\mathrm{GSp}_{2m}^+(K_\p) = \mathrm{GSp}_{2m}(K_\p) \cap M_{2m}(\mathcal{O}_{\p})$, and consider the right Haar measure on $\mathrm{GSp}_{2m}(K_\p)$ normalized so that $\mu_{\mathrm{GSp}_{2m}(K_\p)}(\mathrm{GSp}_{2m}(\mathcal{O}_{\p})) = 1$. 
\begin{pro} \label{thm:gsp}
Let $m \geq 1$.  Let $K$ be a number field and $p$ be any prime.  Then
\begin{equation*}
\zeta^\wedge_{\mathcal{H}_m \tensor \mathcal{O}_K,p}(s) = \prod_{\p | p} \int_{\mathrm{GSp}_{2m}^+(K_\p)} | \det A |_{\p}^{(1 + 1/m)s - 2d} d \mu_{\mathrm{GSp}_{2m}(K_\p)}(A).
\end{equation*}
\end{pro}
\begin{proof}
By the rigidity established in Proposition~\ref{pro:hh.rigid}, we may use Proposition~\ref{pro:rigid.computation}.  From the description of $\aAut H_m$ given above, we see, for all $\p | p$, that $\mathbf{H}(K_\p) \simeq \mathrm{GSp}_{2m}(K_\p)$ consists of matrices of the form
$$
h = \left( \begin{array}{cc} A & 0 \\ 0 & \lambda \end{array} \right),$$
where $A \in \mathrm{GSp}_{2m}(K_\p)$.  Since $\mathbf{N}(K_\p)$ acts trivially on $\langle x_1, \dots, x_m, y_1, \dots, y_m \rangle_{K_\p}$, we have $\widetilde{\theta}^{K_\p}(h) = 1$ for all $h \in \mathbf{H}^+(K_\p)$ by Remark~\ref{rem:trivial.theta}.  Finally, $| \det h |_{\p} = | \det A |^{1 + 1/m}_{\p}$ and $| \det \varepsilon_\p(h) |_{\p} = | \lambda |_{\p} = | \det A |_{\p}^{1/m}$.  Since $\overdim = 2m$, our claim now follows from Proposition~\ref{pro:rigid.computation}.
\end{proof}

Note that in the case $d = 1$ this result matches the formula obtained in~\cite[Proposition 3.14]{Berman/05} and differs from the one obtained in~\cite[Lemma 3.14]{duSLubotzky/96} because of an error in the computation of the function $\theta(h)$ appearing in~\cite{duSLubotzky/96}.  

The integral of Proposition~\ref{thm:gsp} may be computed using Proposition~\ref{pro:bruhat.decomposition}.  We recall some basic facts about the root system $\mathbf{C}_n$ of the reductive group $\mathbf{GSp}_{2m}$ over a field $F$.  The subset of diagonal elements of $\mathbf{GSp}_{2m}$ is
$$\mathbf{T} = \left\{ \mathrm{diag} ( \nu t_1^{-1}, \nu t_2^{-1} \dots, \nu t_m^{-1}, t_1, t_2, \dots, t_m) : \nu, t_1, \dots t_m \in \mathbf{G}_m \right\} $$
and is a split maximal torus of rank $m+1$.  Consider the cocharacters
\begin{equation*}
\xi_k(t) = \begin{cases}
\mathrm{diag} (\underbrace{1, \dots, 1}_{k}, \underbrace{t, \dots, t}_{m - k}, \underbrace{1, \dots, 1}_k, \underbrace{t^{-1}, \dots, t^{-1}}_{m-k} ) &: k \in [m-1]_0 \\
\mathrm{diag} (\underbrace{t, \dots, t}_{m}, \underbrace{1, \dots, 1}_{m} ) &: k = m.
\end{cases}
\end{equation*}
They constitute a $\Z$-basis of $\Xi = \Hom(\mathbf{G}_m, \mathbf{T})$.  Indeed, any element $\xi \in \Xi$ has the form
\begin{equation} \label{equ:xiform}
\xi(t) = \mathrm{diag}(t^{\lambda - a_1}, t^{\lambda - a_2}, \dots, t^{\lambda - a_m}, t^{a_1}, t^{a_2}, \dots, t^{a_m} ),
\end{equation}
and one readily verifies that
\begin{equation} \label{equ:xibasis}
\xi = \xi_0^{-a_1} \xi_1^{a_1 - a_2} \dots \xi_{m-2}^{a_{m-2} - a_{m-1}} \xi_{m-1}^{a_{m-1} - a_{m}} \xi_m^{\lambda}.
\end{equation}
Thus $\{ \xi_0, \dots, \xi_m \}$ spans the $\Z$-module $\Xi$ of rank $m+1$.

Let $\mathrm{Sym}_m \subset M_m$ be the set of symmetric $m \times m$ matrices, and let $\mathbf{U}_m \subset \mathbf{GL}_m$ be the algebraic subgroup of upper triangular matrices.  It is well known that $\mathbf{Sp}_{2m}$ has the same root system as $\mathbf{GSp}_{2m}$, and that
$$ (\mathbf{T} \cap \mathbf{Sp}_{2m}) \ltimes \left\{ \left( \begin{array}{cc} (U^T)^{-1} & MU \\ 0 & U \end{array} \right) : U \in \mathbf{U}_m, M \in \mathrm{Sym}_n \right\}$$
is a Borel subgroup of $\mathbf{Sp}_{2m}$.  Let $\Phi^+$ be the set of $m^2$ positive roots given by this choice of Borel subgroup.  These consist of the $\binom{m}{2}$ roots $e_i - e_j$, for $1 \leq i < j \leq m$, associated to the root subgroups whose $F$-points are $I_{2m} + FE_{m+i,m+j}$, and $\binom{m+1}{2}$ roots $-e_i - e_j$, for $1 \leq i \leq j \leq m$, associated to the subgroups $I_{2m} + F(E_{i,m+j} + E_{j,m+i})$ when $i \neq j$, and $I_{2m} + FE_{i,m+i}$ when $i = j$.  Here $E_{i,j}$ is the usual notation for elementary matrices.  We fix the collection $\Delta = \{ \alpha_0, \alpha_1, \dots, \alpha_{m-1} \}$ of simple roots, where $\alpha_0 = -2e_1$ and $\alpha_i = e_i - e_{i+1}$ for $i \in [m-1]$.  For $\xi \in \Xi$ as in~\eqref{equ:xiform}, one verifies that $\langle e_i - e_j, \xi \rangle = a_i - a_j$ and $\langle -e_i - e_j, \xi \rangle = \lambda - (a_i + a_j)$, for $i,j$ in the suitable ranges as above.  In particular,
\begin{equation} \label{equ:root.pairing}
\langle \prod_{\beta \in \Phi^+} \beta, \xi_k \rangle =
\begin{cases}
2 \left( \binom{m+1}{2} - \binom{k+1}{2} \right) &: k \in [m-1]_0 \\
\binom{m+1}{2} &: k = m.
\end{cases}
\end{equation}
Let $F/\Q_p$ be a finite extension, fix a uniformizer $\pi \in \mathcal{O}_F$, and set $q = | \mathcal{O}_F / (\pi) |$.  Then the condition that $\xi \in \Xi^+$ and $\alpha(\xi(\pi)) \in \mathcal{O}_K$ for all $\alpha \in \Delta$ is equivalent to the conditions $\lambda \geq 2a_1$ and $a_1 \geq a_2 \geq \cdots \geq a_n \geq 0$.  

The Weyl group in this setting is naturally identified with the hyperoctahedral group $B_m$ of signed permutations; see~\cite[Section~8.1]{BjoernerBrenti/05} for an introduction to this group.  Elements $w \in B_m$ may be viewed as permutations of the set $\{ 1, \dots, m \} \cup \{ -1, \dots, -m \}$ satisfying $w(-j) = -w(j)$ for all $j \in [m]$.  We write $w \in B_m$ in the ``window notation'' $w = [w(1), \dots, w(m)]$.  Our choice of simple roots gives rise to the Coxeter generating set $S = \{s_0, s_1, \dots, s_{m-1} \}$ of $B_m$, where $s_0$ transposes $1$ and $-1$ and $s_i$ transposes $i$ and $i+1$ for $i \in [m-1]$.  Note that these are the Coxeter generators considered in~\cite{BjoernerBrenti/05}.  Let $\ell$ be the associated length function.  For every $w \in B_m$ we have the descent set
$$ \mathrm{Des}(w) = \{ i \in [m-1]_0 : \ell(w s_i) < \ell(w) \}.$$
To compute $\mathrm{Des}(w)$ in practice, observe that if $i \in [m-1]$, then $i \in \mathrm{Des}(w)$ if and only if $w(i) > w(i + 1)$.  Also, $0 \in \mathrm{Des}(w)$ if and only if $w(1) < 0$.  For $w \in B_m$, set
\begin{align*}
\mathrm{Inv}(w) & =  \{ (i,j) \in [m]^2 : i < j , w(i) > w(j) \} \\
\mathrm{Npr}(w) & =  \{ (i,j) \in [m]^2 : i \leq j , w(i) + w(j) < 0 \}
\end{align*}
and denote $\mathrm{inv}(w) = | \mathrm{Inv}(w) |$ and $\mathrm{npr}(w) = | \mathrm{Npr}(w) |$.  We refer to the elements of $\mathrm{Inv}(w)$ and of $\mathrm{Npr}(w)$ as {\emph{inversions}} and {\emph{negative pairs}} of $w$, respectively.  The length of $w$ is given by $\ell (w) = \mathrm{inv}(w) + \mathrm{npr}(w)$; cf.~\cite[Proposition~8.1.1]{BjoernerBrenti/05}.  Write $\mathrm{des}(w) = | \mathrm{Des}(w) |$ for the descent number of $w$.  Set $\varepsilon_1(w) = 1$ if $w(1) < 0$ and $\varepsilon_1(w) = 0$ otherwise.  Define the following statistic for $w \in B_m$:
$$ \sigma_C(w) = \binom{m+1}{2} \varepsilon_1(w) + \sum_{i \in \mathrm{Des}(w) \cap [m-1]} (m-i)(m+i+1).$$

\begin{rem} \label{rmk:permutation.descent}
We identify $S_m$ with the subgroup of $B_m$ consisting of signed permutations $w$ such that $w(i) > 0$ for all $i \in [m]$.  The Coxeter length function on $S_m$ with respect to the generators $\{ s_1, \dots, s_{m-1} \}$ coincides with $\ell$, and the usual descent set $\{ i \in [m-1] : \sigma(i) > \sigma(i+1) \}$ of a permutation $\sigma \in S_m$ coincides with $\mathrm{Des}(\sigma)$ as defined above.  
\end{rem}

We will also need two statistics for permutations $\sigma \in S_m$:
$$ \begin{array}{lcr}
\displaystyle{\sigma_A(\sigma) = \sum_{i \in \mathrm{Des}(\sigma)} i(m-i)} & & \displaystyle{\mathrm{rbin}(\sigma) = \sum_{i \in \mathrm{Des}(\sigma)} \binom{m-i+1}{2}.} \end{array} $$
It is readily checked that $\sigma_A$ and $\sigma_C$ are special cases of the statistic $\sigma$ on Weyl groups defined in~\cite{StembridgeWaugh/98}, whereas $\mathrm{rbin}$ is a variant of the statistic $\mathrm{bin}$ introduced in~\cite{BrightSavage/10}.

Recalling the notation of~\eqref{equ:xiform}, we put $r_0 = \lambda - 2a_1$ and $r_i = a_i - a_{i+1}$ for $i \in [m-1]$; in addition, we set $r_m = a_m$.  It is easy to see that
the set $w \Xi^+_w$ defined in~\eqref{equ:wxi.dfn} is parametrized by the conditions $r_i \geq 1$ if $i \in \mathrm{Des}(w)$ and $r_i \geq 0$ otherwise.  Moreover, given $w \in B_n$ and $\xi \in \Xi$, it follows from~\eqref{equ:xibasis} and~\eqref{equ:root.pairing} that
\begin{equation*}
q^{\langle \prod_{\beta \in \Phi^+} \beta, \xi \rangle} | \det \xi(\pi) |^s_F = \left( q^{\binom{m+1}{2} - ms} \right)^{r_0} \prod_{i = 1}^m \left( q^{2 \left( \binom{m+1}{2} - \binom{i+1}{2} \right) - 2ms} \right)^{r_i}.
\end{equation*}
Hence by Proposition~\ref{pro:bruhat.decomposition} we have
\begin{equation} \label{equ:integral.bn}
\int_{\mathrm{GSp}_{2m}^+(F)} | \det A |_{F}^{s} d \mu(A) = 
\frac{\sum_{w \in B_m} q^{- \ell(w)} \prod_{i \in \mathrm{Des}(w)} \widetilde{X}_i}{\prod_{i = 0}^m (1 - \widetilde{X}_i)},
\end{equation}
where $\widetilde{X}_0 = q^{\binom{m+1}{2} - ms}$ and $\widetilde{X}_i = q^{2 \left( \binom{m+1}{2} - \binom{i+1}{2} \right) - 2ms}$ for $i \in [m]$.

\begin{rem}
The integral of~\eqref{equ:integral.bn} was studied by Satake, who computed it explicitly in the case $m = 2$; see~(21) of~\cite[Appendix I]{Satake/63}.  This computation, in terms of spherical functions, was completed by Macdonald~\cite[V.(5.4)]{Macdonald/95} for arbitrary $m$.  In contrast to the formula of~\cite{Macdonald/95}, the right-hand side of~\eqref{equ:integral.bn} provides an expression from which one may readily deduce a functional equation; see Remark~\ref{rmk:higher.heisenberg} below.  
\end{rem}

It turns out that cancellation occurs in the rational function
of~\eqref{equ:integral.bn}, so that the sum over $2^m \cdot m!$ elements of $B_m$ in its numerator reduces to a sum over $m!$ elements of the symmetric group $S_m$.  This will follow from the following identity.  Similar phenomena are observed in~\cite[Proposition~1.7]{StasinskiVoll/14} and~\cite[Proposition~3.4]{CSV/18}; it would be interesting to obtain a theorem on Weyl group statistics admitting all three identities as special cases.

\begin{lem} \label{lem:bm.to.sm}
The following identity of polynomials holds in $\Q[X,Y]$:
\begin{equation*}
\sum_{w \in B_m} X^{(\sigma_C - \ell)(w)} Y^{(2 \,\mathrm{des} - \varepsilon_1)(w)} = \! \left( \prod_{j = 1}^m (1 + X^{\binom{m+1}{2} - \binom{j+1}{2}} Y) \right) \! \sum_{\sigma \in S_m} X^{(\sigma_A - \ell + \mathrm{rbin})(\sigma)} Y^{\mathrm{des}(\sigma)}.
\end{equation*}
\end{lem}
\begin{proof}
We start by defining, for every $j \in [m]$, an involution $\eta_j : B_m \to B_m$.  For every $w \in B_m$, consider the set $S_j(w) = \{ |w(k)| : k \in [j] \} \subseteq [m]$.  Let $c_1 < \cdots < c_j$ be the elements of $S_j(w)$, arranged in increasing order, and consider the signed permutation $w_j \in B_m$ determined by
\begin{equation*}
w_j(x) = \begin{cases}
- c_{j+1-k} &: x = c_k \, \text{for some} \, k \in [j] \\
x &: x \in [m] \setminus S_j(w). \end{cases}
\end{equation*}
Now set $\eta_j (w) = w_j \circ w$.  Informally, to construct $\eta_j(w)$ we take the $j$ leftmost entries in the window notation for $w$, ignoring their signs, and replace the largest of these with the smallest one, the second largest with the second smallest, and so forth.  Then we arrange signs so that $(\eta_j(w))(k)$ has the opposite sign from $w(k)$, for all $k \in [j]$.
For example, if $w = [3, -5, -1, 6, 2, 7, -4] \in B_7$ in window notation, then $\eta_5(w) = [-3, 2, 6, -1, -5, 7, -4]$.

It is easy to see that $\eta_j(\eta_j(w)) = w$ and that $\eta_j \circ \eta_i = \eta_i \circ \eta_j$ for any $i,j \in [m]$.  Thus we obtain an action on the set $B_m$ of the group $\Gamma \simeq (\Z / 2\Z)^m$ generated by $\eta_1, \dots, \eta_m$.  Henceforth write $\eta_j w$ instead of $\eta_j(w)$ for brevity.  Consider a subset $J \subseteq [m]$ and put
$$ J^\prime = \{ j \in [m-1] : j \not\in J, j + 1 \in J \} \cup \{ j \in [m] : j \in J, j + 1 \not\in J \},$$
where $m + 1 \not\in J$ for all $J$.  Setting $\gamma = \prod_{j \in J^\prime} \eta_j \in  \Gamma$, we observe, for every $w \in B_m$ and every $j \in [m]$, that $(\gamma w)(j)$ and $w(j)$ have opposite signs if $j \in J$ and the same sign otherwise.  Hence every orbit of our action contains exactly $| \Gamma | = 2^m$ elements, each with a different arrangement of signs.  In particular, every orbit contains a unique element of $S_m$.  Let $\mathcal{C}_\sigma \in \Gamma \backslash B_m$ denote the orbit containing $\sigma \in S_m$.  To establish our claim, it clearly suffices to prove the following for every $\sigma \in S_m$:
\begin{equation} \label{equ:orbit.equation}
\sum_{w \in \mathcal{C}_\sigma} X^{(\sigma_C - \ell)(w)} Y^{(2 \,\mathrm{des} - \varepsilon_1)(w)} \! = \!\left( \prod_{j = 1}^m (1 + X^{\binom{m+1}{2} - \binom{j+1}{2}} Y)\! \right) \! X^{(\sigma_A - \ell + \mathrm{rbin})(\sigma)} Y^{\mathrm{des}(\sigma)}.
\end{equation}
If $a,b,c \in \Z$, then we say that $c$ lies between $a$ and $b$ if either $a < c < b$ or $b < c < a$.
The remainder of the argument makes use of the following definition.
\begin{dfn} \label{def:property.pj}
Let $j \in [m-1]$.  We say that $w \in B_m$ satisfies property $(P_j)$ if the following statement is true:
$$ w(j) < 0  \textit{ if and only if } w(j+1) \textit{ lies between } w(j) \textit{ and } \eta_jw (j).$$
We say that $w$ satisfies property $(P_m)$ if $w(m) > 0$.  
\end{dfn}
It is clear, for every $j \in [m]$ and $w \in B_m$, that exactly one member of the pair $\{ w, \eta_jw \}$ satisfies property $(P_j)$.  The following statement, whose proof is very technical, is crucial to our argument.

\begin{sublemma} \label{bm.claim}
Suppose that $j \in [m]$ and that $w \in B_m$ satisfies property $(P_j)$.  Then
\begin{equation*}
X^{(\sigma_C - \ell)(\eta_j w)} Y^{(2 \,\mathrm{des} - \varepsilon_1)(\eta_j w)} =  \left(X^{\binom{m+1}{2} - \binom{j+1}{2}} Y \right) X^{(\sigma_C - \ell)(w)} Y^{(2 \,\mathrm{des} - \varepsilon_1)(w)}.
\end{equation*}
\end{sublemma}

We assume Sublemma~\ref{bm.claim} for the moment and show how to deduce Lemma~\ref{lem:bm.to.sm}.  It is obvious from the definitions that the property $(P_j)$, as well as its negation, are preserved by $\eta_k$ for all $k < j$.  This implies, for every $\sigma \in S_m$, that the orbit $\mathcal{C}_\sigma$ contains an element $w_\sigma$ having the property $(P_j)$ for all $j \in [m]$.  Indeed, we choose an arbitrary $w \in \mathcal{C}_\sigma$ and construct $w_\sigma$ recursively as follows: 
$w_\sigma = \eta_1^{\delta_1} \cdots \eta_m^{\delta_m} w,$ 
where $\delta_m = 0$ if $w$ has property $(P_m)$ and $\delta_m = 1$ otherwise.  Similarly, for $k \geq 1$, we take $\delta_{m-k} = 0$ if $\left( \prod_{i = 0}^{k-1} \eta_{m-i}^{\delta_{m-i}} \right) w$ has property $(P_{m-k})$ and $\delta_{m-k} = 1$ otherwise.  Once we know that such an element $w_\sigma$ exists, it is readily seen that for every $J \subseteq [m]$ there is a unique element of $\mathcal{C}_\sigma$ satisfying $(P_j)$ exactly when $j \in J$, namely $(\prod_{j \not\in J} \eta_j) w_\sigma$.

Let $\gamma \in \Gamma$.  Then $\gamma = \prod_{j \in J} \eta_j$ for some $J \subseteq [m]$.  Applying Sublemma~\ref{bm.claim} to the elements of $J$ in increasing order, we find that 
\begin{equation} \label{equ:gamma.action.bm}
X^{(\sigma_C - \ell)(\gamma w_\sigma)} Y^{(2 \,\mathrm{des} - \varepsilon_1)(\gamma w_\sigma)} = \left( \prod_{j \in J} X^{\binom{m+1}{2} - \binom{j+1}{2}} Y \right) X^{(\sigma_C - \ell)(w_\sigma)} Y^{(2 \,\mathrm{des} - \varepsilon_1)(w_\sigma)}.
\end{equation}
Hence the left-hand side of~\eqref{equ:orbit.equation} is equal to
\begin{equation*} \label{equ:lhs.bm}
\left( \prod_{j =1}^m ( 1 + X^{\binom{m+1}{2} - \binom{j+1}{2}} Y) \right) X^{(\sigma_C - \ell)(w_\sigma)} Y^{(2 \,\mathrm{des} - \varepsilon_1)(w_\sigma)}.
\end{equation*}
Thus, to establish~\eqref{equ:orbit.equation} it suffices to prove, for every $\sigma \in S_m$, that
\begin{equation} \label{equ:endgame}
X^{(\sigma_C - \ell)(w_\sigma)} Y^{(2 \mathrm{des} - \varepsilon_1)(w_\sigma)} = X^{(\sigma_A - \ell + \mathrm{rbin})(\sigma)} Y^{\mathrm{des}(\sigma)}.
\end{equation}

The permutation $\sigma \in S_m$, viewed as an element of $B_m$, clearly satisfies property $(P_m)$.  If $j \in [m-1]$, then $\sigma$ fails to satisfy $(P_j)$ if and only if $\sigma(j+1)$ lies between $\sigma(j)$ and $\eta_j \sigma(j) < 0$, which is equivalent to $j \in \mathrm{Des}(\sigma)$.  Hence $\sigma = \left( \prod_{j \in \mathrm{Des}(\sigma)} \eta_j \right) (w_\sigma)$.  Noting that $\varepsilon_1(\sigma) = 0$, we obtain
\begin{multline*}
X^{(\sigma_C - \ell)(w_\sigma)} Y^{(2 \mathrm{des} - \varepsilon_1)(w_\sigma)} = \\
X^{(\sigma_C - \ell)(\sigma)} Y^{2 \mathrm{des} (\sigma)} \left( \prod_{j \in \mathrm{Des}(\sigma)} X^{ - \left( \binom{m+1}{2} - \binom{j+1}{2} \right)}Y^{-1} \right) = 
X^{(\sigma_A - \ell + \mathrm{rbin})(\sigma)} Y^{\mathrm{des}(\sigma)},
\end{multline*}
where the first equality is immediate from~\eqref{equ:gamma.action.bm} and the second arises from a simple calculation.  This verifies~\eqref{equ:endgame}, and Lemma~\ref{lem:bm.to.sm} follows.

It remains to prove Sublemma~\ref{bm.claim}, which is done by a computation.  Let $j \in [m]$ and assume that $w$ satisfies property $(P_j)$.  We first claim that
\begin{equation} \label{equ:sublemma.length}
\ell(\eta_j w) = \ell (w) + \sum_{i \in [j] \atop w(i) > 0} i - \sum_{i \in [j] \atop w(i) < 0} i.
\end{equation} 
Consider pairs $(i_1, i_2) \in [j]^2$ with $i_1 \leq i_2$.  Observe that $|\eta_j w(i_1)| > |\eta_j w(i_2)|$ if and only if $| w(i_1)| < |w(i_2)|$.  Moreover, $\eta_j w(i)$ and $w(i)$ have opposite sign for all $i \in [j]$.  It follows that if $w(i_1)$ and $w(i_2)$ have the same sign, then $(i_1, i_2)$ is an inversion of $\eta_j w$ if and only if it was already an inversion of $w$.  Similarly, if $w(i_1)$ and $w(i_2)$ have opposite sign, then $(i_1, i_2)$ is a negative pair of $\eta_j w$ if and only if it already was a negative pair of $w$.  However:
\begin{itemize}
\item If $w(i_1) < 0$ and $w(i_2) > 0$, then $(i_1, i_2)$ is an inversion of $\eta_j w$ but not of $w$.
\item If $w(i_1) > 0$ and $w(i_2) > 0$, then $(i_1, i_2)$ is a negative pair of $\eta_j w$ but not of $w$.
\item If $w(i_1) > 0$ and $w(i_2) < 0$, then $(i_1, i_2)$ is an inversion of $w$ but not of $\eta_j w$.
\item If $w(i_1) < 0$ and $w(i_2) < 0$, then $(i_1, i_2)$ is a negative pair of $w$ but not of $\eta_j w$.
\end{itemize}
It follows that the contribution to $\ell(\eta_j w) - \ell(w) = \mathrm{inv}(\eta_j w) + \mathrm{npr}(\eta_j w) - \mathrm{inv}(w) - \mathrm{npr}(w)$ arising from pairs $(i_1, i_2) \in [j]^2$ is
\begin{align*}
| \{ (i_1, i_2) \in [j]^2: i_1 \leq i_2, w(i_2) > 0 \} | - | \{ (i_1, i_2) \in [j]^2: i_1 \leq i_2, w(i_2) < 0 \} | = \\\sum_{i \in [j] \atop w(i) > 0} i - \sum_{i \in [j] \atop w(i) < 0} i.
\end{align*}

Next we consider pairs $(i_1, i_2)$ such that $i_1 \leq j$ and $i_2 > j$.  For every $i \in [j]$ we denote by $i^\prime$ the unique element of $[j]$ such that $|\eta_j w(i^\prime)| = |w(i)|$.  If $| w(i_2) | < |w(i_1)|$ and if both of the conditions $w(i_1) < 0$ and $-w(i_1) \in \eta_j w([j])$ hold, then $(i_1, i_2)$ is not an inversion of $w$ yet gives rise to an inversion $(i_1^\prime, i_2)$ of $\eta_j w$.  On the other hand, $(i_1, i_2)$ is a negative pair of $w$, while $(i_1^\prime, i_2)$ is not a negative pair of $\eta_j w$.  Analogously, if $w(i_1) > 0$ and $| w(i_2) | < |w(i_1)|$, then $(i_1, i_2)$ is an inversion and not a negative pair of $w$, while $(i_1^\prime, i_2)$ is a negative pair and not an inversion of $\eta_j w$.  In all other cases, $(i_1^\prime, i_2)$ is an inversion of $\eta_j w$ if and only if $(i_1, i_2)$ is an inversion of $w$, and the same is true for negative pairs.  Thus, the pairs $(i_1, i_2)$ such that $i_1 \leq j$ and $i_2 > j$ make no contribution to $\ell(\eta_j w) - \ell(w)$.  Since the pairs $(i_1, i_2)$ such that $j < i_1 \leq i_2$ are obviously unaffected by $\eta_j$, we have established~\eqref{equ:sublemma.length}.

Now consider the descent sets of $\eta_j w$ and of $w$.  It follows from the observations above that:
\begin{itemize}
\item If $i \in D_- = \{ i \in [j-1]: w(i) > 0, w(i+1) < 0 \}$, then $i \in \mathrm{Des}(w)$ and $i \not\in \mathrm{Des}(\eta_j w)$.  
\item If $i \in D_+ = \{ i \in [j-1]: w(i) < 0, w(i+1) > 0 \}$, then $i \not\in \mathrm{Des}(w)$ and $i \in \mathrm{Des}(\eta_j w)$.  
\item If $i \in [j-1] \setminus (D_+ \cup D_-)$ or $i > j$, then $i \in \mathrm{Des}(w)$ if and only if $i \in \mathrm{Des}(\eta_j w)$.
\end{itemize}
The remaining case $i = j$ will be treated below.  Let $\beta_1 < \cdots < \beta_s$ be the elements of $D_-$ and $\gamma_1 < \cdots < \gamma_t$ the elements of $D_+$.

Suppose first that $w(1) > 0$.  In this case, $0$ is a descent of $\eta_j w$ but not of $w$.  If $w(j) > 0$, then $s = t$ and we have the arrangement
$ 1 \leq \beta_1 < \gamma_1 < \cdots < \beta_s < \gamma_s \leq j - 1$.
By property $(P_j)$ of Definition~\ref{def:property.pj}, observe that if $j \in [m-1]$, then $j \leq m - 1$ and $w(j+1)$ does not lie between $w(j)$ and $\eta_j w (j)$.  Hence $j \in \mathrm{Des}(w)$ if and only if $j \in \mathrm{Des}(\eta_j w)$.  
Letting $\Sigma[a,b]$ denote the sum of the elements of the set $[a,b]$, it follows from~\eqref{equ:sublemma.length} that
\allowdisplaybreaks{
\begin{multline*} 
\ell(\eta_j w) - \ell(w) = \sum_{i \in [j] \atop w(i) > 0} i - \sum_{i \in [j] \atop w(i) < 0} i = \\\Sigma[1,\beta_1] - \Sigma[\beta_1 + 1, \gamma_1] + \cdots - \Sigma[\beta_{s} + 1, \gamma_s] + \Sigma[\gamma_s + 1, j]  = \\ 
\binom{\beta_1 + 1}{2} - \left( \binom{\gamma_1 + 1}{2} - \binom{\beta_1 + 1}{2} \right) + \cdots + \left( \binom{j+1}{2} - \binom{\gamma_s + 1}{2} \right) = \\  
\sigma_C(\eta_j w) - \sigma_C(w) + \binom{j+1}{2} - \binom{m+1}{2}.
\end{multline*} }
Sublemma~\ref{bm.claim} follows for this case by a simple calculation.
If $w(j) < 0$, then $t = s - 1$ and $1 \leq \beta_1 < \gamma_1 < \dots < \gamma_{s-1} < \beta_s \leq j - 1$.  By property $(P_j)$, we see that $j \in [m-1]$ and that $j$ is a descent of $\eta_j w$ but not of $w$.  In this case,
\begin{multline*}
\ell(\eta_j w) - \ell(w) = \sum_{i \in [j] \atop w(i) > 0} i - \sum_{i \in [j] \atop w(i) < 0} i = - \binom{j+1}{2} + 2 \sum_{i = 1}^s \binom{\beta_i + 1}{2} - 2 \sum_{i = 1}^{s-1} \binom{\gamma_i + 1}{2} = \\
\sigma_C(\eta_j w) - \sigma_C(w) + \binom{j+1}{2} - \binom{m+1}{2}.
\end{multline*}
Again Sublemma~\ref{bm.claim} follows.
The case $w(1) < 0$ is treated analogously, completing the proof of Sublemma~\ref{bm.claim} and thus of Lemma~\ref{lem:bm.to.sm}.
\end{proof}

After this combinatorial digression, we return to the computation of the pro-isomorphic zeta function $\zeta^\wedge_{\mathcal{H}_m \tensor \mathcal{O}_K,p}(s)$.

\begin{thm} \label{pro:higher.final}
Let $m, d \in N$.  For any number field $K$ of degree $d$ we have $\zeta^\wedge_{\mathcal{H}_m \tensor \mathcal{O}_K}(s) = \prod_{\p \in V_K} W_{\mathcal{H}_m, d}(q_\p, q_\p^{-s})$, where
\begin{equation*}
W_{\mathcal{H}_m,d}(X,Y) = \frac{\sum_{\sigma \in S_m} X^{- \ell(\sigma)} \prod_{j \in \mathrm{Des}(\sigma)} Z_j}{\prod_{j = 0}^m (1 - Z_j)} \in \Q(X,Y)
\end{equation*}
and $Z_j = X^{\binom{m+1}{2} - \binom{j+1}{2} + 2md} Y^{m+1}$ for $j \in [m]_0$.  The following functional equation holds:
\begin{equation*}
W_{\mathcal{H}_m,d}(X^{-1},Y^{-1}) = (-1)^{m+1} X^{m^2 + 4md}Y^{2(m+1)} W_{\mathcal{H}_m,d}(X,Y).
\end{equation*}
The abscissa of convergence is $\alpha^\wedge_{\mathcal{H}_m \tensor_\Z \mathcal{O}_K} = \frac{m}{2} + \frac{2md + 1}{m+1}$.
\end{thm}
\begin{proof}
Let $F/ \Q_p$ be a finite extension with residue field of cardinality $q$, and set $X_j = q^{\binom{m+1}{2} - \binom{j+1}{2} - ms}$ for any $j \in [m]_0$.
Recall the quantities $\widetilde{X}_j$ defined in~\eqref{equ:integral.bn} for $j \in [m]_0$.  Since $\widetilde{X}_0 = X_0$ and $\widetilde{X}_j = X_j^2$ for $j \in [m]$, it follows from~\eqref{equ:integral.bn} and Lemma~\ref{lem:bm.to.sm} that
\begin{equation} \label{equ:new.form}
 \int_{\mathrm{GSp}_{2m}^+(F)} | \det A |_{F}^{s} d \mu(A) =  \frac{ \sum_{w \in S_m} q^{- \ell (w)} \prod_{j \in \mathrm{Des}(w)} X_j }{\prod_{j = 0}^m (1 - X_j )} .
 \end{equation}
Indeed, the equality of the right-hand sides of~\eqref{equ:integral.bn} and~\eqref{equ:new.form} is  verified by substituting $(X,Y) = (q, q^{-ms})$ into the statement of Lemma~\ref{lem:bm.to.sm} and performing elementary calculations.
Substituting $\frac{m+1}{m} s - 2d$ for $s$ in~\eqref{equ:new.form}, we obtain the first part of our claim by Proposition~\ref{thm:gsp}.  Let $\sigma_0 \in S_m$ be the longest word, and recall that $\ell(\sigma_0 \sigma) = \binom{m}{2} - \ell(\sigma)$ and that $\mathrm{Des}(\sigma_0 \sigma) = [m-1] \setminus \mathrm{Des}(\sigma)$ for all $\sigma \in S_m$.  Hence, 
\begin{multline*}
W_{\mathcal{H}_m,d}(X^{-1}, Y^{-1}) = \frac{\sum_{\sigma \in S_m} X^{\ell(\sigma)} \prod_{j \in \mathrm{Des}(\sigma)} Z_j^{-1}}{\prod_{j = 0}^m (1 - Z_j^{-1})} = \\ (-1)^{m+1} Z_0 Z_m \frac{\sum_{\sigma} X^{\binom{m}{2} - \ell(\sigma_0 \sigma)} \prod_{j \in \mathrm{Des}(\sigma_0 \sigma)} Z_j}{\prod_{j = 0}^m (1 - Z_j)} = \!
(-1)^{m+1} X^{m^2 + 4md} Y^{2(m+1)} W_{\mathcal{H}_m,d}(X,Y).
\end{multline*}
The functional equation follows immediately.  Finally, by Lemma~\ref{lem:abscissa} below we have
$$  \alpha^\wedge_{\mathcal{H}_m \tensor_\Z \mathcal{O}_K} = \max_{j \in [m]_0} \left\{ \frac{\binom{m+1}{2} - \binom{j + 1}{2} + 2md + 1}{m+1} \right\} = \frac{\binom{m+1}{2} + 2md + 1}{m+1}.\qedhere$$ 
\end{proof}

The functions $\zeta^\wedge_{\mathcal{H}_m \tensor \mathcal{O}_K,p}(s)$, in the case $d = 1$ and $m \in \{2,3 \}$, were computed explicitly by du Sautoy and Lubotzky~\cite[\S3.3 and Example~1.4(2)]{duSLubotzky/96}.  Otherwise, Theorem~\ref{pro:higher.final} is new even for $d = 1$.  For comparison we mention that, for arbitrary $m$ and $K$, the ideal zeta factors $\zeta^\vartriangleleft_{\mathcal{H}_m \tensor \mathcal{O}_K,p}(s)$ are known when $p$ is unramified in $K$ by work of Carnevale, the third author, and Voll~\cite[Section~5.4]{CSV/19}.

\begin{rem} \label{rmk:higher.heisenberg}
Observe that the functional equation of Theorem~\ref{pro:higher.final} could have been derived directly from~\eqref{equ:integral.bn} using a similar multiplication by the longest element of $B_m$.  The proof of the functional equation is essentially that of~\cite[Theorem~4]{Voll/05}; see~\cite[p.~707]{Igusa/89} and the proofs of~\cite[Theorem~5.9]{duSLubotzky/96} and~\cite[Corollary~2.7]{Berman/11} for analogous arguments.  It is easily verified that $W_{\mathcal{H}_m,d}(X,Y) = \frac{1}{1 - Z_0} I_m(X^{-1}; Z_1, \dots, Z_m)$,
where $I_m$ is the ``Igusa function'' of~\cite[Definition~2.5]{SV1/15}; the functional equation for these functions is stated as~\cite[Proposition~4.2]{SV1/15} and follows directly from~\cite[Theorem~4]{Voll/05}.  Igusa functions and their generalizations also play central roles in explicit computations of ideal zeta functions of Lie rings~\cite{CSV/19, SV2/16}.
\end{rem}

\begin{rem} \label{rem:heisenberg}
Consider the Heisenberg Lie lattice $\mathcal{H} = \langle x, y, z \rangle_\Z$ with $[x,y] = z$ and $Z(\mathcal{H}) = \langle z \rangle$; this is the smallest non-abelian nilpotent Lie lattice.  Since $\mathcal{H} = \mathcal{H}_1 = \mathcal{F}_{2,2}$, for any number field $K$ we obtain 
$$\zeta^\wedge_{\mathcal{H} \tensor \mathcal{O}_K}(s) = \zeta_K (2s - 2d) \zeta_K (2s - 2d - 1)$$
as a special case of either Theorem~\ref{pro:zeta.free} or Theorem~\ref{pro:higher.final}.  It is clear from Theorem~\ref{pro:higher.final} and from Theorems~\ref{thm:lmn} and~\ref{thm:no.funct.eq} below that $\zeta^\wedge_{\mathcal{L} \tensor \mathcal{O}_K}(s)$, for a general Lie lattice $\mathcal{L}$, cannot be expressed so neatly in terms of Dedekind zeta functions.
\end{rem}

The computation of the abscissa of convergence in Theorem~\ref{pro:higher.final} relies on the following claim.  It is well-known and follows from the proof of~\cite[Corollary~3.1]{Voll/19}.
\begin{lem} \label{lem:abscissa}
Let $m \in \N$.  Let $W(X,Y) \in \Q(X,Y)$ be a rational function of the form
$$ W(X,Y) = \frac{\sum_{\sigma \in S_m} X^{- \ell(\sigma)} \prod_{j \in \mathrm{Des}(\sigma)} X^{a_j}Y^{b_j}}{\prod_{j = 1}^{m-1} (1 - X^{a_j} Y^{b_j})},$$
where $a_j \in \N \cup \{ 0 \}$ and $b_j \in \N$ for all $j \in [m-1]$, while $\ell(\sigma)$ and $\mathrm{Des}(\sigma)$ are as in Remark~\ref{rmk:permutation.descent}.  Let $K$ be a number field. Then the abscissa of convergence of the function
$ F(s) = \prod_{\p \in V_K} W(q_{\p}, q_{\p}^{-s})$
is $\max_{j \in [m-1]} \left\{ \frac{a_j + 1}{b_j} \right\}$.
\end{lem}

\subsection{The Lie lattices $\mathcal{L}_{m,n}$} \label{sec:lmn}
We recall a family of Lie lattices introduced by the first author, Klopsch, and Onn~\cite[Definition~2.1]{BKO/Dstar}.  Let $m,n \in \N$, with $n \geq 2$, and consider the sets
\begin{eqnarray*}
\mathbf{E} & = & \{ \mathbf{e} = (e_1, \dots, e_n) \in \N_0^n : e_1 + \cdots + e_n = m - 1 \} \\
\mathbf{F} & = & \{ \mathbf{f} = (f_1, \dots, f_n) \in \N_0^n : f_1 + \cdots + f_n = m \}.
\end{eqnarray*}
Let $(\mathbf{b}_1, \dots, \mathbf{b}_n)$ be the standard generators of the additive monoid $\N_0^n$, so that $\mathbf{b}_i = (0, \dots, 0, 1, 0, \dots, 0)$, with $1$ in the $i$-th position. 
Then $\mathcal{L}_{m,n}$ is the Lie lattice with $\Z$-basis 
$$ \{ x_{\mathbf{e}} : \mathbf{e} \in \mathbf{E} \} \cup \{ y_{\mathbf{f}} : \mathbf{f} \in \mathbf{F} \} \cup \{ z_j : j \in [n] \},$$
and with the Lie bracket defined by the relations
$$ [x_{\mathbf{e}}, y_{\mathbf{f}}] = \begin{cases}
z_i &: \mathbf{f} - \mathbf{e} = \mathbf{b}_i \\
0 &: \mathbf{f} - \mathbf{e} \not\in \{ \mathbf{b}_1, \dots, \mathbf{b}_n \}
\end{cases}
$$
and where all other pairs of elements of the $\Z$-basis above commute.  Then $\mathcal{L}_{m,n}$ is nilpotent of class two, and $Z(\mathcal{L}_{m,n}) = [\mathcal{L}_{m,n}, \mathcal{L}_{m,n}] = \langle z_1, \dots, z_n \rangle$.  

\begin{rem} \label{rem:lmn.families}
As noted in~\cite{BKO/Dstar}, the Lie lattices $\mathcal{L}_{m,n}$ provide a common generalization of two well-known families of Lie lattices:
\begin{itemize}
\item The Lie lattices $\mathcal{L}_{1,n}$, for $n \geq 2$, are the Grenham Lie lattices 
$$G_n = \langle x_0, x_1, \dots, x_n, z_1, \dots, z_n \rangle,$$
with relations $[x_0, x_i] = z_i$ for $i \in [n]$; all other pairs of generators commute.  Their pro-isomorphic zeta functions were computed by the first author~\cite[\S3.3.2]{Berman/05}.
\item The Lie lattices $\mathcal{L}_{m,2}$ are associated via the correspondence of~\eqref{equ:class.two} to the $D^\ast$-groups of odd Hirsch length defined in~\cite[\S1.1]{BKO/Dstar}.  These represent commensurability classes introduced by Grunewald and Segal~\cite[\S6]{GSegal/84} in the course of their classification of torsion-free nilpotent radicable groups of class two with finite rank and center of rank two.
\end{itemize}
\end{rem}

\begin{lemma} \label{lem:bko.abs.indec}
Let $k$ be a field.  The $k$-Lie algebra $\mathcal{L}_{m,n} \tensor_{\Z} k$ is indecomposable.
\end{lemma}
\begin{proof}
Write $L_k$ for $\mathcal{L}_{m,n} \tensor_{\Z} k$.  By slight abuse of notation, we denote the elements of the natural $k$-basis of $L_k$ by $x_{\mathbf{e}}, y_{\mathbf{f}}, z_i$.  Let
$$ v = \sum_{\mathbf{e} \in \mathbf{E}} a_{\mathbf{e}} x_{\mathbf{e}} + \sum_{\mathbf{f} \in \mathbf{F}} c_{\mathbf{f}} y_{\mathbf{f}} + \sum_{i = 1}^n d_i z_i \in L_k,$$
where the coefficients lie in $k$.
Suppose that $a_{\mathbf{e}} \neq 0$ for some $\mathbf{e} \in \mathbf{E}$.  We claim that $\dim_k [L_k, v] = n$.  Indeed, consider the lexicographical total ordering on $\mathbf{E}$, for which $\mathbf{e} \leq \mathbf{e}^\prime$ if there exists some $i \in [n]$ such that $e_i \leq e^\prime_i$ and $e_j = e^\prime_j$ for all $j < i$.  Let $\widehat{\mathbf{e}}$ be maximal, with respect to this ordering, among all $\mathbf{e} \in \mathbf{E}$ such that $a_{\mathbf{e}} \neq 0$.  It is easy to see that $[v, y_{\widehat{\mathbf{e}} + \mathbf{b}_1}] = a_{\widehat{\mathbf{e}}} z_1$, since $a_{\widehat{\mathbf{e}} + \mathbf{b}_1 - \mathbf{b}_j} = 0$ for all $j > 1$ by the maximality of $\widehat{\mathbf{e}}$.

Consider the map $\xi : \mathbf{E} \to \mathbf{E}$ given by $\xi(\mathbf{e}) = (e_2, e_3, \dots, e_n, e_1)$.  Similarly to the above, we see that for all $j \in [n]$, if $\widehat{\mathbf{e}}_j$ is such that $\xi^{j-1}(\widehat{\mathbf{e}}_j)$ is maximal in the set $\{ \xi^{j-1}(\mathbf{e}) : a_{\mathbf{e}} \neq 0 \}$, then $[v, y_{\widehat{\mathbf{e}}_j + \mathbf{b}_j}] = a_{\widehat{\mathbf{e}}_j} z_j$.  Thus $\dim_k [v, L_k] = \dim_k Z(L_k) = n$.

Suppose that $L_k = L_1 \oplus L_2$ is a direct sum of non-trivial $k$-Lie subalgebras.  If $[L_2, L_2] = [L_k, L_k] = Z(L_k)$, then $L_1$ must be abelian.  But then $L_1 \subseteq Z(L_k) \subseteq L_2$, which is impossible since we assumed $L_1 \neq 0$.  Hence $\dim_k [L_2, L_2] < n$, and similarly $\dim_k [L_1, L_1] < n$.
If $w \in L_1$, then $\dim_k [w, L_k] = \dim_k [w, L_1] \leq \dim_k [L_1, L_1] < n$.  Similarly, $\dim_k [w, L_k] < n$ for all $w \in L_2$.  Thus $L_k$ is spanned by the set $\{ v \in L_k : \dim_k [v, L_k] < n \}$.  However, we just showed that this set is contained in the proper subspace spanned by $\{ y_{\mathbf{f}} : \mathbf{f} \in \mathbf{F} \} \cup \{ z_i : i \in [n] \}$, giving rise to a contradiction.
\end{proof}

As in the case of the higher Heisenberg Lie algebras considered in Section~\ref{sec:higher.heisenberg}, there is no ideal $M \leq \gamma_2 L_k$ such that $\mathcal{X}(M)$ generates $L_k$.  Thus the $Z(L_k)$-rigidity of $L_k$ cannot be shown using Corollary~\ref{cor:segal}.  Write $L_{m,n} = \mathcal{L}_{m,n} \tensor_\Z \Q$.

\begin{pro} \label{pro:lmn.rigid}
Let $k$ be a field of characteristic zero.  Then the $k$-Lie algebra $L_{m,n} \tensor_\Q k$ is $Z(L_{m,n} \tensor_\Q K)$-rigid.
\end{pro}
\begin{proof}
By virtue of Remark~\ref{rem:extensions}, it suffices to verify the hypotheses of Theorem~\ref{thm:rigidity} for $L_{m,n}$.  Since $Z(L_{m,n})$ is central, the absolute indecomposability of $L_{m,n} / [Z(L_{m,n}), L_{m,n}]$ is given by Lemma~\ref{lem:bko.abs.indec}.  Thus it suffices to show that $\mathcal{Y}(Z(L_{m,n}))$ generates $L_{m,n}$.  We first check that $x_{\mathbf{e}} \in \mathcal{Y}(Z(L_{m,n}))$ for all $\mathbf{e} \in \mathbf{E}$.  Indeed, let $v \in C_{L_{m,n}}(C_{L_{m,n}}(x_\mathbf{e}))$.  Since the linear span of $\{ x_{\mathbf{e}^\prime} : \mathbf{e}^\prime \in \mathbf{E} \}$ is contained in $C_{L_{m,n}}(x_\mathbf{e})$, it is clear that
$ v \equiv \sum_{\mathbf{e}^\prime \in \mathbf{E}} a_{\mathbf{e}^\prime} x_{\mathbf{e}^\prime} \, \mathrm{mod} \, Z(L_{m,n})$.  Suppose that $v \not\in \Q x_{\mathbf{e}} + Z(L_{m,n})$.  Then $a_{\mathbf{e}^\prime} \neq 0$ for some $\mathbf{e}^\prime \neq \mathbf{e}$.  Recall the order on $\mathbf{E}$ and the map $\xi : \mathbf{E} \to \mathbf{E}$ defined in the proof of Lemma~\ref{lem:bko.abs.indec}.  There is some $i \in [n]$ such that $e_i < e_i^\prime$.  Let $\widehat{\mathbf{e}} \in \mathbf{E}$ be such that $\xi^{i-1}(\widehat{\mathbf{e}}) = \max \{ \xi^{i-1}(\mathbf{e}^\prime) : a_{\mathbf{e}^\prime} \neq 0 \}$.  Then $\widehat{\mathbf{e}} \neq \mathbf{e}$ and $v$ does not commute with $y_{\widehat{\mathbf{e}} + \mathbf{b}_i} \in C_{L_{m,n}}(x_{\mathbf{e}})$, contradicting the assumption $v \in C_{L_{m,n}}(C_{L_{m,n}}(x_\mathbf{e}))$.  It follows that $x_{\mathbf{e}} \in \mathcal{Y}(Z(L_{m,n}))$ as claimed.

However, it is not true that $y_{\mathbf{f}} \in \mathcal{Y}(Z(L_{m,n}))$ for all $\mathbf{f} \in \mathbf{F}$.  Indeed, 
$C_{L_{m,n}}(y_{(m-1,1,0,\dots,0)})$ is the $\Q$-linear span of 
$$\{ x_{\mathbf{e}} : \mathbf{e} \in \mathbf{E} \setminus \{ (m-2, 1, 0, \dots, 0), (m-1,0,\dots,0) \} \} \cup \{ y_{\mathbf{f}} : \mathbf{f} \in \mathbf{F} \} \cup \{z_i : i \in [n] \}.$$  
Therefore $y_{(m,0,\dots, 0)} \in C_{L_{m,n}} (C_{L_{m,n}} (y_{(m-1, 1, 0, \dots, 0)}))$, whence $y_{(m-1,1,0, \dots, 0)} \not\in \mathcal{Y}(Z(L_{m,n}))$.
Instead, fix an arbitrary vector $\mathbf{c} = (c_1, \dots, c_{n-1}) \in \Q^{n-1}$ and set
$$ v_{\mathbf{c}} = \sum_{\mathbf{f} \in \mathbf{F}} c_1^{f_1} c_2^{f_2} \cdots c_{n-1}^{f_{n-1}} y_{\mathbf{f}}.$$
First we show that $v_{\mathbf{c}} \in \mathcal{Y}(Z(L_{m,n}))$.  Obviously, the linear span of $\{ y_{\mathbf{f}} : \mathbf{f} \in \mathbf{F} \} \cup Z(L_{m,n})$ centralizes $v_{\mathbf{c}}$.  Consequently, $C_{L_{m,n}}(C_{L_{m,n}}(v_{\mathbf{c}}))$ is contained in the $\Q$-linear span of $\{ y_{\mathbf{f}} : \mathbf{f} \in \mathbf{F} \} \cup Z(L_{m,n})$.  Set $\boldsymbol{\delta}_j = \mathbf{b}_j - \mathbf{b}_n$ for $j \in [n-1]$.  Then 
$$ [x_{\mathbf{e}}, v_\mathbf{c}] = c_1^{e_1} \cdots c_{n-1}^{e_{n-1}} (c_1 z_1 + \cdots + c_{n-1} z_{n-1} + z_n)$$
for all $\mathbf{e} \in \mathbf{E}$,
from which it is clear that $x_{\mathbf{e} + \boldsymbol{\delta}_j} - c_j x_{\mathbf{e}} \in C_{L_{m,n}}(v_\mathbf{c})$ for all $j \in [n-1]$ and all $\mathbf{e} \in \mathbf{E}$ such that $e_n > 0$.  
Hence if 
$$ v = \sum_{\mathbf{f} \in \mathbf{F}} a_{\mathbf{f}} y_{\mathbf{f}} \in C_{L_{m,n}}(C_{L_{m,n}}(v_{\mathbf{c}})),$$
then $a_{\mathbf{f} + \boldsymbol{\delta}_j} = c_j a_{\mathbf{f}}$ for all $j \in [n-1]$ and all $\mathbf{f} \in \mathbf{F}$ such that $f_n > 0$.  Since $(f_1, \dots, f_n) = (0,\dots, 0,m) + \sum_1^{j-1} f_j \boldsymbol{\delta}_j$ and the coefficient of $y_{(0, \dots, 0,m)}$ in $v_{\mathbf{c}}$ is $1$, it is easy to see that $v$ is necessarily a $\Q$-scalar multiple of $v_{\mathbf{c}}$.  It follows that $v_{\mathbf{c}} \in \mathcal{Y}(Z(L_{m,n}))$ for all $\mathbf{c} \in \Q^{n-1}$.  Thus it remains only to show that the $\Q$-linear span of $\{ y_{\mathbf{f}} : \mathbf{f} \in \mathbf{F} \}$ is spanned by elements of the form $v_{\mathbf{c}}$.

Fix a natural number $N > m$.  If $\mathbf{f} = (f_1, \dots, f_n) \in \mathbf{F}$, set 
$$\sigma(\mathbf{f}) = f_1 + f_2 N + \cdots + f_{n-1} N^{n-2}.$$  
Since $\mathbf{f}$ is determined by its first $n - 1$ coordinates, the $\sigma(\mathbf{f})$ are all distinct.  Note that $\sigma((0, \dots, 0,m)) = 0$.  Given $\lambda \in \Q$, set $\mathbf{c}(\lambda) = (\lambda, \lambda^N, \lambda^{N^2}, \dots, \lambda^{N^{n-2}}) \in \Q^{n-1}$, so that $v_{\mathbf{c}(\lambda)} = \sum_{\mathbf{f} \in \mathbf{F}} \lambda^{\sigma(\mathbf{f})} y_{\mathbf{f}}$.  Let $\lambda_1, \dots, \lambda_{| \mathbf{F} |}$ be distinct elements of $\Q$.  We order the elements of $\mathbf{F} = \{ \mathbf{f}_1, \dots, \mathbf{f}_{| \mathbf{F} |} \}$ so that the sequence $\sigma(\mathbf{f}_i)$ is decreasing.  The $| \mathbf{F} | \times | \mathbf{F} |$ matrix whose rows are $v_{\mathbf{c}(\lambda_1)}, \dots, v_{\mathbf{c}(\lambda_{| \mathbf{F} |})}$, with respect to the basis $( y_{\mathbf{f}_1}, \dots, y_{\mathbf{f}_{| \mathbf{F} |}} )$, has the form
$$
\left(  \begin{array}{llcr}
\lambda_1^{\sigma(\mathbf{f}_1)} & \lambda_1^{\sigma(\mathbf{f}_2)} & \cdots & \lambda_1^{\sigma(\mathbf{f}_{| \mathbf{F} |})} \\
\lambda_2^{\sigma(\mathbf{f}_1)} & \lambda_2^{\sigma(\mathbf{f}_2)} & \cdots & \lambda_2^{\sigma(\mathbf{f}_{| \mathbf{F} |})} \\
\ddots & \ddots && \ddots \\
\lambda_{| \mathbf{F} |}^{\sigma(\mathbf{f}_1)} & \lambda_{| \mathbf{F} |}^{\sigma(\mathbf{f}_2)} & \cdots & \lambda_{| \mathbf{F} |}^{\sigma(\mathbf{f}_{| \mathbf{F} |})} \end{array} \right).
$$
This is a generalized Vandermonde matrix whose determinant is, by definition, 
$$ S_\nu(\lambda_1, \dots, \lambda_{| \mathbf{F} |}) \prod_{1 \leq i < j \leq | \mathbf{F} |} (\lambda_i - \lambda_j),$$
where $S_\nu$ is the Schur polynomial associated to the partition $\nu = (\nu_1, \dots, \nu_{| \mathbf{F} | })$ whose parts are the non-negative integers $\{ \sigma(\mathbf{f}_i) - | \mathbf{F}| + i : i \in [ | \mathbf{F} | ] \}$.  Since $S_\nu$ is well-known to be a non-zero symmetric polynomial with integer coefficients (see, for instance,~\cite[Section~I.3]{Macdonald/95}),
we may choose the distinct $\lambda_i$ so that $S(\lambda_1, \dots, \lambda_{| \mathbf{F} |}) \neq 0$, in which case the above matrix is invertible.  For any $\mathbf{f} \in \mathbf{F}$ it follows that $y_{\mathbf{f}}$ lies in the $\Q$-linear span of the $v_{\mathbf{c}(\lambda_i)}$, and we conclude that $L_{m,n}$ is indeed generated by $\mathcal{Y}(Z(L_{m,n}))$.
\end{proof}

The remaining hypotheses of Corollary~\ref{cor:rigid.lifting} are verified in~\cite{BKO/Dstar}.  Moreover, the functions $\theta_1^F$ and $\theta_2^F$, for arbitrary finite extensions $F/\Q_p$, are computed in $(4.8)$ and $(4.9)$ of~\cite{BKO/Dstar}.  As in that paper, set $r_1 = \binom{m+n-2}{m-1}$ and $r_2 = \binom{m+n-1}{m}$.  Since the computation of the integral resulting from Corollary~\ref{cor:rigid.lifting} is completely analogous to the case $d = 1$, which relies on Proposition~\ref{pro:bruhat.decomposition} and is performed at the end of~\cite[\S4]{BKO/Dstar}, we omit it and only record the final result.

\begin{thm} \label{thm:lmn}
Let $m,n \in \mathbb{N}$, with $n \geq 2$.  Let $d \in \N$.  If $K$ is a number field of degree $d$, then
$
\zeta^\wedge_{\mathcal{L}_{m,n} \tensor \mathcal{O}_K}(s) = \prod_{\p \in V_K} W_{\mathcal{L}_{m,n},d}(q_\p,  q_\p^{-s})$,
where 
$$ W_{\mathcal{L}_{m,n},d}(X,Y) = \frac{\sum_{w \in S_n} X^{- \ell(w)} \prod_{i \in \mathrm{Des}(w)} X_i}{\prod_{i = 0}^{n} (1 - X_i)}.$$
The monomials $X_i = X^{\beta_i}  Y^{\gamma_i}$ are given by
\allowdisplaybreaks{
\begin{eqnarray*}
\beta_i & = & 
\begin{cases} 
\begin{aligned}[b] & i(n-i) + d(r_1 + r_2)((m-1)n + i) + \\ &\sum_{j = 1}^i \left( 1 + \frac{(m-1)(i - j + 1)}{n-j+1} \right) \binom{m+j-2}{m-1} \binom{m+n-j-1}{m-1} \end{aligned} &: i \in [n-1] \\
dn (r_1 + r_2) &: i = 0 \\
dn (r_1 + r_2)  + \binom{2m + n - 2}{2m - 1} &: i = n
\end{cases} \\
\gamma_i & = &
\begin{cases}
(1 + r_1)((m-1)n + i) - m(m-1)r_1 &: i \in [n-1] \\
r_1 + n &: i = 0 \\
r_2 + n &: i = n.
\end{cases}
\end{eqnarray*} }

\end{thm}
\begin{cor} \label{cor:lmn}
The following functional equation holds:
\begin{equation*}
W_{\mathcal{L}_{m,n},d}(X^{-1}, Y^{-1}) = (-1)^{n+1} X^{\binom{n}{2} + \binom{2m + n - 2}{2m - 1} + 2dn(r_1 + r_2)}Y^{r_1 + r_2 + 2n} W_{\mathcal{L}_{m,n},d}(X,Y).
\end{equation*}
The abscissa of convergence of $\zeta^\wedge_{\mathcal{L}_{m,n} \tensor \mathcal{O}_K}(s)$ is $\alpha^\wedge_{\mathcal{L}_{m,n} \tensor_\Z \mathcal{O}_K} = \max_{i \in [n]_0} \left\{ \frac{\beta_i + 1}{\gamma_i} \right\}$.
\end{cor}
\begin{proof}
The functional equation is found as in the proof of Theorem~\ref{pro:higher.final}.  The claim regarding the abscissa of convergence follows from Theorem~\ref{thm:lmn} and Lemma~\ref{lem:abscissa}.  
\end{proof}

Determining in general which fraction from the set of Corollary~\ref{cor:lmn} is maximal is laborious already when $K = \Q$; see~\cite[\S5]{BKO/Dstar}.  The set $\{ \alpha^\wedge_{\mathcal{L}_{m,n}} : m \geq 1, n \geq 2 \}$ has infinitely many accumulation points~\cite[Corollary~1.8]{BKO/Dstar}, and it would be interesting to study the structure of the sets $\{ \alpha^\wedge_{\mathcal{L}_{m,n} \tensor \mathcal{O}_K} : m \geq 1, n \geq 2, [K:\Q] < \infty \}$.  By contrast, the set of abscissae of convergence of subgroup zeta functions of $\mathcal{T}$-groups has no accumulation points~\cite[Proposition~1.1]{duSG/06}.

The ideal zeta functions $\zeta^\vartriangleleft_{\mathcal{L}_{m,n},p}(s)$ were determined by Voll~\cite[Theorem~1.1]{Voll/19} for arbitrary $(m,n)$ and all primes $p$, when $K = \Q$.  More is known for the Grenham Lie lattices $\mathcal{L}_{1,n}$: for any number field $K$ and any unramified prime $p$, the functions $\zeta^\vartriangleleft_{\mathcal{L}_{1,n} \tensor \mathcal{O}_K, p}(s)$ were computed by Carnevale, the third author, and Voll~\cite[Proposition~5.8]{CSV/19}.  Snocken~\cite[Theorem~5.11 and Example~6.2]{Snocken/12} determined the representation zeta functions of the $\mathcal{T}$-groups of class two associated to the two families of Remark~\ref{rem:lmn.families}.

\subsection{A family of filiform Lie lattices} \label{sec:max.class}
Let $c \geq 2$ and let $\mathcal{M}_c$ be the Lie lattice over $\Z$ with the following presentation:
\begin{equation} \label{equ:max.class.present}
\mathcal{M}_c = \langle z, x_1, \dots, x_c | [z, x_i] = x_{i+1}, i \in [c-1] \rangle.
\end{equation}
Here, as always, we follow the convention that all pairs of generators not explicitly mentioned commute.  Then $\mathcal{M}_c$ is a nilpotent Lie lattice of class $c$, which is the maximal possible  class of a Lie lattice of rank $c + 1$.  Observe that $\mathcal{M}_2$ is the Heisenberg Lie lattice; henceforth, let $c \geq 3$.  Let $M_c = \mathcal{M}_c \tensor_\Z \Q$ denote the associated $\Q$-Lie algebra.  

We claim that $M_{c} \tensor_\Q k$ is $Z(M_{c} \tensor_\Q k)$-rigid for any field $k$ of characteristic zero.  Indeed, it is easily verified that $z$ and $z + x_1$ each belong to $\mathcal{X}(Z(M_{c}))$, and hence $\mathcal{X}(Z(M_{c}))$ generates $M_{c}$ as a Lie algebra.  Then rigidity follows by Segal's criterion (Corollary~\ref{cor:segal}; note also Remarks~\ref{rem:segal} and~\ref{rem:extensions}).
Consider the decomposition $M_{c} \tensor k = U_1 \oplus \cdots \oplus U_{c+1}$, where $U_1 = \langle z \rangle_k$ and $U_i = \langle x_{i-1} \rangle_k$ for all $i \in [2,c+1]$.  The subspace $V_i = U_i \oplus \cdots \oplus U_{c+1}$, for any $i \in [c+1]$, is preserved by any $k$-automorphism of $M_{c} \tensor k$; indeed, $V_2$ is the unique maximal abelian subalgebra (since $c \geq 3$), whereas $V_i = \gamma_{i-1} (M_{c} \tensor k)$ for all $i \in [3,c+1]$.  Let
$y = \lambda z + \sum_{i = 1}^c a_i x_i \in M_{c} \tensor k$ and $y^\prime = \mu x_1 + \sum_{i = 2}^c b_i x_i \in V_2$ be arbitrary, where $\lambda, \mu \in k^\times$ and $a_i, b_i \in k$.  The matrix
\begin{equation} \label{equ:max.class.matrix}
\left(
\begin{array}{cccccc}
\lambda & a_1 & a_2 & a_3 & \cdots & a_c \\
& \mu & b_2 & b_3 & \cdots & b_{c} \\
& & \lambda \mu & \lambda b_2  & \cdots & \lambda b_{c-1} \\
&&& \lambda^2 \mu & \cdots & \lambda^2 b_{c-2} \\
&&&& \ddots & \vdots \\
&&&&& \lambda^{c-1} \mu
\end{array} \right),
\end{equation}
with respect to the basis $(z, x_1, \dots, x_c)$, corresponds to $\varphi \in \Aut_k (M_{c} \tensor k)$ such that $(z)\varphi = y$ and $(x_1)\varphi = y^\prime$; it is the unique automorphism with this property since $z$ and $x_1$ generate $M_{c} \tensor k$.  We have thus determined the structure of the algebraic automorphism group $\aAut M_{c}$, and it is clear that the three assumptions of Section~\ref{sec:assumptions} are satisfied.

Recall the notation of Section~\ref{sec:consequences}; 
for instance, $\mathbf{H}$ is the reductive subgroup of $\aAut M_c$ corresponding to diagonal matrices in~\eqref{equ:max.class.matrix}.
If $F/\Q_p$ is a finite extension and $h = \mathrm{diag}(\lambda, \mu, \lambda \mu, \dots, \lambda^{c-1} \mu) \in \mathbf{H}^+(F)$, then 
\begin{equation*}
\theta_j^F (h) = 
\begin{cases}
| \mu |_F^{-1} &: j = 1 \\ 
| \lambda^{j-1} \mu |^{-2}_F &: 2 \leq j \leq c.
\end{cases}
\end{equation*}
Indeed, $\theta^F_j (h)$ measures the size of the set of elements $a_1 \in F$, if $j = 1$, or of pairs $(a_j, b_j) \in F^2$ if $j > 1$, such that $\lambda^{j-1} \mu a_j \in \mathcal{O}_F$ and $\lambda^{j-1} \mu b_j \in \mathcal{O}_F$.  Note that $\det h = \lambda^{\binom{c}{2} + 1} \mu^c$.
Moreover, Assumption~\ref{lifting.condition} is realized by a polynomial map in the sense of Assumption~\ref{polynomial.lifting}: the class $\overline{g} \in NH^+ / N_i \cap (G/N_i)^+$, in the notation of Section~\ref{sec:consequences}, determines $\lambda, \mu, a_1, \dots, a_{i-1}, b_2, \dots, b_{i-1} \in \mathcal{O}_F$ in~\eqref{equ:max.class.matrix}, and one may take the lifting $g \in G^+$ to correspond to the matrix with $a_j = b_j = 0$ for all $j \in [i,c]$.  

The remaining hypotheses of Corollary~\ref{cor:rigid.lifting} obviously hold, enabling us to conclude, for any prime $p$ and any number field $K$ of degree $d = [K : \Q]$, that
\begin{align*}
\zeta^\wedge_{\mathcal{M}_c \tensor \mathcal{O}_K, p} (s) & = 
& \prod_{\p | p} \int_{\mathrm{GL}^+_1(K_\p)^2} | \lambda |_{\p}^{\left( \binom{c}{2} + 1 \right) s - (c-1)(2d + c - 2)} |\mu|_{\p}^{cs - (2d + 2c - 3)} d \mu_{\mathrm{GL}_1(K_\p)^2} (\lambda, \mu).
\end{align*}
Using Example~\ref{exm:abelian} to evaluate the integral, we arrive at the following statement.
\begin{thm}
Let $c \geq 2$ and $d \geq 1$.  For any number field $K$ of degree $d$, we have $\zeta^\wedge_{\mathcal{M}_c \tensor \mathcal{O}_K}(s) = \prod_{\p \in V_K} W_{\mathcal{M}_c, d}(q_\p, q_\p^{-s})$, where
\begin{equation*}
W_{\mathcal{M}_c, d}(X,Y) = \frac{1}{( 1 - X^{(c-1)(2d + c - 2)} Y^{\binom{c}{2} + 1})(1 - X^{2d + 2c - 3}Y^c)}.
\end{equation*}
This rational function satisfies the functional equation
\begin{equation*}
W_{\mathcal{M}_c, d}(X^{-1} Y^{-1}) = X^{c(2d + c - 2) - 1} Y^{\binom{c+1}{2} + 1} W_{\mathcal{M}_c, d}(X,Y).
\end{equation*}
The global pro-isomorphic zeta function 
\begin{equation*}
\zeta^\wedge_{\mathcal{M}_c \tensor \mathcal{O}_K} (s) = \zeta_K \left( \left( \binom{c}{2} + 1 \right) s - (c-1)(2d + c - 2) \right) \zeta_{K} \left( cs - (2d + 2c - 3) \right)
\end{equation*}
has abscissa of convergence 
$$\alpha^\wedge_{\mathcal{M}_c \tensor \mathcal{O}_K} = 
\begin{cases}
2 &: d = 1 \\
\frac{(c-1)(2d + c - 2) + 1}{\binom{c}{2} + 1} &: d \geq 2.
\end{cases} 
$$
\end{thm}
\begin{proof}
The case $c = 2$ follows from Remark~\ref{rem:heisenberg}, so let $c \geq 3$.
We only discuss the abscissa of convergence, as the rest of the statement is clear.  From the properties of the Dedekind zeta function,
$$ \alpha^\wedge_{\mathcal{M}_c \tensor \mathcal{O}_K} = \max \left\{ \frac{2(d + c - 1)}{c}, \frac{(c-1)(2d + c - 2) + 1}{\binom{c}{2} + 1} \right\}.$$
Computing the difference between these two fractions and observing that its numerator is a quadratic polynomial in $c$, an elementary analysis allows us to determine its sign.  Observe that $\frac{2(d + c - 1)}{c} = 2$ if $d = 1$.
\end{proof}

In the case $K = \Q$, this result was obtained in the first author's thesis~\cite[Section 3.3.1]{Berman/05}.  The functions $\zeta^\vartriangleleft_{\mathcal{M}_c}(s)$ and $\zeta^\leq_{\mathcal{M}_c}(s)$ are known for $c \in \{ 3, 4 \}$ by work of Taylor and Woodward, as well as the local factors of $\zeta^\vartriangleleft_{\mathcal{M}_3 \tensor \mathcal{O}_K}(s)$ for quadratic number fields $K$ at split primes $p$; see Theorems~2.26, 2.29, and 2.37 of~\cite{duSWoodward/08}.  We are not aware of explicit computations of ideal or subring zeta functions for any nilpotent Lie lattice of class greater than $4$.  However, functional equations for almost all local factors of $\zeta^\vartriangleleft_{\mathcal{M}_c}(s)$, for arbitrary $c$, were proved by Voll~\cite[Theorem~4.8]{Voll/17}; the analogous statement for any $\zeta^\leq_{\mathcal{M}_c \tensor \mathcal{O}_K}(s)$ is a special case of~\cite[Corollary~1.1]{Voll/10}.  If $\mathcal{M}_c$ is taken to be the nilpotent {\emph{group}} of class $c$ with presentation~\eqref{equ:max.class.present} (as before, with all other commutators trivial), then Ezzat~\cite{Ezzat/nexc, Ezzat/exc} computed the representation zeta functions $\zeta^{\mathrm{irr}}_{\mathcal{M}_3}(s)$ and $\zeta^{\mathrm{irr}}_{\mathcal{M}_4}(s)$, as well as almost all local factors of $\zeta^{\mathrm{irr}}_{\mathcal{M}_c}(s)$ for $c \geq 5$.

\subsection{The Lie lattice $\mathcal{F}_4$} \label{sec:filiform}
Consider the Lie lattice $\mathcal{F}_4$ of class $4$ with the following presentation:
$$ \mathcal{F}_4 = \left\langle z, x_1, x_2, x_3, x_4 \mid [x_1, x_2] = x_4, [z, x_i] = x_{i + 1} \, \text{for all} \, i \in [3]  \right\rangle_\Z,$$
with the convention that all other pairs of generators commute.  Let $\mathrm{F}_4 = \mathcal{F}_4 \tensor_\Z \Q$, and let $k$ be any field of characteristic zero.  Observe that the underlying $\Z$-module of $\mathcal{F}_4$ is the same as that of $\mathcal{M}_4$, and consider the decomposition $\mathrm{F}_4 \tensor_\Q k = U_1 \oplus U_2 \oplus U_3 \oplus U_4 \oplus U_5$ that was defined in the previous section for $M_{4} \tensor_\Q k$.  For every $i \in [5]$, the subspace $V_i = U_i \oplus \cdots \oplus U_5$ is preserved by any $\varphi \in \mathrm{Aut}_k (\mathrm{F}_4 \tensor_\Q k)$; indeed, $V_2 = \{ v \in \mathrm{F}_4 \tensor_\Q k : \dim_k C_{\mathrm{F}_4 \tensor k}(v) \geq 3 \}$, whereas $V_i = \gamma_{i-1} (\mathrm{F_4} \tensor_\Q k)$ for $i \in \{ 3,4,5 \}$.  Thus $\varphi$ corresponds to an upper triangular matrix with respect to the basis $(z, x_1, x_2, x_3, x_4)$, with the same first two rows and diagonal elements as in~\eqref{equ:max.class.matrix}.  The relation $[x_1, x_2] = x_4$ implies that $\lambda^3 \mu x_4 =  (x_4)\varphi = [(x_1)\varphi, (x_2)\varphi] = \lambda \mu^2 x_4$.  Since $\lambda, \mu \in k^\times$, we obtain $\mu = \lambda^2$.  One checks that $\mathrm{Aut}_k (\mathrm{F}_4 \tensor_\Q K)$ is exactly the following:
\begin{equation} \label{equ:filiform.aut}
\left\{
\left( \begin{array}{ccccl}
\lambda & a_1 & a_2 & a_3 & a_4 \\
0 & \lambda^2 & b_2 & b_3 & b_4 \\
0 & 0 & \lambda^3 & \lambda b_2 & \lambda b_3 + a_1 b_2 - \mu a_2 \\
0 & 0 & 0 & \lambda^4 & \lambda^2 b_2 + a_1 \lambda^3 \\
0 & 0 & 0 & 0 & \lambda^5
\end{array} \right) \, \middle| \,
\begin{array}{rrl}
\lambda & \in & k^\times \\ a_1, a_2, a_3, a_4 & \in & k \\ b_2, b_3, b_4 & \in & k \end{array} \right\} .
\end{equation}
It is readily verified that $z$ and $z + x_1$ are contained in $\mathcal{X}(Z(\mathrm{F}_4 \tensor_\Q k))$ and hence that $\mathrm{F}_4 \tensor_\Q k$ is $Z(\mathrm{F}_4 \tensor_\Q k)$-rigid by Corollary~\ref{cor:segal} and Remark~\ref{rem:segal}.  For any prime $p$, the decomposition of $\mathrm{F}_4$ considered above satisfies the hypotheses of Corollary~\ref{cor:rigid.lifting}.  For any finite extension $F / \Q_p$ with residue field of cardinality $q$, we read off from~\eqref{equ:filiform.aut} that the elements $h \in \mathbf{H}^+(F)$ are those of the form $h = \mathrm{diag}(\lambda, \lambda^2, \lambda^3, \lambda^4, \lambda^5)$ for some $\lambda \in \mathrm{GL}_1^+(F) = \mathcal{O}_F \setminus \{ 0 \}$.  By essentially the same computation as for $\mathcal{M}_4$, keeping in mind that $\mu = \lambda^2$, we obtain that 
\begin{equation*}
\theta_j^F (h) = \begin{cases}
| \lambda |_F^{-2} &: j = 1 \\
| \lambda |_F^{-2(j+1)} &: 2 \leq j \leq 4.
\end{cases}
\end{equation*}
Since $\det h = \lambda^{15}$, it follows from Corollary~\ref{cor:rigid.lifting} that for any number field $K$ of degree $d$ and any prime $p$, we have
\begin{equation*}
\zeta^\wedge_{\mathcal{F}_4 \tensor \mathcal{O}_K, p}(s) = \prod_{\p | p} \int_{\mathrm{GL}_1^+(K_\p)} | \lambda |^{15s - 16 - 10d} d \mu_{K_\p^\times} (\lambda) = \prod_{\p | p} (1 - q_\p^{16 + 10d - 15s})^{-1}.
\end{equation*}

We have thus established the following.
\begin{thm} \label{thm:filiform}
Let $\mathcal{F}_4$ be the filiform Lie lattice defined above, and let $d \in \N$.  If $K$ is a number field of degree $d$, then
\begin{equation*}
\zeta^\wedge_{\mathcal{F}_4 \tensor \mathcal{O}_K}(s) = \zeta_K(15s - 16 - 10d) = \prod_{\p \in V_K} W_{\mathcal{F}_4, d}(q_p, q_p^{-s}),
\end{equation*}
where $W_{\mathcal{F}_4, d}(X,Y) = \frac{1}{1 - X^{16 + 10d}Y^{15}}$ satisfies the functional equation $W_{\mathcal{F}_4, d}(X^{-1}, Y^{-1}) = - X^{16 + 10d}Y^{15} W_{\mathcal{F}_4, d}(X,Y)$ and the abscissa of convergence of $\zeta^\wedge_{\mathcal{F}_4 \tensor \mathcal{O}_K}(s)$ is $\frac{17 + 10d}{15}$.
\end{thm}

The Lie lattice $\mathcal{F}_4$ is not naturally graded.  However, it has a grading given by: $(\mathcal{F}_4)_1 = \langle x_2 \rangle$, $(\mathcal{F}_4)_2 = \langle x_1 \rangle$, $(\mathcal{F}_4)_3 = \langle x_4 \rangle$, $(\mathcal{F}_4)_4 = \langle x_3 \rangle$, $(\mathcal{F}_4)_5 = \langle x_5 \rangle$.  It can be shown that this grading is minimal.  Hence $\mathrm{wt}(\mathcal{F}_4) = 15$, and Conjecture~\ref{graded.conj} holds for all base extensions $\mathcal{F}_4 \tensor \mathcal{O}_K$ by the discussion in Section~\ref{sec:intro.funct.eq}.

Theorem~\ref{thm:filiform} generalizes a computation of the first author~\cite[\S3.3.11]{Berman/05} for the case $K = \Q$.  The ideal zeta function $\zeta^\vartriangleleft_{\mathcal{F}_4}(s)$ was computed by Woodward~\cite[Theorem~2.39]{duSWoodward/08}.  Its local factors do not satisfy functional equations.  By~\cite[Theorem~1.2]{Voll/17} this implies that $\mathcal{F}_4$ does not have a grading affording the homogeneity condition of~\cite[Condition~1.1]{Voll/17}.  

\subsection{The Lie lattice $\mathcal{Q}_5$} \label{sec:Q5}
We now consider the Lie lattice 
\begin{equation} \label{equ:Q5.present}
\mathcal{Q}_5 = \langle x_1, x_2, x_3, x_4, x_5 | [x_1,x_2] = x_4, [x_1, x_4] = [x_2, x_3] = x_5 \rangle,
\end{equation}
whose associated Lie algebra $Q_5 = \mathcal{Q}_5 \tensor_\Z \Q$ is so denoted in~\cite[\S3.3.12]{Berman/05}, up to a relabeling of the generators; this Lie lattice is called $\mathfrak{g}_{5,3}$ in~\cite[\S2.14]{duSWoodward/08}.  The novel feature of this example is that, while $Q_5 \tensor k$ is $Z(Q_5 \tensor k)$-rigid for any field $k$ of characteristic zero, it does not satisfy the rigidity criterion of Theorem~\ref{thm:rigidity}.  Indeed, set $U_i = \langle x_i \rangle$ for $i \in [5]$, and let $V_i = U_i \oplus \cdots \oplus U_5$ as usual.  It is easy to verify that, if $y = \sum_{i = 1}^5 x_i \tensor a_i \in Q_5 \tensor k \setminus (V_3  \tensor k)$, then $C_{Q_5 \tensor k}(y) = C_{Q_5 \tensor k}(y - x_3 \tensor ba_1 + x_4 \tensor ba_2)$ for any $b \in k$.  Therefore, $y \not\in \mathcal{Y}(Z(Q_5) \tensor k)$, so the hypotheses of Theorem~\ref{thm:rigidity} fail.  Instead, we will compute $\Aut_k(Q_5 \tensor K)$ directly for any finite extension $K/k$.  First we require an elementary lemma.

\begin{lem} \label{lem:lin.funct}
Let $k$ be a field of characteristic zero and $K/k$ a finite extension.  Let $f, g : K \to K$ be $k$-linear functions.  Suppose, for all $\kappa \in K$, that $\alpha g(\beta) - \beta f(\alpha)$ is constant for all pairs $\alpha, \beta \in K$ such that $\alpha \beta = \kappa$.  Then $f$ and $g$ are $K$-linear.
\end{lem}
\begin{proof}
By assumption, we have
$ \alpha \gamma g(\gamma^{-1} \beta) - \gamma^{-1} \beta f(\alpha \gamma) = \alpha g(\beta) - \beta f(\alpha)$ for all $\alpha, \beta \in K$ and all $\gamma \in K^\times$.  Equivalently,
$\alpha( \gamma g(\gamma^{-1} \beta) - g(\beta)) = \beta ( \gamma^{-1} f(\alpha \gamma) - f(\alpha))$ for all such $\alpha, \beta, \gamma$.  Hence there is a constant $C \in K$ such that $\gamma g(\gamma^{-1} \alpha) - g(\alpha) = \gamma f(\gamma^{-1} \alpha) - f(\alpha) = C \alpha$ for all $\alpha \in K$ and $\gamma \in K^\times$.

Consider the $k$-algebra embedding $\iota: K \hookrightarrow \mathrm{End}_k(K)$, where $\iota(\alpha)$ is the map $\beta \mapsto \alpha \beta$ for every $\alpha \in K$.  By the previous paragraph, $\iota(\gamma) \circ f \circ \iota(\gamma^{-1}) - f = C$ for all $\gamma \in K^\times$.  Taking $\gamma = 1$, we see that $C = 0$.  Thus $f \in \mathrm{End}_k(K)$ centralizes $\iota(K)$.  Since $\iota(K)$ is self-centralizing, by the Double Centralizer Theorem applied to the central simple $k$-algebra $\mathrm{End}_k(K)$, the map $f$ is multiplication by some fixed element of $K$.  The same holds for $g$.
\end{proof}

\begin{pro}
Let $k$ be a field of characteristic zero.  Then the $k$-Lie algebra $Q_5 \tensor_\Q k$ is $Z(Q_5 \tensor_\Q k)$-rigid.
\end{pro}
\begin{proof}
Let $K/k$ be a finite extension, and let $\varphi \in \Aut_k(Q_5 \tensor K)$.  Write $Q_{5,K} = Q_5 \tensor K$ and $V_{i,K} = V_i \tensor K$ for $i \in [5]$.  Observe that $\varphi$ stabilizes the subspace $V_{i,K}$ for every $i \in [5]$.  Indeed, $V_{3,K} = \{ y \in Q_{5,K} : \dim_K C_{Q_{5,K}}(y) \geq 4 \}$, whereas $V_{4,K} = \gamma_2 Q_{5,K}$ and $V_{5,K} = Z(Q_{5,K}) = \gamma_3 Q_{5,K}$.  Finally, $V_{2,K} = C_{Q_{5,K}}(V_{4,K})$.  Thus for every $i \in [5]$ and every $\alpha \in K$ we may write $\varphi(x_i \tensor \alpha) = \sum_{j = i}^5 x_j \tensor f_{ij}(\alpha)$, where the functions $f_{ij} : K \to K$ are $k$-linear.  Moreover, $f_{ii}$ is bijective for every $i \in [5]$.  Define $g_{11} = f_{11} / f_{11}(1)$ and $g_{22} = g_{22}/g_{22}(1)$.

Observe that $(x_4 \tensor \alpha \beta)\varphi = [(x_1 \tensor \alpha)\varphi, (x_2 \tensor \beta)\varphi] \equiv x_4 \tensor f_{11}(\alpha) f_{22}(\beta) \, \mathrm{mod} \, V_{5,K}$ for all $\alpha, \beta \in K$.  In particular, $f_{11}(\alpha \beta) f_{22}(1) = f_{11}(\alpha) f_{22}(\beta) = f_{11}(1) f_{22}(\alpha \beta)$.  Dividing by $f_{11}(1)f_{22}(1)$, we conclude that $g_{11} = g_{22}$ is a multiplicative function.  It is easy to see that there is a $K$-linear automorphism $\psi \in (\aAut Q_5)(K)$ such that $(x_i \tensor 1)\psi = x_i \tensor f_{ii}(1)$ for $i \in \{1, 2 \}$.  Composing $\varphi$ with $\psi$ and with a suitable element of $\mathrm{Gal}(K/k)$, we may assume without loss of generality that $f_{11}(\alpha) = f_{22}(\alpha) = \alpha$ for all $\alpha \in K$.

For all $\alpha, \beta \in K$ we have $0 = [(x_1 \tensor \alpha)\varphi, (x_1 \tensor \beta)\varphi] \equiv x_4 \tensor (\alpha f_{12}(\beta) - \beta f_{12}(\alpha)) \, \mathrm{mod} \, V_{5,K}$.  Hence $f_{12}$ is $K$-linear by Lemma~\ref{lem:lin.funct}.  Similarly, considering $[(x_1 \tensor \alpha)\varphi, (x_1 \tensor \beta)\varphi]$ we obtain that $f_{23}$ is $K$-linear.  Now observe that
$$ (x_4 \tensor \alpha \beta)\varphi = [(x_1 \tensor \alpha)\varphi,(x_2 \tensor \beta)\varphi] = x_4 \tensor \alpha\beta + x_5 \tensor (\alpha f_{24}(\beta) + f_{12}(\alpha)f_{23}(\beta) - \beta f_{13}(\alpha)).$$
Since we already know that $f_{12}$ and $f_{23}$ are $K$-linear, it follows that $\alpha f_{24}(\beta) - \beta f_{13}(\alpha)$ depends only on the product $\alpha\beta$.  Hence $f_{13}$ and $f_{24}$ are $K$-linear by Lemma~\ref{lem:lin.funct}.  Moreover, $f_{44}$ is the identity function.  By analogous considerations, the functions $f_{14}$, $f_{33}$, and $f_{34}$ are $K$-linear.  Thus the automorphism of $Q_{5,K}/Z(Q_{5,K})$ induced by $\varphi$ is $K$-linear, proving our claim.
\end{proof}

The algebraic automorphism group $\aAut Q_5$ is determined in~\cite[\S3.3.12]{Berman/05}; alternatively, this is left as an exercise for the reader along with the details of the computation below.  The decomposition of $Q_5$ defined above satisfies the hypotheses of Corollary~\ref{cor:rigid.lifting} at all primes $p$.  For any finite extension $F/\Q_p$ with residue field of cardinality $q$, we find that $\mathbf{H}(F)$ consists of the diagonal matrices of the form 
$h = \mathrm{diag}(a,b,a^2,ab,a^2b)$ for $a,b \in F^\times$.  We determine that
$$ (\theta_1^F(h), \theta_2^F(h), \theta_3^F(h), \theta_4^F(h)) = (|b|^{-1}_F, |a|^{-4}_F, |a^2 b^2|^{-1}_F, |a^6 b^3|^{-1}_F).$$
By Corollary~\ref{cor:rigid.lifting}, at all primes $p$ the local pro-isomorphic zeta function is given by
$$ \zeta^\wedge_{\mathcal{Q}_5 \tensor \mathcal{O}_K, p}(s) =  \prod_{\p | p} \int_{K_\p^\times \times K_\p^\times} |a|^{6s - (6 + 6d)}_{\p} |b|^{3s - (3 + 3d)}_{\p} d\mu_{K_\p^\times \times K_\p^\times}(a,b).$$

\begin{thm}
Let $d \in \N$.  If $K$ is a number field of degree $d$, then $\zeta^\wedge_{\mathcal{Q}_5 \tensor \mathcal{O}_K}(s) =  \zeta_K(6s - (6+6d)) \zeta_K(3s - (3+3d)) = \prod_{\p \in V_K} W_{\mathcal{Q}_5, d}(q_{\p}, q_{\p}^{-s})$, where
$$ W_{\mathcal{Q}_5, d}(X,Y) = \frac{1}{(1 - X^{6 + 6d} Y^6)(1 - X^{3 + 3d} Y^3)}.$$
The functional equation $W_{\mathcal{Q}_5, d}(X^{-1}, Y^{-1}) = X^{9+9d}Y^9 W_{\mathcal{Q}_5,d}(X,Y)$ holds, and the abscissa of convergence is $\alpha^\wedge_{\mathcal{Q}_5 \tensor \mathcal{O}_K} = d + \frac{4}{3}$.
\end{thm}

In the case $d = 1$, we recover results of the first author~\cite[\S3.3.12]{Berman/05}.  It is not hard to show that the grading $(\mathcal{Q}_5)_1 = \langle x_1, x_2 \rangle$, $(\mathcal{Q}_5)_2 = \langle x_3, x_4 \rangle$, $(\mathcal{Q}_5)_3 = \langle x_5 \rangle$ is minimal.  Hence $\mathrm{wt}(\mathcal{Q}_5) = 9$, and Conjecture~\ref{graded.conj} holds for all base extensions $\mathcal{Q}_5 \tensor \mathcal{O}_K$.  The ideal and subring zeta functions $\zeta^\vartriangleleft_{\mathcal{Q}_5}(s)$ and $\zeta^\leq_{\mathcal{Q}_5}(s)$ were computed by Woodward~\cite[Theorem~2.41]{duSWoodward/08}.  The representation zeta function of the nilpotent group with presentation~\eqref{equ:Q5.present} may be found in~\cite[Table~5.1]{Ezzat/thesis}.

\subsection{A family of Lie lattices lacking functional equations} \label{sec:bk}
The first author and Klopsch~\cite{BK/15} have constructed a nilpotent Lie lattice $\mathcal{L}$, which they denote as $\Lambda$, none of whose local pro-isomorphic zeta functions satisfy functional equations in the sense defined in the introduction to the present paper.  We show that this property is retained by the base extensions $\mathcal{L} \tensor_\Z \mathcal{O}_K$, for any number field $K$, by computing their pro-isomorphic zeta functions explicitly.

As in Section~\ref{sec:free}, let $\mathcal{F}_{4,3}$ be the free nilpotent Lie lattice of class four on three generators $X,Y,Z$.  For brevity we use the left-normed simple product notation for the Lie bracket, so that, for instance, we write $XYZ$ for $[[X,Y],Z]$.  Let $\mathcal{I} \leq \mathcal{F}_{4,3}$ be the ideal generated by $YXXX - YZY$ and $ZXXX - ZYZ$, and set $\mathcal{L} = \mathcal{F}_{4,3} / \mathcal{I}$.  Write $x,y,z$ for the projections to $\mathcal{L}$ of $X,Y,Z$, respectively.  In~\cite[(3.4)]{BK/15} it is determined that $\mathcal{L}$ is a free $\Z$-module of rank $25$ and that a basis is given by 
\begin{eqnarray*}
(b_1, \dots, b_{25}) & = & (x,y,z,xy,xz,yz,xyy,xzz,xyz,xzy,xyx,xzx,xyyy,xzzz,xyxx,xzxx, \\ && xyxy,xzxz,xyxz,xzxy,xyzx,xyzz,xzyy,xyzy,xzyz).
\end{eqnarray*}  

As usual, write $L = \mathcal{L} \tensor_\Z \Q$ and $L_p = \mathcal{L} \tensor_\Z \Q_p$ for any prime $p$; it is easy to verify that $Z(L_p) = \gamma_4 L_p$.  By explicit computations involving repeated application of the identities $(3.2)$ and $(3.3)$ of~\cite{BK/15}, we find that $x,y,z \in \mathcal{X}(Z(L_p))$.  Hence $L_p$ is $Z(L_p)$-rigid by Corollary~\ref{cor:segal} and Remark~\ref{rem:segal}.
The algebraic automorphism group $\aAut L$ was determined in~\cite[Theorem 4.2]{BK/15}.  It follows from the description of $\aAut L$ given there that the decomposition $L_p = U_1 \oplus U_2 \oplus U_3 \oplus U_4$, where 
$$ \begin{array}{cccc}
U_1 = \langle b_1, b_2, b_3 \rangle_{\Q_p}, & U_2 = \langle b_4, b_5, b_6 \rangle_{\Q_p}, & U_3 = \langle b_{7}, \dots, b_{12} \rangle_{\Q_p}, & U_4 = \langle b_{13}, \dots, b_{25} \rangle_{\Q_p},
\end{array}
$$
satisfies Assumptions~\ref{first.assumption} and~\ref{second.assumption}; note that $V_i = \gamma_i L_p$ for all $i \in [4]$.  The argument on~\cite[p.~505]{BK/15} shows that Assumption~\ref{polynomial.lifting} holds.  Thus Corollary~\ref{cor:rigid.lifting} is applicable.

As in~\cite[(5.1)]{BK/15}, we see that $\mathbf{H}(F)$, for any field $F / \Q_p$ with residue field of cardinality $q$, consists of the diagonal matrices of the form
\begin{align} \label{equ:bk.example}
h = \mathrm{diag} & (a,b,c,ab,ac,bc,ab^2, ac^2, abc,abc,a^2 b, a^2 c, ab^3, ac^3, a^3 b, a^3 c, \\ \nonumber
& a^2 b^2, a^2 c^2, a^2 bc, a^2 bc, a^2 bc, abc^2, a b^2 c, ab^2 c, abc^2),
\end{align}
where $a,b,c \in F^\times$ satisfy $a^3 = bc$.  
Analogously to the computations of~\cite[Section 5]{BK/15}, where the case $F = \Q_p$ is treated, we determine that, for $h \in \mathbf{H}^+(F)$ as in~\eqref{equ:bk.example}, we have
\begin{eqnarray*}
\theta_1^F (h) & = &  | a^3 b^4 c^4 |_{F}^{-1} \min \{ |b|_{F}^{-1} , |c|^{-1}_{F} \} \\
\theta_2^F (h) & = &  | a^{24} b^{15} c^{15} |^{-1}_{F} \\
\theta_3^F (h) & = &  | a^{66} b^{45} c^{45} |^{-1}_{F}.
\end{eqnarray*}
It is clear from~\eqref{equ:bk.example} that $|\det h |_{F} = | a^{33} b^{23} c^{23} |_{F}$.  Thus, for any number field $K$ of degree $d$ and any prime $p$, by Corollary~\ref{cor:rigid.lifting} we find that $\zeta^\wedge_{\mathcal{L} \tensor \mathcal{O}_K,p}(s)$ is given by:
$$
\prod_{\p | p} \int_{\mathbf{H}^+(K_\p)} |a|_{\p}^{33s - (27 + 66d)} |b|_{\p}^{23s - (19 + 45d)} |c|_{\p}^{23s - (19 + 45d)} \min \{ |b|_{\p}^{-1}, |c|_{\p}^{-1} \} d \mu_{\mathbf{H}(K_\p)}(h).
$$

Since the computation of this integral is very similar to the one performed in~\cite{BK/15}, we omit the details and proceed to state the final result.

\begin{thm} \label{thm:no.funct.eq}
Let $K$ be a number field of degree $d = [K : \Q]$.  Then $\zeta^\wedge_{\mathcal{L} \tensor \mathcal{O}_K}(s) = \prod_{\p \in V_K} W_{\mathcal{L},d}(q_\p, q_\p^{-s})$, where
\begin{equation*}
W_{\mathcal{L},d}(X,Y) = \frac{1 + X^{84 + 201d}Y^{102} + 2 X^{85 + 201d} Y^{102} + 2 X^{170 + 402d} Y^{204}}{(1 - X^{84 + 201d} Y^{102})(1 - X^{171 + 402d}Y^{204})}.
\end{equation*}
\end{thm}
Neither the functions $W_{\mathcal{L},d}$ nor the local pro-isomorphic zeta functions $\zeta^\wedge_{\mathcal{L} \tensor \mathcal{O}_K, p}(s)$ admit functional equations.  Indeed, while the denominator of the rational function $\prod_{i = 1}^r W_{\mathcal{L},d}(X^{f_i}, Y^{f_i})$ as in the proof of Corollary~\ref{cor:finitely.uniform} clearly admits a functional equation, the numerator has constant term $1$ and leading coefficient $2^r$ and therefore does not satisfy a functional equation.  
One can show that $\mathcal{L}$ is a graded Lie lattice with $\mathrm{wt}(\mathcal{L}) = 102$.  The ``reduced pro-isomorphic zeta function'' $W_{\mathcal{L},d}(1,Y)$ as in~\cite[(1)]{Evseev/09} still fails to satisfy a functional equation.  However, observe that
$$ \frac{W_{\mathcal{L},d}(1,Y^{-1})}{W_{\mathcal{L},d}(1,Y)} \sim \frac{Y^{102}}{2} =  \frac{Y^{\mathrm{wt}(\mathcal{L})}}{2}$$
as $Y \to \infty$.  It is an interesting challenge to formulate and prove a result extending Conjecture~\ref{graded.conj} to such cases.

\bibliographystyle{amsplain}

\end{document}